\documentclass[11pt,reqno]{amsart}

\usepackage{amsfonts}
\usepackage{amssymb}
\usepackage{graphicx}
\usepackage{amsmath,mathtools}
\usepackage{latexsym}
\usepackage{xypic}
\usepackage{mathrsfs}
\usepackage{enumitem} 
\usepackage{braket}
\usepackage{amsthm}
\usepackage{tikz}
\usepackage{comment}

\usepackage{marginnote}

\usepackage[pdftex]{hyperref}

\setlist{itemsep=3pt}

\numberwithin{equation}{section}

\allowdisplaybreaks

\newtheorem{prop}{Proposition}
\newtheorem{theo}[prop]{Theorem}
\newtheorem{lemm}[prop]{Lemma}
\newtheorem{coro}[prop]{Corollary}

\newtheorem{claim}[prop]{Claim}

\theoremstyle{definition}
\newtheorem{defi}[prop]{Definition}

\newtheorem{rema}[prop]{Remark}

\numberwithin{prop}{section}

\newcommand{\CC}{\mathbb{C}}

\newcommand{\HH}{\mathbb{H}}

\newcommand{\NN}{\mathbb{N}}

\newcommand{\RR}{\mathbb{R}}
\renewcommand{\SS}{\mathbb{S}}

\newcommand{\ZZ}{\mathbb{Z}}

\newcommand{\bC}{\mathbf{C}}

\newcommand{\bF}{\mathbf{F}}

\newcommand{\bL}{\mathbf{L}}
\newcommand{\bM}{\mathbf{M}}

\newcommand{\bv}{\mathbf{v}}

\newcommand{\cA}{\mathcal A}
\newcommand{\cB}{\mathcal B}
\newcommand{\cC}{\mathcal C}

\newcommand{\cF}{\mathcal F}
\newcommand{\cG}{\mathcal G}
\renewcommand{\cH}{\mathcal H}
\newcommand{\cI}{\mathcal I}

\renewcommand{\cL}{\mathcal L}
\newcommand{\cM}{\mathcal M}
\newcommand{\cN}{\mathcal N}

\newcommand{\cP}{\mathcal P}

\renewcommand{\cR}{\mathcal R}
\newcommand{\cS}{\mathcal S}
\newcommand{\cT}{\mathcal T}
\newcommand{\cU}{\mathcal U}
\newcommand{\cV}{\mathcal V}

\newcommand{\cZ}{\mathcal Z}

\def\fB{\mathfrak{B}}

\DeclareMathOperator{\tr}{tr}

\DeclareMathOperator{\im}{im}

\DeclareMathOperator{\supp}{supp}

\newcommand{\bangle}[1]{\left\langle #1 \right\rangle}

\DeclareMathOperator{\area}{area}

\DeclareMathOperator{\Vol}{Vol}

\let\oldmarginpar\marginpar
\renewcommand\marginpar[1]{\-\oldmarginpar[\raggedleft\footnotesize #1]%
{\raggedright\footnotesize #1}}

\DeclareMathOperator{\re}{Re}
\DeclareMathOperator{\imag}{Im}
\DeclareMathOperator{\Id}{Id}

\DeclareMathOperator{\length}{length}
\DeclareMathOperator{\Div}{div}

\DeclareMathOperator{\vol}{vol}

\DeclareMathOperator{\genus}{genus}

\DeclareMathOperator{\proj}{proj}
\DeclareMathOperator{\dist}{dist}
\DeclareMathOperator{\Index}{index}
\DeclareMathOperator{\sing}{sing}
\DeclareMathOperator{\inj}{inj}

\newcommand{\eps}{\varepsilon}
\DeclareMathOperator{\image}{image}

\DeclareMathOperator{\met}{Met}
\DeclareMathOperator{\reg}{reg}
\DeclareMathOperator{\imm}{Imm}
\DeclareMathOperator{\emb}{Emb}
\DeclareMathOperator{\Tan}{Tan}
\DeclareMathOperator{\VarTan}{VarTan}
\newcommand{\su}{\mathfrak{su}}
\newcommand{\bI}{\mathbf{I}}
\DeclareMathOperator{\dmn}{dmn}
\DeclareMathOperator{\Lim}{Lim}
 
\newcommand{\restr}{\mathbin{\vrule height 1.6ex depth 0pt width
0.13ex\vrule height 0.13ex depth 0pt width 1.3ex}}
\DeclareMathOperator{\sech}{sech}

\usepackage{graphicx}

\allowdisplaybreaks

\title{The p-widths of a surface}
\author{Otis Chodosh} 
\address{Department of Mathematics, Stanford University, Building 380, Stanford, CA 94305, USA}
\email{ochodosh@stanford.edu}

\author{Christos Mantoulidis} 
\address{Department of Mathematics, Rice University, Houston, TX 77005, USA}
\email{christos.mantoulidis@rice.edu}

\begin{document}

\maketitle

\begin{abstract}
The $p$-widths of a closed Riemannian manifold are a nonlinear analogue of the spectrum of its Laplace--Beltrami operator, which corresponds to areas of a certain min-max sequence of possibly singular minimal submanifolds. We show that the $p$-widths of any closed Riemannian two-manifold correspond to a union of closed immersed geodesics, rather than simply geodesic nets. 

We then prove optimality of the sweepouts of the round two-sphere constructed from the zero set of homogeneous polynomials, showing that the $p$-widths of the round sphere are attained by $\lfloor \sqrt{p}\rfloor$ great circles. As a result, we find the universal constant in the Liokumovich--Marques--Neves--Weyl law for surfaces to be $\sqrt{\pi}$. 

En route to calculating the $p$-widths of the round two-sphere, we prove two additional new results: a bumpy metrics theorem for stationary geodesic nets with fixed edge lengths, and that, generically, stationary geodesic nets with bounded mass and bounded singular set have Lusternik--Schnirelmann category zero.
\end{abstract}

\tableofcontents

\section{Introduction}

\subsection{Our setting and results} 
Fix a closed (i.e., compact and without boundary) Riemannian manifold $(M^{n+1},g)$. The \emph{$p$-widths} of $(M,g)$, denoted $\omega_p(M,g) \in (0,\infty)$ for $p \in \NN^*$, are a geometric nonlinear analogue of the spectrum of its Laplace--Beltrami operator. They are obtained by replacing the Rayleigh quotient of the Laplace--Beltrami operator along families of scalar-valued functions on $M$ with the $n$-dimensional area along sweepouts of $M$ of (possibly singular) hypersurfaces. See Section \ref{subsec:p-widths} for the definition. They were introduced by Gromov \cite{Gromov:spectra.width,Gromov:waist,Gromov:sing.exp.top}, studied further by Guth \cite{Guth:minimax}, and have  played a central and exciting role in minimal surface theory when combined with the Almgren--Pitts--Marques--Neves Morse theory program for the area functional \cite{Almgren:htpy, Pitts, MarquesNeves:posRic,IrieMarquesNeves,MarquesNevesSong,Song:full-yau,song:dichotomy,Li:infinitely-high-dim,GasparGuaraco,GasparGuaraco:weyl,HaslhoferKetover,CM:3d,Zhou:multiplicity-one,SongZhou:scarring}.  We invite the reader to refer to \cite{Gromov:spectra.width} for the analogy between the Laplace spectrum and the volume spectrum, and to \cite{MarquesNeves:uper-semi-index} for a thorough overview of the importance of this analogy in minimal surface theory.

Let us recall the main existence theorem for $p$-widths. By the combined work of Almgren--Pitts, Schoen--Simon, Marques--Neves, and Li, it is known that in ambient dimensions $n+1 \geq 3$ every $p$-width is attained as the area of a smoothly embedded minimal hypersurface $\Sigma_p$ whose singular set $\bar \Sigma_p \setminus \Sigma_p$ has dimension $\leq n-7$, whose connected components may contribute to area with different multiplicities, and whose total Morse index (discounting multiplicities) is bounded by $p$. That is:

\begin{theo}[{\cite{Pitts, SchoenSimon, MarquesNeves:multiplicity,Li:morse}}] \label{theo:ap.existence}
	Let $(M^{n+1}, g)$ be a closed Riemannian manifold with $n+1 \geq 3$. For every $p \in \NN^*$, there exists a smoothly embedded minimal hypersurface $\Sigma_p \subset M$, with $\bar \Sigma_p \setminus \Sigma_p$ of Hausdorff dimension $\leq n-7$ and components $\Sigma_{p,1}, \ldots, \Sigma_{p,N(p)} \subset \Sigma_p$, such that
	\[ \omega_p(M, g) = \sum_{j=1}^{N(p)} m_j \cdot \area_g(\Sigma_{p,j}), \]
	where $m_j \in \NN^*$ for all $j \in \{1, \ldots, N(p)\}$ and $\Index(\Sigma_{p})\leq p$.
\end{theo}
See also \cite{SimonSmith,ColdingDeLellis,DeLellisPellandini,DeLellisTasnady,Zhou:posRic,Zhou:Ricci-sing,Ketover:genus,KMN:catenoid,KLM,RL,CLS,ZWang:deform,Li-ZWang,MarquesMontezumaNeves} for related work. 

Note that, when $3 \leq n+1 \leq 7$, $\Sigma_p$ is necessarily smoothly embedded. On the other hand, in the case of a two-dimensional Riemannian manifold ($n+1=2$), 
min-max methods not only need not produce \textit{embedded} geodesics (see \cite{Aiex:ellipsoids} for examples of immersed geodesics being produced), but in full generality they could \emph{a priori} produce \textit{geodesic nets} as opposed to (immersed) geodesics (see \cite[Remark 1.1]{MarquesNeves:multiplicity}). 

Our first main result shows that the min-max methods described above can be guaranteed to produce (immersed) geodesics regardless of the number of parameters. Throughout the paper, a geodesic is said to be primitive if it is connected and traversed with multiplicity one.

\begin{theo}\label{theo:immersed.geo.min.max}
Let $(M^{2},g)$ be a closed Riemannian manifold. For every $p \in \NN^{*}$, there exists a $\sigma_p \subset M$ consisting of distinct primitive closed geodesics $\sigma_{p,1}, \ldots, \sigma_{p,N(p)} \subset \sigma_p$  such that
\[
\omega_{p}(M,g) = \sum_{j=1}^{N(p)} m_j \cdot \length_g(\sigma_{p,j}),
\]
where $m_j \in \NN^*$ for all $j \in \{1, \ldots, N(p) \}$.
\end{theo}

Our approach uses a phase-transition regularization of the area-functional. In this direction, we note the following contributions: \cite{HutchinsonTonegawa00,Tonegawa05,Wickramasekera14,TonegawaWickramasekera12,Guaraco,Hiesmayr,Gaspar,WangWei,WangWei2,Dey:APAC,Bellettini:generic.exist,bellettini:mult1-AC}. There also exist other successful regularization approaches \cite{Riviere:visc,Riviere:lsc.index,PigatiRiviere:parametrized.varifolds,PigatiRiviere,MichelatRiviere,ChengZhou} (note that our technique precisely allows us to circumvent the fundamental issue discussed in \cite[p. 1984]{PigatiRiviere:parametrized.varifolds}.), but the phase transition technique is the only one known to allow for the study of $p$-widths across $p \in \NN^*$ via its relationship to Almgren--Pitts (\cite{Dey:APAC}).

The existence of immersed geodesics representing $p$-widths was previously known in the following cases: 
\begin{itemize}
	\item $p=1$ by Calabi--Cao \cite{CalabiCao} and, independently using phase transitions, by the second-named author\footnote{\cite{Mantoulidis} works with $p$-widths defined via phase transitions instead, but those agree with the ones above by Dey \cite{Dey:APAC}. See Propositions \ref{prop:AC-AP}, \ref{prop:phase.vs.ap.limit} below.} \cite{Mantoulidis} (see also \cite{ZhouZhu:networks,KetoverLiokumovich:c11.curves}).
	\item $p\in \{1,\dots,8\}$ for nearly round metrics on $\SS^2$  by  Aiex \cite{Aiex:ellipsoids}.
\end{itemize}

\begin{rema}
The min-max approach to finding closed geodesics date to Birkhoff's work in 1917 \cite{Birkhoff} in which he proved that the two-sphere with any Riemannian metric admits a closed geodesic (this question was posed by Poincar\'e \cite{Poincare}). See also \cite{CM:curve.shortening}. Lusternik--Schnirelmann have established the (sharp) result that any metric on the two-sphere admits at least three simple closed geodesics \cite{LS:three} (see also \cite{Grayson,Jost:LS,Klingenberg,L:top-calc-var-large,Tauimanov}). Similarly, Franks \cite{Franks} and Bangert \cite{Bangert} have proven that such a surface admits infinitely many principal (immersed) closed geodesics (see also \cite{Hingston}). 
\end{rema} 

Our second main result is a computation of the full $p$-width spectrum of the round two-sphere. 

\begin{theo}\label{theo:sphere.pwidths}
	Let $g_0$ denote the unit round metric on $\SS^2$. For every $p \in \NN^*$,
	\[ \omega_p(\SS^2,g_0) = 2\pi \lfloor \sqrt{p} \rfloor, \]
	and is attained by a sweepout constructed out of homogeneous polynomials. The corresponding $\sigma_p$ is a union of $\lfloor \sqrt{p}\rfloor$ great circles (repetitions allowed).
\end{theo}
As explained to us by Guth, the optimality of the sweepout constructed out of homogeneous polynomials fits into the theme of the ``efficiency of polynomials,'' which is loosely connected to the polynomial method in combinatorics; see \cite{Guth:unexpected, Guth:book}. Our result in Theorem \ref{theo:sphere.pwidths} is in line with Guth's conjecture for the open problem in \cite[Exercise 14.2]{Guth:book}.

We highlight the following prior results regarding some low-frequency $p$-widths of round two- and three-spheres: 
\begin{itemize}
\item Aiex \cite{Aiex:ellipsoids} proved that in the unit two-sphere, $\sigma_1, \sigma_2, \sigma_3$ can be taken to be great circles while $\sigma_4,\dots,\sigma_8$ can be taken to be the union of two great circles.
\item Nurser \cite{Nurser} proved that in the unit three-sphere\footnote{The study of minimal surfaces in $\SS^3$ is a nontrivial subject with a rich history; see the survey \cite{Brendle:survey}.} $\Sigma_1,\dots,\Sigma_4$ are\footnote{This work relied, in particular, on recent advances in the study of embedded minimal surfaces in $\SS^3$ by Marques--Neves \cite{MarquesNeves:Willmore} and Brendle \cite{Brendle:Lawson}.} totally geodesic spheres and $\Sigma_5,\dots,\Sigma_7$ are Clifford tori. Furthermore, he proved that $\Sigma_9$ is some embedded minimal surface in $\SS^3$ having $\genus(\Sigma_9)> 1$ and $\area(\Sigma_9) = \omega_9(\SS^3,g_{\SS^3}) \in (2\pi^2,8\pi)$. See also the related works \cite{CajuGasparGuaracoMatthisen,Hiesmayr:clifford} computing low parameter phase-transition widths of the round three-sphere.
\end{itemize}
However, to this point there had not been a single\footnote{Ignoring the trivial example of $\SS^1$.} $(M,g)$ for which the areas $\omega_p(M,g)$ (let alone the surfaces $\Sigma_p$) are known for all $p \in \NN^*$, not even in the two-dimensional case. (For comparison, the spectrum of the Laplacian is completely determined for a large class of Riemannian manifolds, cf.\ \cite{mathoverflow:laplace.spectrum}).

One application of Theorem \ref{theo:sphere.pwidths} concerns Weyl law for the $p$-widths. Recall that the Laplacian spectrum (denoted by $\lambda_p(M,g)$) of a closed Riemannian $(n+1)$-manifiold satisfies the celebrated Weyl law
\[
\lim_{p\to\infty} \lambda_p(M,g) p^{-\frac{2}{n+1}} = 4\pi^2 \vol(B)^{-\frac{2}{n+1}} \vol(M,g)^{-\frac{2}{n+1}} 
\]
showing that the high-frequency behavior of the spectrum is universal in a certain sense. Liokumovich--Marques--Neves have recently proven  \cite{LMN:Weyl} that the $p$-widths satisfy the following Weyl-type law 
\begin{equation}\label{eq:LMN.Weyl}
\lim_{p\to\infty}\omega_p(M,g)p^{-\frac{1}{n+1}} = a(n) \vol(M,g)^{\frac{n}{n+1}}
\end{equation}
for some constant $a(n)>0$. (See also \cite{GasparGuaraco:weyl}.) This result has had important implications for existence of minimal hypersurfaces, cf.\ \cite{IrieMarquesNeves,Li:infinitely-high-dim}. However, the constant $a(n)$ has not been determined for any dimension $n$ (see \cite[\S 1.5]{LMN:Weyl}). This is in contrast with the classical Weyl law, where one can use e.g. the (explicitly known) spectrum of a cube to compute the constant in a straightforward manner. Here, our full computation of the $p$-widths of the round two-sphere in Theorem \ref{theo:sphere.pwidths} readily implies:
\begin{coro}\label{coro:const.weyl}
When $n=1$, the constant in \eqref{eq:LMN.Weyl} satisfies $a(1) = \sqrt{\pi}$. 
\end{coro}
This settles the ``simplest case'' of the first question in \cite[\S 1.5]{LMN:Weyl}.

\begin{rema} \label{rema:guth.bound}
	It is interesting to compare this to Guth's estimates for the $p$-widths of the unit disk \cite[p.\ 1974]{Guth:minimax}. By \cite{LMN:Weyl} and Corollary \ref{coro:const.weyl}, it holds that
\[
 \omega_p(D^2) = \pi p^{\frac 12} + o(p^{\frac 12})
\]
as $p\to\infty$. As such, (as predicted in \cite{Guth:minimax}) the estimate given in \cite[p.\ 1974]{Guth:minimax} is not sharp for large $p$. On the other hand, the conjectural value corresponding to straight lines through the origin in \cite[p.\ 1974]{Guth:minimax} would be too small for sufficiently large $p$ (see also \cite[Exercise 14.2]{Guth:book}). To this end, it would be interesting to understand the analogue of Theorem \ref{theo:immersed.geo.min.max} for manifolds with boundary. 
\end{rema}

\subsection{Strategy of the proof of Theorem \ref{theo:immersed.geo.min.max}: phase transitions and the Liu--Wei tangent cone theorem} 

We study the singularities of a limiting object of the Almgren--Pitts apparatus (i.e., an element of the \textit{Almgren--Pitts critical set}, as we define in Section \ref{sec:AP.AC}), which \emph{a priori} is at best only a stationary geodesic network per \cite[Remark 1.1]{MarquesNeves:multiplicity}. Our result will follow if we can prove that at least one limiting object exists whose singular points have tangent cones (unique by \cite{AllardAlmgren:1varifold}) corresponding to lines in $\RR^2$ through the origin.

As usual with Almgren--Pitts theory, we do not  prove this for all possible limiting objects coming out of the Almgren--Pitts apparatus, but content ourselves with showing that at least one good limiting object exists. Unlike the standard approach in Almgren--Pitts theory, however, we do not show their existence by contradiction. We argue directly, by ``regularizing'' the length functional in a way that favors better behaved singularities in the limiting objects.

In our prior work \cite{CM:3d} we used a phase transition regularization first introduced to this min-max setting by Guaraco \cite{Guaraco}; we showed that for generic metrics the limiting objects occur with multiplicity one, in analogy with\footnote{At the time \cite{CM:3d} was written, it was not clear whether or not the phase transition limiting objects were Almgren--Pitts limiting objects. This has since been resolved by Dey \cite{Dey:APAC} (see Propositions \ref{prop:AC-AP} and \ref{prop:phase.vs.ap.limit}). This was not important for \cite{CM:3d} but one of its consequences (Proposition \ref{prop:AC-AP}) is essential for our current paper.} the Almgren--Pitts multiplicity one conjecture.\footnote{It is worthwhile to note that Zhou \cite{Zhou:multiplicity-one}'s subsequent result on the Almgren--Pitts multiplicity-one conjecture also relied on a regularization process, albeit one of a different type: he used a prescribed mean curvature regularization.} 

The phase transition regularization forms the basis of our approach to Theorem \ref{theo:immersed.geo.min.max}, too. It allows us to study the singularities of the (\emph{a priori}) limiting stationary geodesic network \textit{before} they actually form (as the phase transition scale tends to zero), exactly as in \cite{Mantoulidis}. However, rather than work with any double-well potential as in \cite{CM:3d, Mantoulidis}, the novelty is that we choose a very specific potential (see \eqref{eq:SG-potential}) with particularly favorable properties. It is based on the (elliptic) sine-Gordon equation
\[
\Delta u = \sin u.
\] 
The relevance of the sine-Gordon equation is that with this precise potential it becomes an \emph{integrable PDE}. This has been employed in a recent remarkable work of Liu--Wei \cite{LiuWei} to give a full classification of finite index entire solutions to the sine-Gordon equation on $\RR^2$, as well as a computation of their Morse index and nullity. In the context of Theorem \ref{theo:immersed.geo.min.max} we need to rely on a consequence of the Liu--Wei result, namely, that the tangent cone at infinity to any entire solution of the phase transition regularization blows down to a varifold that is supported on lines through the origin. Given this fact, we can use the curvature estimates of Wang--Wei \cite{WangWei} to propogate this information outwards from the phase transition scale to the (original) manifold scale and get the desired conclusion about the tangent cone of the singularity. Similar arguments were used in \cite{CKM,Mantoulidis}.

Because the Liu--Wei result is so central to our work, and because the integrable PDE techniques are potentially unfamiliar, we have given a complete proof of the tangent cone theorem in this work, following the ideas of Liu--Wei. See Theorem \ref{theo:SG.cone}. At a heuristic level, one can see that integrability implies the tangent cone result by thinking of the ends of an entire solution as a initial pulse of a wave, which will then interact with the other ends in the compact region, but then (thanks to integrability) continue to propagate in the same direction. (Of course, the equation of interest here is elliptic, so this is not a very precise explanation.)

\subsection{Strategy of the proof of Theorem \ref{theo:sphere.pwidths}: bumpy metrics for geodesic nets and Lusternik--Schnirelmann theory}

We seek to deform the round metric $g_0$ to a nearby well-behaved metric $g_\mu$, $\mu = o(1)$, whose Almgren--Pitts $p$-widths we can guarantee to be:
\begin{enumerate}
	\item well-quantized, and
	\item strictly increasing.
\end{enumerate}
Once we have arranged (1) and (2), we use a counting argument to estimate the $p$-widths by an expression of the form
\[ 2 \pi \lfloor \sqrt{p} \rfloor \leq \omega_p(\SS^2, g_\mu) \leq (2\pi + 2\mu) \lfloor \sqrt{p} \rfloor, \]
and conclude (using the continuity of $g \mapsto \omega_p(\SS^2, g)$) by sending $\mu \to 0$.

Theorem \ref{theo:immersed.geo.min.max} provides a partial (but important) step toward (1): it ensures that the $p$-widths are attained by unions of immersed geodesics. So, one can hope to perturb the unit round metric $g_0$ to a nearby ellipsoidal metric $g_\mu'$, $\mu = o(1)$, whose only immersed geodesics (with controlled mass) are made up of iterates of three principal curves (this guarantees \textit{quantization}) and whose three principal lengths form an arithmetic progression (this guarantees \textit{good} quantization). Such metrics were already known to Morse  \cite{Morse:calculus.variations} and were instrumental (without the arithmetic progression property) in \cite{Aiex:ellipsoids}.

However, these metrics need not guarantee (2) above. But as is well-known in Lusternik--Schnirelmann theory, the failure of (2) implies that the set of candidate limiting objects (in our case, stationary integral 1-varifolds with controlled mass and singular set) has to have Lusternik--Schnirelmann category $\geq 1$. (This was highlighted by Aiex in \cite[Appendix A]{Aiex:ellipsoids}.) In particular, if we can guarantee the existence of a metric with few (three) principal geodesics, whose lengths form an arithmetic progression, and whose set of stationary geodesic networks with controlled mass and singular set has Lusternik--Schnirelmann category $0$, we are done.

This is arranged by Theorem \ref{theo:good.metric.final}, whose proof builds on two new tools:
\begin{enumerate}	
	\item[(a)] A proof that, for bumpy metrics on $\SS^2$, the space of stationary geodesic networks with bounded mass and singular sets has Lusternik--Schnirelmann category $0$. This is the content of a varifold/flat-chain covering lemma (Lemma \ref{lemm:cover.cF.bF}) inspired by the Marques--Neves \cite[\S 6]{MarquesNeves:posRic} covering in case the set in question were to be only finite (ours isn't) and a trichotomy theorem for stationary integral 1-varifolds that builds on a stratification of their moduli space that follows from refining the conclusions of Allard--Almgren \cite{AllardAlmgren:1varifold}; see Theorem \ref{theo:geodesic.network.trichotomy}.
	\item[(b)] A bumpy metrics theorem for stationary geodesic nets\footnote{Previous work on geodesic nets includes \cite{AllardAlmgren:1varifold,HassMorgan,Heppes,NabutovskyRotman:minimal.net,Rotman:shortest.net,NabutovskyRotman:shapes,IvanovTuzhilin:deform,IvanovTuzhilin:review,Parsch:thesis,NabutovskyParsch}.}  subject to a certain length constraint (Theorem \ref{theo:geodesic.network.constrained} and Corollary \ref{coro:geodesic.network.constrained.generic.smooth}). The unconstrained versions of these results are presented first (Theorem \ref{theo:geodesic.network.manifold} and Corollary \ref{coro:geodesic.network.manifold.generic.smooth}) for expository simplicity. We note that a version of the \textit{unconstrained} of the bumpiness theorem holding in all codimensions was independently proven by Staffa \cite{Staffa:bumpy.geodesic.nets} as a means to proving generic density of geodesic nets in all codimensions in his joint work with Liokumovich \cite{LiokumovichStaffa:density.geodesic.nets}. 
\end{enumerate}


\subsection{Organization of the paper}
In Section \ref{sec:AP.AC} we review background on the relevant min-max theories. In Section \ref{sec:SG} we specialize to phase transition min-max based on the sine-Gordon equation; this culminates in the proof of Theorem \ref{theo:immersed.geo.min.max} in Section \ref{sec:immersed.p.width}. The bumpy metric theorem for stationary nets is proven in Section \ref{sec:geo.net}. The Lusternik--Schnirelmann covering argument and choice of good metric used in the proof of Theorem \ref{theo:sphere.pwidths} is discussed in Section \ref{sec:LS}. The proof of Theorem \ref{theo:sphere.pwidths} is then completed in Section \ref{sec:sphere.p.widths}. Section \ref{sec:open} contains some open problems and further discussion. Appendix \ref{app:metric} contains some results about metric spaces, Appendix \ref{sec:GMT} recalls several notions from geometric measure theory, and Appendix \ref{sec:ac.regularity} contains an overview of regularity results for phase transitions. Finally, Appendix \ref{sec:upp.bds} contains a proof of Aiex's upper bounds for the $p$-widths coming from homogeneous polynomials.

\subsection{Acknowledgements}
O.C.\ was supported by an NSF grant (DMS-2016403), a Terman Fellowship, and a Sloan Fellowship. C.M. was supported by an NSF grant (DMS-2050120). We are grateful to Larry Guth, Yevgeny Liokumovich, Yong Liu, Luca Spolaor, Kelei Wang, Juncheng Wei, Brian White, and Alex Wright for useful discussions. We are also grateful to the anonymous referee for their careful reading of the manuscript.

\section{The $p$-widths and two min-max theories} \label{sec:AP.AC}

In this section we review the definition of $p$-widths and relevant properties of the  Almgren--Pitts and double-well phase transition min-max theories following \cite{Almgren:htpy,Pitts,MarquesNeves:Willmore,MarquesNeves:posRic,MarquesNeves:multiplicity,MarquesNeves:uper-semi-index,Guaraco,GasparGuaraco,Dey:APAC}. We will make heavy use of geometric measure theory. We direct the reader to Appendix \ref{sec:GMT}, where all the necessary geometric measure theory notation is presented. 

In this paper will write $\NN= \{0,1,2,\dots\}$ and $\NN^*=\{1,2,\dots\}$. We need the notions of a cubical complex and subcomplex. Let $I = [0,1]$. Denote by $I(1,j)$ the cube complex on $I^1$ with $1$-cells $[0,3^{-j}],[3^{-j},2 \cdot 3^{-j}],\dots,[1-3^{-j},1]$ and then define $I(m,j)$ to be the cell complex on $I^m$ given by the $m$-fold tensor product of $I(1,j)$ with itself. A \emph{cubical subcomplex} $X\subset I^k$ is a subcomplex of $I(k,j)$ for some $j \in \NN$. If $X$ is a subcomplex of $I(k,j)$, for $\ell\in\NN$, denote by $X(\ell)$ the subcomplex of $I(k,j+\ell)$ given by the union of all cells whose support is contained in some cell of $X$. Write $X(\ell)_0$ for the set of $0$-cells in $X(\ell)$.

Fix $(M,g)$ a closed Riemannian $2$-manifold for the rest of the section.

\subsection{Gromov--Guth $p$-widths} \label{subsec:p-widths}

 In what follows $X$ denotes a cubical subcomplex of some $I^k$. Recall that $\cZ_{1}(M;\ZZ_{2})$ is weakly homotopic to $\RR P^{\infty}$ (see \cite{Almgren:htpy} or \cite[Theorem 5.1]{MarquesNeves:uper-semi-index}), so
\[
H^{1}(\cZ_{1}(M;\ZZ_{2});\ZZ_{2}) = \ZZ_{2}.
\]
Denote by $\bar \lambda$ the generator. We define:

\begin{defi}[{\cite[Definitions 4.1]{MarquesNeves:posRic}}] \label{defi:sweepout}
	A map $\Phi : X\to \cZ_{1}(M;\ZZ_{2})$ is a \emph{$p$-sweepout} if it is continuous (with the standard flat norm ``$\cF$'' topology on $\cZ_1(M; \ZZ_2)$) and $\Phi^{*}(\bar\lambda^{p}) \neq 0$. 
\end{defi}

\begin{defi}[{\cite[\S 3.3]{MarquesNeves:posRic}}] \label{defi:no.concentration.of.mass}
	A map $\Phi : X \to \cZ_1(M; \ZZ_2)$ is said to have \textit{no concentration of mass} if
	\[ \lim_{r \to 0} \sup \{ \Vert \Phi(x) \Vert(B_r(p)) : x \in X, p \in M \} = 0. \]
\end{defi}

\begin{defi}[{\cite{Gromov:waist,Guth:minimax,MarquesNeves:posRic})}] \label{defi:p-width}
	We define $\cP_{p} = \cP_{p}(M)$ to be the set of all $p$-sweepouts, out of any cubical subcomplex $X$, with no concentration of mass. The \emph{$p$-width} of $(M,g)$ is
\[
\omega_{p}(M,g) = \inf_{\Phi \in \cP_{p}} \sup\{\bM(\Phi(x)) : x \in \dmn(\Phi)\}.
\]
\end{defi}

We also note the following lemma that is key in perturbative proofs such as ours of Theorem \ref{theo:sphere.pwidths} or that of Irie--Marques--Neves \cite{IrieMarquesNeves} for the generic existence of infinitely many hypersurfaces:

\begin{lemm}[{\cite[Lemma 2.1]{IrieMarquesNeves}}]\label{lem:p-width-cont-metric}
The $p$-width $\omega_p(M,g)$ depends continuously on $g$ with respect to the $C^0$-topology. 
\end{lemm}

\subsection{Almgren--Pitts theory}\label{subsec:cont-AP} 

Thanks to the interpolation theory developed by Almgren and Marques--Neves (\cite[\S 3]{MarquesNeves:posRic}), we can avoid discussing the discretized version of Almgren--Pitts theory and simply give references where necessary. We will work with a refined class of sweepouts that are continuous with respect to a stronger topology on $\cZ_1(M; \ZZ_2)$ than the $\cF$-norm topology, which is given by the $\bF$-metric. We write $\cZ_1(M; \bF; \ZZ_2)$ for the space with this topology. We have:

\begin{lemm}[{\cite[p. 472]{MarquesNeves:multiplicity}}] \label{lemm:ap.bf.is.good}
	Let $X$ be a cubical subcomplex. If $\Phi : X \to \cZ_1(M; \bF; \ZZ_2)$ is continuous, then it is also continuous with the $\cF$-topology on the target (i.e., $\Phi : X \to \cZ_1(M; \ZZ_2)$ is continuous) and has no concentration of mass. 
\end{lemm}

Such refined sweepouts still capture the $p$-widths $\omega_p(M, g)$, even if we restrict the dimension of the cubical subcomplexes used. That is, if:
\[ \cP^{\bF}_{p,m} : = \{\Phi \in \cP_p : \dmn(\Phi) \subset I^m \text{ and } \Phi \text{ is } \bF\text{-continuous} \} \]
(cf. \cite{Xu:pm-width}), then:

\begin{lemm} \label{lemm:P.p.m}
If $m = 2p+1$, then
\[
\omega_{p}(M,g) = \inf_{\Phi \in \cP^{\bF}_{p,m}} \sup\{\bM(\Phi(x)) : x \in \dmn(\Phi)\}.
\]
\end{lemm}
\begin{proof}
	This was shown in {\cite[Corollary 3.1]{Li:infinitely-high-dim}} for $\cF$-continuous maps with no concentration of mass. The result for $\bF$-continuous maps follows from Lemma \ref{lemm:ap.bf.is.good}, a discretization argument (see, e.g., the proof of \cite[Theorem 3.8]{MarquesNeves:multiplicity}, and \cite[Proposition 3.1]{MarquesNeves:uper-semi-index}.
\end{proof}

\begin{defi}[{\cite[Definitions 2.2, 2.4-2.7]{MarquesNeves:uper-semi-index}}] \label{defi:ap.homotopy}
	Let $X$ be a cubical subcomplex and fix a continuous $\Phi : X \to \cZ_1(M; \bF; \ZZ_2)$. We define the \textit{homotopy class} of $\Phi$ to be the set
	\begin{align*}
		\Pi & := \{ \text{continuous } \Phi' : X \to \cZ_1(M; \bF; \ZZ_2) \text{ that are} \\
			& \qquad \qquad \text{ homotopic to } \Phi \text{ in the } \cF \text{-topology} \}.
	\end{align*}
	The \emph{Almgren--Pitts width} of the homotopy class $\Pi$ is defined by
	\[
		\bL_{\textrm{AP}}(\Pi) = \inf_{\Phi \in \Pi} \sup_{x\in X} \bM(\Phi(x)). 
	\]
	We will write $\bL_{\textrm{AP}}(\Pi,g)$ when the dependence on the metric is relevant. 
	
	A sequence $\{\Phi_i\}_{i=1}^\infty  \subset \Pi$ is a \emph{minimizing sequence} if   
	\[
		\limsup_{i\to\infty}\sup_{x\in X} \bM(\Phi_i(x)) = \bL_\textrm{AP}(\Pi). 
	\]
	The image set $\mathbf{\Lambda}(\{\Phi_i\})$ of $\{\Phi_i\}$ is defined to be the set of $V \in \cV_1(M)$ so that there is $i_j\to\infty$ and $x_j\in X$ with $\lim_{j\to\infty}\bF(|\Phi_{i_j}(x_j)|,V) = 0$. Assuming that $\{\Phi_i\}$ is a minimizing sequence, the \emph{critical set} of $\{\Phi_i\}$ is 
	\[
		\bC(\{\Phi_i\}) = \{V\in\mathbf{\Lambda}(\{\Phi_i\}) : \Vert V\Vert(M) = \bL_\textrm{AP}(\Pi)\}. 
	\]
\end{defi}

Given these definitions, we proceed to summarize the main results of the Almgren--Pitts theory needed here, still following Marques--Neves. We first recall the ``pull-tight'' procedure, which improves arbitrary minimizing sequences $\{\Phi_i\}_{i=1}^\infty$ into ones whose critical set consists only of stationary varifolds.

\begin{prop}[{\cite[\S 2.8]{MarquesNeves:uper-semi-index}}]\label{prop:pull.tight}
Suppose $\bL_{\textrm{AP}}(\Pi)>0$. For any minimizing sequence $\{\Phi_i\}_{i=1}^\infty \subset \Pi$ there exists another  minimizing sequence $\{\Phi_i^*\}_{i=1}^\infty \subset \Pi$ with $\bC(\{\Phi_i^*\}_{i=1}^\infty) \subset \bC(\{\Phi_i\}_{i=1}^\infty)$ and every element of $\bC(\{\Phi_i^*\}_{i=1}^\infty)$ stationary. 
\end{prop}

It is possible to further improve minimizing sequences whose critical set consists of only stationary varifolds. We can arrange for the existence of one whose critical set has at least one stationary \textit{integral} varifold with a \textit{controlled} number of singular points. In our intended application of such a result, it will be important that we even allow the minimizing sequences to have varying domains. We can do this following \cite[\S 2.5]{MarquesNeves:posRic}. For $i \in \NN^*$, consider $Y_i$ cubical subcomplexes of $I^k$ and continuous maps $\Phi_i : Y_i\to\cZ_1(M;\bF;\ZZ_2)$. Set
\[
\bL_{\textrm{AP}}(\{\Phi_i\}) = \limsup_{i\to\infty} \sup_{x\in Y_i} \bM(\Phi_i(x))
\]
and define the image set $\Lambda(\{\Phi_i\})$ and critical set $\bC(\{\Phi_i\})$ in the obvious way (see \cite[\S 2.5]{MarquesNeves:posRic}).

\begin{prop}[{cf.\ \cite[Theorem 2.8]{MarquesNeves:posRic}}]\label{prop:am.vary.dom}
Fix $k \in \NN^*$ and assume that 
\[
\Phi_i : Y_i\to\cZ_1(M;\bF;\ZZ_2)
\]
is a sequence of continuous maps for $Y_i$ cubical subcomplexes of $I^k$ so that every $V\in\bC(\{\Phi_i\})$ is stationary in $M$. Then, at least one of the following holds:
\begin{enumerate}
	\item $\bC(\{\Phi_i\})$ contains a stationary integral varifold with $\leq 5^k$ singular points, or
	\item there exists a sequence of continuous $\Psi_i^* : Y_i \to \cZ_1(M;\bF;\ZZ_2)$, with each $\Psi_i^*$ homotopic to $\Psi_i$ in the $\cF$-topology, so that
\[
\bL_\textrm{AP}(\{\Psi_i^*\}) <  \bL_{\textrm{AP}}(\{\Psi_i\}).
\]
\end{enumerate}
\end{prop}
\begin{proof}
Lemma \ref{lemm:ap.bf.is.good}, \cite[Proposition 3.1]{MarquesNeves:uper-semi-index}, and \cite[Lemma 3.1]{Li:infinitely-high-dim} together  imply that either:
\begin{enumerate}
\item[(1')] there is some $V \in \bC(\{\Phi_i\})$ with the property that for any $5^k$ distinct points $\{p_j\}_{j=1}^{5^k}$ with minimal pairwise distance $d$, it holds that $V$ is almost minimizing in at least one of $\{B_{d/16}(p_j)\}_{j=1}^{5^k}$, or
\item[(2')]
 conclusion (2) above holds.
\end{enumerate}
(In \cite[Lemma 3.1]{Li:infinitely-high-dim}, the domain of the maps was assumed to be fixed as $i$ varies, but like in \cite[Theorem 2.8]{MarquesNeves:posRic}, the proof clearly extends to yield the given statement.) 

We may assume that (1') holds. It follows from \cite[Theorem 3.13]{Pitts} that $V \in \cI\cV_1(M)$ so by \cite[Section 3]{AllardAlmgren:1varifold}, $V$ is a stationary geodesic net. By \cite[Proposition 3.4]{ZhouZhu:networks} (cf.\ \cite[p.\ 3]{ZhouZhu:networks} and \cite{Aiex:ellipsoids}), if $V$ is almost minimizing in $U$ then $\sing V \cap U = \emptyset$. Thus, we find that $\# \sing V \leq 5^k$.
\end{proof}

\subsection{Double-well phase transition theory} \label{subsec:AC} 

\begin{defi} \label{defi:ac.potential.general}
A smooth function $W : \RR \to \RR$ is said to be a \textit{double-well potential} if it has the following properties:
\begin{enumerate}
\item[(W$_1$)] $W\geq 0$,
\item[(W$_2$)] $W(-t) = W(t)$ for all $t \in \RR$,
\item[(W$_3$)] $t W'(t) < 0$ for $0 < |t| < 1$,
\item[(W$_4$)] $W''(\pm 1) > 0$.
\end{enumerate}
\end{defi}

Fix a double-well potential $W$. For $\eps > 0$, define the \emph{$\eps$-phase transition energy} of a function $u \in C^\infty(M)$ to be 
\begin{equation} \label{eq:ac.energy}
	E_{\eps}[u] = \int_{M} \frac \eps 2 |\nabla u|^{2} + \frac 1 \eps W(u) 
\end{equation}
A critical point $u$ of $E_\eps$ necessarily solves the \emph{$\eps$-phase transition equation}
\begin{equation} \label{eq:ac.pde}
	\eps^2 \Delta u = W'(u).
\end{equation}
\begin{rema}
	In the special case where $W(t) = \tfrac14 (1-t^2)^2$, \eqref{eq:ac.pde} describes the Van der Waals--Cahn--Hilliard theory of phase transitions. This is not the potential we will choose to work with later. See Section \ref{sec:SG}.
\end{rema}

An easy computation shows the \emph{second variation} of $E_\eps$ at a critical point $u$ to be
\[
D^2 E_\eps[u]\{\zeta,\psi\} = \int_M \eps \bangle{\nabla \zeta,\nabla \psi} + \eps^{-1}W''(u) \zeta\psi \text{ for all } \zeta, \psi \in C^\infty(M).
\]
The second variation allows us to count the number of linearly unstable directions at a critical point. A solution $u$ of \eqref{eq:ac.pde} is said to have \emph{Morse index} $k \in \NN$ on $U\subset M$, denoted $\Index_\eps(u;U) = k$, if 
\begin{multline*}
\max \{\dim V : V\subset C^\infty_c(U) \textrm{ linear subspace with } D^2 E_\eps[u]\{\zeta,\zeta\} < 0 \\
\textrm{ for all } \zeta\in V\setminus\{0\}\} = k. 
\end{multline*}

We now discuss the min-max construction of solutions of \eqref{eq:ac.pde} following \cite{Guaraco, GasparGuaraco, Dey:APAC}. As in the previous subsection, $X$ is a cubical subcomplex of $I^{k}$ for some $k \in \NN^*$. 
Fix any double cover $\pi : \tilde X \to X$. Write $\Pi$ for the $\cF$-homotopy class of $\bF$-continuous maps corresponding to $\pi$, i.e., $\Phi : X\to \cZ_{1}(M;\bF;\ZZ_{2})$ is in $\Pi$ whenever\footnote{We use the $\cF$-topology (instead of the $\bF$-topology) on $\cZ_1(M;\ZZ_2)$ when computing $\ker \Phi_*$, since our homotopies are always $\cF$-homotopies.}
\[
\ker(\Phi_{*} : \pi_1(X) \to \pi_1(\cZ_1(M;\ZZ_2))) =\im \pi_{*} \subset \pi_{1}(X).
\]
Note that fixing $\Pi$ is the same as fixing the double cover $\pi:\tilde X\to X$; we will implicitly use this in many places below.

Since $\bI_{2}(M;\ZZ_{2})$ is contractible and paracompact, 
\[
\partial : \bI_{2}(M;\ZZ_{2}) \to \cZ_{1}(M;\ZZ_{2})
\]
is a $\ZZ_{2}$-principal bundle. Since $H^{1}(M)\setminus\{0\}$ is contractible and paracompact with a free $\ZZ_{2}$ action $u\mapsto -u$, it is the total space of a $\ZZ_{2}$-principal bundle. We denote by $\tilde \Pi$ the space of $\ZZ_{2}$-equivariant maps $h: \tilde X \to H^{1}(M)\setminus\{0\}$. We define \emph{$\eps$-phase transition width} of $\tilde \Pi$ by 
\[
\bL_{\eps}(\tilde \Pi) = \inf_{h\in \tilde\Pi} \sup_{x\in\tilde X} E_{\eps}(h(x))
\]
We say that $u \in H^{1}(M)\setminus\{0\}$ is a \emph{min-max critical point} if $E_{\eps}(u) = \bL_{\eps}(\tilde\Pi)$ and there is a minimizing sequence $\{h_{i}\}_{i=1}^{\infty}\subset \tilde\Pi$ with
\[
\lim_{i\to\infty} d_{H^{1}(M)}(u,h_{i}(\tilde X)) = 0. 
\]
The main existence result for min-max critical points is as follows. 

\begin{prop}[{\cite[Proposition 4.4]{Guaraco}, \cite[Theorem 3.3]{GasparGuaraco}; cf.\ \cite[\S 2.4]{Dey:APAC}}]\label{prop:AC.min.max}
If $\bL_{\eps}(\tilde \Pi) < E_{\eps}(0) = \Vol(M,g)/\eps$ then there is a min-max critical point $u_{\eps}$ of $E_{\eps}$; the function $u_\eps$ satisfies $|u_\eps| < 1$, solves \eqref{eq:ac.pde}, and has $\Index_\eps(u_\eps) \leq \dim X = k$.
\end{prop}

\subsection{Comparison between the min-max theories}\label{subsec:dey}

The following summarizes the equivalence between the two theories:

\begin{prop}[{\cite[Theorem 6.1]{GasparGuaraco}, \cite[Theorem 1.2]{Dey:APAC}}]\label{prop:AC-AP}
	The $\eps$-phase transition widths and Almgren--Pitts width are related by
	\[ h_0^{-1} \lim_{\eps\to0}\bL_{\eps}(\tilde\Pi) = \bL_{\textnormal{AP}}(\Pi), \] 
	where $h_0$ is the squared $L^2$ energy of the heteroclinic solution 
	\begin{equation} \label{eq:ac.heteroclinic}
		\mathbb{H} : \RR \to (-1, 1), \; \mathbb{H}(0) = 0, \; \lim_{t \to \pm \infty} \mathbb{H}(t) = \pm 1
	\end{equation}
	of \eqref{eq:ac.pde} on $\RR$ with $\eps=1$.
\end{prop}

In fact, more is true. Recall that every solution $u$ of \eqref{eq:ac.pde} has an associated 1-varifold on $M$, the \emph{$\eps$-phase transition 1-varifold}. It is defined as the unique  $V_\eps[u] \in \cV_1(M)$ such that
\[ V_\eps[u]\{f\} := h_0^{-1} \int_M \eps |\nabla u|^2 f(x, \Tan_x \{ u = u(x) \}) \text{ for all } f \in C^0(G_1(M)). \]
The Hutchinson--Tonegawa compactness theorem (Proposition \ref{prop:HT.theory}) shows that if $u_i$ are a solutions of \eqref{eq:ac.pde} with $\eps_i \to 0$ and suitable a priori $L^\infty$ and $\eps_i$-energy bounds (which hold along our min-max critical points $(u_i, \eps_i)$ in Proposition \ref{prop:AC-AP}), then the corresponding 1-varifolds $V_{\eps_i}[u_i]$ subsequentially converge to a stationary integral 1-varifold $V$ with mass $h_0^{-1} \lim_i E_{\eps_i}[u_i]$.

If the set of all limiting stationary integral 1-varifolds arising from our phase transition min-max critical points $(u_i, \eps_i)$, $\eps_i \to 0$, is denoted $\bC_{\textnormal{PT}}(\tilde \Pi)$, and the set of all Almgren--Pitts min-max critical points produced in Section \ref{subsec:cont-AP} is denoted $\bC_{\textnormal{AP}}(\Pi)$, then:
\begin{prop}[{\cite[Theorem 1.4]{Dey:APAC}}]\label{prop:phase.vs.ap.limit}
$\bC_{\textnormal{PT}}(\tilde \Pi) \subset \bC_{\textnormal{AP}}(\Pi)$.
\end{prop}
We will not use this result here but it is relevant to some of the discussion in Section \ref{sec:open}.

\section{The sine-Gordon double-well potential} \label{sec:SG}

The following is the main result of this section. 

\begin{theo}[Sine-Gordon limit theorem]\label{theo:SG-lim}
	Let $(M,g)$ be a closed Riemannian 2-manifold. Fix the ``sine-Gordon'' double-well potential
	\begin{equation} \label{eq:SG-potential}
		W(t) := \frac{1 + \cos(\pi t)}{\pi^2}.
	\end{equation}
	(Note that $W$ satisfies Definition \ref{defi:ac.potential.general}.) Let $u_{i}\in C^\infty(M)$ be solutions of \eqref{eq:ac.pde} on $(M,g)$ with the sine-Gordon potential \eqref{eq:SG-potential}, $\eps_i \to 0$, and such that
	\begin{equation} \label{eq:SG-lim.finite.complexity}
		\Index_{\eps_i}(u_i) + E_{\eps_{i}}[u_{i}] \leq \Lambda \text{ for all } i = 1, 2, \ldots
	\end{equation}
	Then, passing to a subsequence, the $\eps_i$-phase transition 1-varifolds $V_{\eps_{i}}[u_{i}]$ converge to a stationary integral 1-varifold $V$ such that
	\[ V = \sum_{j=1}^N \mathbf{v}(\sigma_j, \mathbf{1}_{\sigma_j}) \]
	for $\sigma_1, \ldots, \sigma_N$ (possibly repeated) primitive closed geodesics in $(M, g)$.
\end{theo}

\begin{rema}
	One can derive bounds on $\sing V$ similarly to \cite{Mantoulidis}, but we do not need them for our proof of Theorem \ref{theo:sphere.pwidths}.
\end{rema}

\subsection{Phase transitions on the plane for general double-well potentials}

We prove Theorem \ref{theo:SG-lim} by analyzing the singularities of $V$ before they occur in the $\eps_i \to 0$ limit. To do so, we perform blow-ups of our phase transitions near the singularities that are about to form. This leads us to consider nontrivial phase transitions on all of $\RR^2$. We need a few definitions.

We call $u \in C^\infty(\RR^2)$ an \emph{entire phase transition} on $\RR^2$ that is regular at infinity if $|u| < 1$, $u$ solves \eqref{eq:ac.pde} on $\RR^2$ (with any double-well potential), and\footnote{It is a deep result of Wang--Wei \cite{WangWei} that the energy growth bound in \eqref{eq:entire.finite.complexity} is \textit{implied} by the index bound. Thus, a bounded entire solution is finite-index if and only if it is regular at infinity. We will not need this fact in our paper.}
\begin{equation} \label{eq:entire.finite.complexity}
	\limsup_{R\to\infty} \left( \Index_1(u;B_R) + R^{-1} (E_1 \restr B_R(0))[u] \right) < \infty. 
\end{equation}
Notice that, by Lemma \ref{lemm:monotonicity.AC.energy} and \eqref{eq:SG-lim.finite.complexity}, blow ups of solutions $(u_i, \eps_i)$ to Theorem \ref{theo:SG-lim}, rescaled so that $\eps_i \mapsto 1$ and taking $i \to \infty$, are all of this form.

The study of such solutions on $\RR^2$ has a rich literature (cf.\ \cite{dPKPW:multiple-end,Gui:symmetry,KowalczykLiuPacard12,KLP:towards-4-end,delPinoKowalczykPacard:moduli,KLP:classification,KLPW:end-to-end,GLW:variational,Wang:finite-morse-index,WangWei}), and we briefly summarize the facts needed in our work. We first recall a result of Wang:

\begin{prop}[{\cite[Theorem 1.1]{Wang:finite-morse-index}}]\label{prop:tangent.cone.infty.AC} 
	Let $u$ be an entire phase transition on $\RR^2$ that is regular at infinity. Then, there exist distinct unit vectors $v_1,\dots,v_{2m} \in \RR^2$ so that
	\[
		V^\infty := \sum_{j=1}^{2m} \bv([0,\infty)v_j,\mathbf{1}_{[0,\infty)v_j}) \in \cI \cV_1(\RR^2) \text{ is stationary}
	\]
	and for all $\{ \lambda_i \}_{i=1}^\infty \subset (0, \infty)$ with $\lim_i \lambda_i = \infty$, the $\lambda_i^{-1}$-phase transition 1-varifolds of the rescalings $u_i(x) := u(x/\lambda_i)$ (which solve \eqref{eq:ac.pde} with $\eps = \lambda_i^{-1}$) satisfy
	 \[ \lim_{i} V_{\lambda_i^{-1}} [u_i] \restr G_1(B_R) = V^\infty \restr G_1(B_R) \text{ for all } R>0. \]
\end{prop}

\begin{defi} \label{defi:asymptotic.directions}
For $u$ an entire phase transition on $\RR^2$ that is regular at infinity we say that the $V^\infty$ of Proposition \ref{prop:tangent.cone.infty.AC} is the \emph{tangent cone to $u$ at infinity} and call $\{v_1,\dots,v_{2m}\} \subset \SS^1$ the \emph{asymptotic directions of $u$}. For each asymptotic direction $v_j$ of $u$, define $v_j^\perp$ to be the unit vector orthogonal to $v_j$ oriented so that, when taken along $[0, \infty) v_j \subset \supp V^\infty$, it points into the ``$+1$'' limit region of the BV limit of the $u_i$ above per Proposition \ref{prop:HT.theory}.
\end{defi}

It is a consequence of Proposition \ref{prop:tangent.cone.infty.AC} that entire phase transitions which are regular at infinity have a unique tangent cone at infinity or, equivalently, unique asymptotic directions, and they are precisely $2m$-ended solutions as considered by del Pino--Kowalczyk--Pacard \cite{delPinoKowalczykPacard:moduli}.

\begin{prop}[{\cite[Theorem 1.4]{Wang:finite-morse-index}, \cite[Theorem 2.1]{delPinoKowalczykPacard:moduli} cf. \cite[Proposition 3.10]{Mantoulidis}}] \label{prop:asymp.dir.behavior}
	 Let $u$ be an entire phase transition on $\RR^2$ which is regular at infinity with asymptotic directions $\{ v_1,\dots,v_{2m} \} \subset \SS^1$. There are $C, \kappa > 0$ depending on $u$ and the double-well potential $W$ so that for $x \in \RR^2$,
\begin{equation}\label{eq:exp.decay.entire}
(1-u(x)^2) + |\nabla u(x)| \leq C e^{-\kappa \mathscr{D}(x)}
\end{equation}
where 
\[
\mathscr{D}(x) := \dist\left( x , \cup_{j=1}^{2m} [0,\infty)v_j\right).
\]
Furthermore, for any $\{(z_i,r_i) \}_{i=1}^\infty \subset \SS^1\times (0,\infty)$ with $\lim_i z_i = z_\infty$ and $\lim_i r_i = \infty$, then up to passing to a subsequence (not relabeled),
\[
u_{i}(\cdot) =  u(\cdot + r_{i}z_{i})
\]
converges in $C^{\infty}_{\textnormal{loc}}(\RR^{2})$ to a function $u_\infty$ and either
\begin{enumerate}
	\item there exists $j \in \{ 1, \ldots, 2m \}$ such that
		\begin{equation}\label{eq:ends.deviate.asymptot.dir.bd.dist}
			\lim_i z_i = v_j \text{ and } \lim_i r_i (v_j^\perp \cdot z_i) \text{ exists and is finite},
		\end{equation}
		(where $v_j^\perp$ is as in Definition \ref{defi:asymptotic.directions}) in which case
		\[ u_\infty(x, y) = \HH((x,y) \cdot v_j^\perp + \eta_j) \]
		where $\HH$ is as in \eqref{eq:ac.heteroclinic} and $\eta_j \in \RR$; or,
	\item $u_\infty \equiv \pm 1$.
\end{enumerate} 
\end{prop}

\subsection{Reduction of the sine-Gordon limit theorem}

The following remarkable result is contained in the work of Liu--Wei \cite{LiuWei}. In fact, they prove much more, namely a full classification of entire phase transitions on $\RR^2$ that are regular at infinity along with a computation of their Morse index and nullity, \textit{provided} one uses the sine-Gordon double well-potential \eqref{eq:SG-potential}. We just need to study the asymptotic directions, and thus only need the following consequence of their work:

\begin{theo}[\cite{LiuWei}]\label{theo:SG.cone}
	If $u$ is an entire solution of \eqref{eq:ac.pde} on $\RR^2$ with the sine-Gordon potential \eqref{eq:SG-lim.finite.complexity}, which is regular at infinity, then up to relabeling the asymptotic directions of $u$, we have $v_{2k+1} = - v_{2k}$ for $k=1,\dots,m$. In other words, there are lines $\ell_1,\dots,\ell_m\subset \RR^2$ through the origin so that the tangent cone of $u$ at infinity is
	\[
		V^\infty = \sum_{j=1}^m \bv(\ell_j,\mathbf{1}_{\ell_j}).
	\]
\end{theo}
Since this result is central to this paper, we give a complete proof in Section \ref{subsec:Liu.Wei.proof}, but we emphasize that our proof will follow the work \cite{LiuWei} reasonably closely and we will not prove the full uniqueness/nondegeneracy result obtained there.

\begin{rema}
Theorem \ref{theo:SG.cone} is the only place that the exact form of the potential $W$ enters. It would be interesting to know if counterexamples to such  a statement exist for other potentials $W$ such as the standard $W(t) = \frac 14 (1-t^2)^2$. 
\end{rema}

In this section, we show how the Liu--Wei tangent cone theorem (Theorem \ref{theo:SG.cone}) implies the sine-Gordon limit theorem (Theorem \ref{theo:SG-lim}). Specifically, we will deduce Theorem \ref{theo:SG-lim} from the following blow-up result, which is related to arguments used in \cite[\S 4]{CKM} and \cite[\S 4.4]{Mantoulidis}. 

\begin{prop}\label{prop:induct.SG.lim.thm}
Suppose that we are working with the sine-Gordon potential \eqref{eq:SG-potential}, that Theorem \ref{theo:SG.cone} is true, and that:
\begin{itemize}
\item $\{R_i\}_{i=1}^\infty \subset (0,\infty)$ with $\lim_i R_i = \infty$;
\item  $\{g_i\}_{i=1}^\infty\subset \met(B_{R_i}(0)\subset \RR^2)$ with $\lim_i g_i =$ Euclidean metric in $C^\infty_\textnormal{loc}(\RR^2)$;
\item $\{(u_i,\eps_i)\})_{i=1}^\infty \subset C^\infty(B_{R_i}) \times (0,\infty)$ satisfying \eqref{eq:ac.pde} on $B_{R_i}$;
\item $\limsup_{i\to\infty} (E_{\eps_i}\restr B_R(0))[u_i] \leq CR$ for some $C > 0$;
\item $\limsup_{R\to\infty} \limsup_{i\to\infty} \Index_{\eps_i}(u_i;B_R(0)) \leq I \in \NN$. 
\end{itemize}
Then, after to passing to a subsequence, there are lines $\ell_1,\dots,\ell_m \subset \RR^2$ (with repetition allowed) so that 
\[
\lim_i V_{\eps_i}[u_i] = \sum_{j=1}^m \bv(\ell_j,\mathbf{1}_{\ell_j}). 
\]
\end{prop}
\begin{proof}
We induct on $I$. If $I=0$ the result easily follows from the curvature estimates for stable solutions in Proposition \ref{prop:WW.est} (in this case, all of the lines $\ell_i$ are parallel). 

Assuming the assertion holds for the index $\leq I-1$, we consider 
\[
\{(R_i, g_i, u_i, \eps_i)\}_{i=1}^\infty
\]
as in the statement of the proposition (where we assume index $\leq I$). We can apply the Hutchinson--Tonegawa compactness theorem (Proposition \ref{prop:HT.theory}) and pass to a subsequence to find a stationary integral $1$-varifold 
\[ V := \lim_i V_{\eps_i}[u_i] \in \cI \cV_1(\RR^2). \]
By \cite[\S 3]{AllardAlmgren:1varifold}, $V$ is a geodesic net and for any $x \in \supp V$ there is a unique tangent cone. Assume, for the sake of contradiction, that $V$ is not the varifold associated to the union of lines. There must then be a singular point on $V$, which we may assume to be $0 \in \sing V$ for simplicity, so that
\begin{equation} \label{eq:induct.SG.lim.thm.vartan}
	\VarTan(V, 0) = \sum_{j=1}^{2m} \bv([0,\infty)v_j,\mathbf{1}_{[0,\infty)v_j})
\end{equation}
and unit vectors $v_1,\dots,v_{2m}\in\RR^2$ with repetitions allowed\footnote{ The even parity is a consequence of Proposition \ref{prop:HT.theory}} so that
\begin{equation} \label{eq:induct.SG.lim.thm.nolines}
	 -v_j \not \in \{ v_1, \ldots, v_{2m} \} \text{ for some } j \in \{ 1, \ldots, 2m \}.
\end{equation}
By a dilation, we can assume that
\begin{equation} \label{eq:induct.SG.lim.thm.nosing}
	\sing V \cap B_2(0) = \{0\}.	
\end{equation}

\begin{claim} \label{clai:induct.SG.lim.thm.limr}
There is $\rho \in (0,2)$ so that 
\[
\limsup_{r\to0}\limsup_{i\to\infty} \Index_{\eps_i}(u_i; B_\rho(0)\setminus \bar B_r(0)) = 0
\]
\end{claim}
\begin{proof}[Proof of claim]
If the claim failed, there would exist a map $\rho \mapsto r(\rho)$ with $r(\rho) < \rho$ and 
\[
\limsup_{i\to\infty} \Index_{\eps_i}(u_i; B_\rho(0)\setminus \bar B_{r(\rho)}(0)) \geq 1
\]
for all $\rho \in (0,2)$. Set $\rho_1 = 1$ and inductively define $\rho_k = r(\rho_{k-1})/2$. We can pass to a subsequence of $u_i$ (not relabeled) so that 
\[
\Index_{\eps_i}(u_i; B_{\rho_k}(0)\setminus \bar B_{r(\rho_k)}(0)) \geq 1
\]
for all $i\in\NN^*$ and $k\in \{1,\dots,I+1\}$. Since $\rho_k < r(\rho_{k-1})/2$, we have found $I+1$ disjoint unstable regions for $u_i$, contradicting the index bound. 
\end{proof}
Dilating again, we can assume that $\rho =2$, i.e., 
\[
\limsup_{r\to0}\limsup_{i\to\infty} \Index_{\eps_i}(u_i; B_2(0)\setminus \bar B_r(0)) = 0. 
\]
For $q \in \bar B_1$, define
\begin{equation} \label{eq:induct.SG.lim.thm.r}
	r_i(q) : = \inf\{r>0 : \Index_{\eps_i}(u_i;B_2(0)\setminus \bar B_r(q)) = 0\}.
\end{equation}
Set $r_i : = \inf_{q\in \bar B_1} r_i(q)$. Note that, by Claim \ref{clai:induct.SG.lim.thm.limr},
\begin{equation} \label{eq:induct.SG.lim.thm.limr}
	\lim_i r_i = 0.
\end{equation}
\begin{claim}
There are $p_i \in \bar B_1(0)$ with $r_i(p_i) = r_i$ and $\lim_i p_i = 0$.
\end{claim}
\begin{proof}[Proof of claim]
Choose $\{p_{i,\ell}\}_{\ell=1}^\infty\subset \bar B_1$ with $\lim_\ell r_i(p_{i,\ell}) = r_i$. Set 
\[ p_i := \lim_\ell p_{i,\ell}. \]
Let $\delta > 0$ be arbitrary. For all $\ell$ large enough so that $|p_{i,\ell}-p_i|<\delta$, we have
\[ B_2(0)\setminus \bar B_{r_i(p_{i,\ell})+2\delta}(p_i) \subset B_2(0)\setminus \bar B_{r_i(p_{i,\ell})+\delta}(p_{i,\ell}). \]
This implies that $r_i(p_i) \leq \lim_\ell r_i(p_{i,\ell}) + 2\delta= r_i+2\delta$. Since $\delta>0$ was arbitrary, $r_i(p_i) \leq r_i$. Obviously $r_i(p_i)\geq r_i$, so $r_i(p_i) = r_i$.

Finally, if $\lim_i p_i = p\not = 0$, then by \eqref{eq:induct.SG.lim.thm.limr}:
\[ \Index_{\eps_i}(u_i;B_{|p|/2}(0)) = 0 \text{ for all } i \text{ sufficiently large}. \]
This would violate $0\in\sing \tilde V$ by Proposition \ref{prop:WW.est}.\footnote{ The curvature estimates of \cite{Tonegawa05} would suffice here.} So, $\lim_i p_i = 0$.
\end{proof}

Pass to a subsequence so that either:
\begin{enumerate}
	\item[(A)] $\lim_i r_i^{-1} \eps_i = 0$, or
	\item[(B)] $\lim_i r_i^{-1} \eps_i \in (0,\infty]$.
\end{enumerate}

We begin with case (A).

\begin{claim} \label{clai:induct.SG.lim.thm.A}
	For $i$ sufficiently large, 
	\[ u_i^{-1}(0)\cap (\bar B_{1}(0)\setminus  B_{2r_i}(p_i)) = \cup_{j=1}^{2m} \beta_{i,j} \]
	for pairwise-disjoint curves $\{ \beta_{i,j} \}_{j=1}^{2m}$ that are properly embedded in $\bar B_1(0) \setminus B_{2r_i}(p_i)$ that additionally satisfy
	\begin{equation} \label{eq:induct.SG.lim.thm.A}
		\lim_{i \to \infty} \max_{s \in [0, T_{i,j}]} |\beta'_{i,j}(s) - v_j| = 0
	\end{equation}
	when parametrized by unit speed with $\beta_{i,j}(0) \in \partial B_{2r_i}(p_i)$, $\beta_{i,j}(T_{i,j}) \in \partial B_1(0)$.
\end{claim}
\begin{proof}[Proof of claim]
By rescaling the curvature estimates in Proposition \ref{prop:WW.est} (cf.\ \cite[Theorem 3.8]{WangWei} and \cite[(4.8)]{Mantoulidis}), for $i$ sufficiently large and all $y_i \in u_i^{-1}(0)\cap (\bar B_{1}(0)\setminus B_{2r_i}(p_i))$ we have the curvature estimate
\begin{equation} \label{eq:induct.SG.lim.thm.A.curv}
	|\cA_{u_i}|(y_i) \leq C \eps_i^{\theta} D_i(y_i)^{-1-\theta}
\end{equation}
where 
\begin{equation} \label{eq:induct.SG.lim.thm.A.dist}
	D_i(y_i) := \dist_{g_i}(y_i, \bar B_{r_i}(p_i)).
\end{equation}
Indeed, we have
\[
\Index_{\eps_i}(u_i,B_{D_i(y_i)}(y_i)) = 0
\]
and rescaling by $D_i(y_i)$ around $y_i$ yields a stable solution $\tilde u_i$ on $(B_1(0),\tilde g_i)$ to $\tilde\eps_i^2\Delta_{\tilde g_i}\tilde u_i = W'(\tilde u_i)$, where 
\[
\tilde\eps_i := D_i(y_i)^{-1} \eps_i \leq r_i^{-1} \eps_i
\]
and $\tilde g_i$ limits to the flat metric in $C^\infty_\textrm{loc}(B_1(0))$. Thus, our assumption (A) guarantees that $\tilde\eps_i \leq \eps_0$ (with $\eps_0$ as defined in Proposition \ref{prop:WW.est}) for large $i$, so
\[
 \tilde \eps_i |\nabla \tilde u_i| \geq C^{-1}, \qquad |\cA_{\tilde u_i}|(0) \leq C \tilde \eps_i^{\theta}
\]
Rescaling back yields \eqref{eq:induct.SG.lim.thm.A.curv}.

We can now  integrate \eqref{eq:induct.SG.lim.thm.A.dist} as in \cite[(4.11)]{Mantoulidis} to prove the claim. Indeed, define $s_i$ to be the infimum of $s \in [2r_i,1/2]$ so that for all $s' \in [s,1/2]$, $u_i^{-1}(0)$ intersects $\partial B_{s'}(p_i)$ transversely in $2m$ points and so that if $w$ is a unit tangent vector to $u^{-1}(0)$ at some $x \in u^{-1}(0)\cap \partial B_{s'}(p_i)$ then $|w \cdot x| \geq 3s'/4$. By the curvature estimates (see Remark \ref{rema:gen.sff}), we see that $\lim_i s_i =0$. Furthermore, we can find pairwise disjoint embedded curves 
\[
\{\beta_{i,j} : [0,T_{i,j}] \to (\bar B_1(0)\setminus B_{s_j}(p_j) :  j \in \{1,\dots,2m\}\}
\]
parametrized by unit speed so that $\beta_{i,j}(0) \in \partial B_{s_i}(p_i)$, $\beta_{i,j}(T_{i,j})\in\partial B_1(0)$, and
\[
u_i^{-1}(0) \cap (\bar B_1(0)\setminus B_{s_j}(p_j)) = \cup_{j=1}^{2m} \beta_{i,j}([0,T_{i,j}]). 
\]
Recalling \eqref{eq:induct.SG.lim.thm.vartan}, \eqref{eq:induct.SG.lim.thm.nosing}, we can assume that the ordering of the curves is fixed so that $\lim_i \beta_i(T_{i,j}) = v_j$ for all $j \in \{ 1, \ldots, 2m \}$ and thus
\begin{equation} \label{eq:induct.SG.lim.thm.A.tang}
	\lim_i \beta_{i,j}'(T_{i,j}) = v_j \text{ for all } j \in \{ 1, \ldots, 2m \}.
\end{equation}
The transversality assumption and \eqref{eq:induct.SG.lim.thm.A.dist} imply
\[
	D_i(\beta_{i,j}(t)) \geq s_i - r_i + ct \text{ for all } t \in [0, T_{i,j}]
\]
with $c>0$ independent of $i,j$. Integrating the curvature estimate, we find (allowing $C$ to change from line to line, but not depend on $i,j$)  that for all $t_1$, $t_2 \in [0, T_{i,j}]$:
\begin{align*}
	|\beta_{i,j}'(t_2) - \beta_{i,j}'(t_1)|
		& \leq C \eps_i^{\theta} \int_0^{T_{i,j}} D_i(\beta_{i,j}(t))^{-1-\theta} dt\\
		& \leq C \eps_i^{\theta} \int_0^{T_{i,j}} (s_i - r_i + ct)^{-1-\theta} dt\\
		& \leq C \eps_i^{\theta} (s_i-r_i)^{-\theta} \leq C (\eps_i / r_i)^{\theta}
\end{align*}
since $s_i\geq 2r_i$. Because we have assumed condition (A), \eqref{eq:induct.SG.lim.thm.A.tang} implies
\[
\lim_i \max_{t \in [0, T_{i,j}]} |\beta'_{i,j}(t) - v_j| = 0.  
\]
This proves that, for $i$ sufficiently large, $s_i = 2r_i$ and thus the curves $\beta_{i,j}$ we already constructed have the asserted properties. 
\end{proof}

Define 
\[
\{(\tilde R_i,\tilde g_i,\tilde u_i,\tilde \eps_i)\}_{i=1}^\infty \subset (0,\infty)\times \met(B_{\tilde R_i}(0)) \times C^\infty(B_{\tilde R_i}(0)) \times (0,\infty)
\]
by dilating by $r_i^{-1}$ around $p_i$ (so $\tilde \eps_i = r_i \eps_i$ and so on). We have 
\[
\tilde \eps_i^2 \Delta_{\tilde g_i} \tilde u_i = W'(\tilde u_i). 
\]
By Lemma \ref{lemm:monotonicity.AC.energy}, it holds that 
\[
\limsup_{i\to\infty} (E_{\tilde \eps_i}\restr B_R(0))[\tilde u_i] \leq CR 
\]
for all $R>0$. Thus, by Proposition \ref{prop:HT.theory}, we can pass to a subsequence and find a stationary integral $1$-varifold $\tilde V \in \cI\cV_1(\RR^2)$ so that $V_{\tilde\eps_i}[\tilde u_i]$ converges to $\tilde V$. Note that
\begin{equation} \label{eq:induct.SG.lim.thm.vtilde.sing}
	\sing \tilde V\subset \bar B_1(0)
\end{equation}
thanks to the choice of rescaling, the definition of $r_i$ in \eqref{eq:induct.SG.lim.thm.r}, and Proposition \ref{prop:WW.est}.\footnote{The curvature estimates of \cite{Tonegawa05} would suffice here.}

By the scaling-invariance of Claim \ref{clai:induct.SG.lim.thm.A}'s \eqref{eq:induct.SG.lim.thm.A}, the tangent cone at infinity to $\tilde V$ is 
\[
\VarTan(\tilde V, \infty) = \sum_{j=1}^{2m} \bv([0,\infty)v_j,\mathbf{1}_{[0,\infty)v_j})
\]
where the right hand side is as in \eqref{eq:induct.SG.lim.thm.vartan}, \eqref{eq:induct.SG.lim.thm.nolines}. In particular, $\tilde V$ is also not a varifold associated to a union of lines and therefore the same is true for at least one of its tangent cones at a singular point. There are two possibilities:
\begin{enumerate}
	\item[(A$_1$)] $\# \sing \tilde V = 1$, or
	\item[(A$_2$)] $\# \sing \tilde V \geq 2$.
\end{enumerate}

We begin with (A$_1$). Write $\sing \tilde V = \{z\}$ and note that $z \in \bar B_1(0)$ by \eqref{eq:induct.SG.lim.thm.vtilde.sing}. By definition of $r_i$ in \eqref{eq:induct.SG.lim.thm.r}, and undoing the rescaling procedure to get a point $z_i$ (out of $z$) on the original scale, we find that
\[
\limsup_{i\to\infty} \Index_{\eps_i}(u_i; B_\rho(0) \setminus \bar B_{r_i/2}(z_i)) \geq 1.
\]
Therefore, after passing to a subsequence and returning to the current scale,
\begin{equation}\label{eq:SG.lim.thm.induct.drop.index.A1}
\Index_{\tilde \eps_i}(\tilde u_i;B_{1/2}(z)) \leq I - 1.
\end{equation}
We now choose $\lambda_i\to\infty$ sufficiently slowly so that defining 
\[
\{(\hat R_i,\hat g_i,\hat u_i,\hat \eps_i)\}_{i=1}^\infty \subset (0,\infty)\times \met(B_{\hat R_i}(0)) \times C^\infty(B_{\hat R_i}(0)) \times (0,\infty)
\]
by dilating $(\tilde R_i,\tilde g_i,\tilde u_i,\tilde \eps_i)$ by $\lambda_i$ around $z$, it holds that $V_{\hat \eps_i}[\hat u_i]$ converges to $\VarTan(\tilde V, z)$. (This uses the monotonicity formula, Lemma \ref{lemm:monotonicity.AC.energy}, in the same way as above.) Since $z$ was the only singular point, $\VarTan(\tilde V, z)$ is not a union of lines. Thanks to \eqref{eq:SG.lim.thm.induct.drop.index.A1}, 
\[
\{(\hat R_i,\hat g_i,\hat u_i,\hat \eps_i)\}_{i=1}^\infty 
\]
satisfies the hypothesis of the proposition but with index $\leq I-1$ instead of $I$. By the inductive hypothesis, this is a contradiction. 

We now consider case (A$_2$), i.e., there are $z_1\neq z_2 \in \sing\tilde V$. As before, we can assume that $\VarTan(\tilde V, z_1)$ is not the union of lines. For $\delta<\tfrac 12 |z_{1}-z_{2}|$, the curvature estimates\footnote{The curvature estimates of \cite{Tonegawa05} would suffice here} imply that $\Index_{\tilde \eps_i}(\tilde u_i;B_{\delta}(z_{2})) \geq 1$ for $i$ sufficiently large, implying that
\[
\Index_{\tilde \eps_{i}}(\tilde u_{i};B_{\delta}(z_{1})) \leq I-1
\]
for $i$ large. As in the previous case, we can dilate around $z_{1}$ by $\lambda_{i}\to\infty$ sufficiently slowly to obtain 
\[
\{(\hat R_i,\hat g_i,\hat u_i,\hat \eps_i)\}_{i=1}^\infty 
\]
satisfying the hypothesis of the proposition, but with index $\leq I-1$, a contradiction as before. 

It remains to consider case (B), namely $\lim_i r_i^{-1}\eps_i \in (0,\infty]$. As in case (A), we begin by showing that we can parametrize the zero set in an appropriate region. 
\begin{claim}[cf. Claim \ref{clai:induct.SG.lim.thm.A}] \label{clai:induct.SG.lim.thm.B}
	Fix $\mu>0$. There is $R = R(\mu)>0$ so that for $i$ sufficiently large,
	\[
	u_i^{-1}(0)\cap (\bar B_{1}(0)\setminus  B_{R\eps_i}(p_i)) = \cup_{j=1}^{2m} \beta_{i,j}
	\]
	for of $2m$ pairwise-disjoint curves $\beta_{i,1},\dots,\beta_{i,2m}$ properly embedded in $\bar B_1(0) \setminus B_{R\eps_i}(p_i)$ that additionally satisfy, for each $j = 1, \ldots, 2m$,
	\begin{equation} \label{eq:induct.SG.lim.thm.B}
		\limsup_{i\to\infty} \max_{s \in [0,T_{i,j}]} |\beta'_{i,j}(s) - v_j| \leq \mu.
	\end{equation}
	when parametrized by unit speed with $\beta_{i,j}(0) \in \partial B_{R\eps_i}(p_i)$, $\beta_{i,j}(T_{i,j}) \in \partial B_1(0)$.
\end{claim}
\begin{proof}
As in Case (A), Proposition \ref{prop:WW.est} implies, for $i$ sufficiently large and all $y_i \in u_i^{-1}(0)\cap (B_{1}(0)\setminus \bar B_{R\eps_i}(p_i))$ that $\nabla u_i(y_i) \neq 0$ and
\begin{equation} \label{eq:induct.SG.lim.thm.B.curv}
	|\cA_{u_i}|(y_i) \leq C  \eps_i^\theta D_i(y_i)^{-1-\theta}
\end{equation}
where
\[ D_i(y_i) := \dist_{g_i}(y_i, \bar B_{R \eps_i/2}(p_i)).  \]
Indeed, condition (B) implies that  $R \eps_i/2 \geq r_i$ as long as $R$ and $i$ are sufficiently large, so rescaling around $y_i$ with $D_i(y_i)$ gives a new phase transition with 
\[
\tilde\eps_i = D_i(y_i)^{-1} \eps_i \leq (R \eps_i/2)^{-1} \eps_i,
\]
so by taking $R$ even larger if necessary, Proposition \ref{prop:WW.est} is applicable and yields \eqref{eq:induct.SG.lim.thm.B.curv} as before.

We can argue as in case (A) to prove the remaining assertions except, at this scale, we only get
\[
|\beta_{i,j}'(t) - v_j| \leq C \eps_i^{\theta}(R \eps_i/2)^{-\theta} = C R^{-\theta} \text{ for all } t \in [0, T_{i,j}].
\]
This can be made arbitrarily small by choosing $R$ large, so the remaining assertions in the claim follow easily. 
\end{proof}
Define $\tilde u_i$ by rescaling by $\eps_i^{-1}$ around $p_i$. Passing to a subsequence, $\tilde u_i$ converges to an entire solution $\tilde u$ to sine-Gordon on $\RR^2$ which is regular at infinity.  As before, the tangent cone at infinity to $\tilde u$ is 
\[
\VarTan(\tilde V, \infty) = \sum_{j=1}^{2m} \bv([0,\infty)v_j,\mathbf{1}_{[0,\infty)v_j})
\]
where the right hand side is as in \eqref{eq:induct.SG.lim.thm.vartan}, \eqref{eq:induct.SG.lim.thm.nolines}. In particular, the tangent cone of $\tilde u$ at infinity is not a union of lines. This contradicts Theorem \ref{theo:SG.cone}, showing that (B) cannot occur either. This completes the proof. 
\end{proof}

\begin{proof}[Proof of Theorem \ref{theo:SG-lim}]
	Consider $(M,g)$ a closed Riemannian $2$-manifold. Assume $\{(u_i,\eps_i)\}_{i=1}^\infty \subset C^\infty(M)\times (0,\infty)$ solve \eqref{eq:ac.pde} with the sine-Gordon potential \eqref{eq:SG-potential} and
\[
\Index_{\eps_i}(u_i) + E_{\eps_i}[u_i] \leq \Lambda. 
\]
By Proposition \ref{prop:HT.theory}, we can pass to a subsequence so that $V_{\eps_i}[u_i]$ converges to a stationary integral $1$-varifold $V$. By \cite[\S 3]{AllardAlmgren:1varifold}, $V$ is a geodesic net. We claim that there are primitive closed geodesics $\sigma_1,\dots,\sigma_N$ so that 
\[
V = \sum_{j=1}^N \bv(\sigma_j,\mathbf{1}_{\sigma_j})
\]
It suffices to prove that for $p \in \supp V$, $\VarTan(V, p)$ is the varifold associated to a union of lines. To this end, for $p\in\supp V$, choose $\lambda_i\to\infty$ sufficiently slowly so that if we dilate $(u_i,\eps_i)$ around $p_i$ by $\lambda_i$ to $(\tilde u_i,\tilde \eps_i)$, it holds that $\lim_i\tilde \eps_i = 0$ and $V_{\tilde \eps_i}[\tilde u_i]$ converges to $\VarTan(V, p)$. Applying Proposition \ref{prop:induct.SG.lim.thm}, we find that $\VarTan(V, p)$ is the varifold associated to a union of lines. This completes the proof. 
\end{proof}

\subsection{Proof of the Liu--Wei tangent cone theorem}\label{subsec:Liu.Wei.proof}

Suppose that $\tilde u$ satisfies \eqref{eq:ac.pde} on $\RR^2$ with the sine-Gordon potential \eqref{eq:SG-potential}, i.e.,
\begin{equation} \label{eq:sg.pde.pretransform}
	\Delta \tilde u = - \pi^{-1}\sin \pi \tilde u.
\end{equation}
It is convenient (in this section only) to adjust the equation slightly so that certain formulas become simpler. Set
\begin{equation} \label{eq:SG.transformation}
	u = \pi (1 + \tilde u)
\end{equation}
so that 
\begin{equation}\label{eq:simplified.SG}
\Delta u = \sin u
\end{equation}
and $|\tilde u|< 1 \iff u \in (0,2\pi)$. We say that $u$ is an entire solution to \eqref{eq:simplified.SG} that is regular at infinity if $\tilde u$ is and we will associate various quantities to $u$ that we previously associated to $\tilde u$, e.g., the asymptotic directions. 

\begin{rema} \label{rema:sg.potential}
	The double-well potential that gives rise to \eqref{eq:simplified.SG} is:
	\begin{equation} \label{eq:sg.potential.simplified}
		W(t) := 1 - \cos t.
	\end{equation}
	Due to the translation we made, the wells are no longer $\pm 1$ but rather $0$ and $2\pi$. In particular, $W(\pi^{-1} t - 1)$ satisfies Definition \ref{defi:ac.potential.general}, but not $W(t)$.
\end{rema}

The goal of this section is to prove Theorem \ref{theo:SG.cone} closely following \cite{LiuWei}, with some changes allowed by our simpler setting. Besides \cite{LiuWei}, a non-comprehensive list of references concerning the inverse scattering method for (elliptic) sine-Gordon include \cite{FT:solitons,Hirota:direct,GL:bvSG,NS:bdStSG,P:SGquarter,PP:elliptic-half,FP:DNmap,FLP:SGsemistrip}.

Before turning to the proof, we need to discuss some preliminary results about the asymptotic behavior of solutions to \eqref{eq:simplified.SG}. 

\subsubsection{The Heteroclinic solution} An important example of a solution to \eqref{eq:simplified.SG} is the rotated heteroclinic $\HH(px+qy+\eta)$ for $|p|^2+|q|^2=1$, where
\begin{equation} \label{eq:sg.heteroclinic}
	\HH(t) = 4 \arctan (e^t)
\end{equation}
is the standard heteroclinic solution \eqref{eq:ac.heteroclinic} with the sine-Gordon potential \eqref{eq:SG-potential} under  the transformation \eqref{eq:SG.transformation}. Recall we will encounter this whenever we fall under case (1) of Proposition \ref{prop:asymp.dir.behavior}.

\subsubsection{The Lax pair} 
In this section, we show that \eqref{eq:simplified.SG} is the compatibility equation for a certain system of ODE's on $\RR^2$. The discussion here closely follows \cite[\S 5]{LiuWei} with some added details. 

Let $\sigma_j \in M(\CC, 2)$ denote the Pauli spin matrices
\[
\sigma_1 = \left( \begin{matrix} 0 & 1 \\ 1 & 0 \end{matrix} \right), \qquad \sigma_2 = \left( \begin{matrix} 0 & -i \\ i & 0 \end{matrix} \right), \qquad \sigma_3 = \left( \begin{matrix} 1 & 0 \\ 0 & -1 \end{matrix} \right). 
\]
Note that we have the usual spin formulas
\[ \sigma_1^2 = \sigma_2^2 = \sigma_3^2 = \Id, \]
\[ \sigma_1\sigma_2 = i\sigma_3 = - \sigma_2\sigma_1, \; \sigma_2\sigma_3 = i \sigma_1 = -\sigma_3\sigma_2, \; \sigma_3\sigma_1 = i\sigma_2 = - \sigma_1\sigma_3, \]
and $\su(2)$ is generated by $i\sigma_1,i\sigma_2,i\sigma_3$. It will be convenient to denote:
\begin{equation} \label{eq:SG.KL}
	K(\lambda) := \lambda - \lambda^{-1}, \; J(\lambda) := \lambda + \lambda^{-1}
\end{equation}
For $u \in C^{\infty}_{\textrm{loc}}(\RR^{2})$, $\lambda \in \CC \setminus \{0\}$, define $A,B \in C^{\infty}_{\textrm{loc}}(\RR^{2};M(\CC,2))$ by
\begin{align}
	A 
		& = \frac{i}{4} \left( \left( \lambda - \frac{\cos u}{\lambda} \right) \sigma_3 - (\partial_xu - i \partial_y u) \sigma_2 - \frac{\sin u}{\lambda} \sigma_1 \right) \nonumber \\
		&  = \frac{i}{4} \left( (\lambda^{-1} W(u) + K(\lambda)) \sigma_3 - (\partial_x u - i \partial_y u) \sigma_2 - \lambda^{-1} W'(u) \sigma_1 \right), \label{eq:lax.pair.A} \\
	B
		& = \frac{1}{4} \left( -\left( \lambda + \frac{\cos u}{\lambda} \right) \sigma_3 + (\partial_xu - i \partial_y u) \sigma_2 - \frac{\sin u}{\lambda} \sigma_1 \right) \nonumber \\
		&  = \frac{1}{4} \left( (\lambda^{-1} W(u) - J(\lambda)) \sigma_3 + (\partial_x u - i \partial_y u) \sigma_2 - \lambda^{-1} W'(u) \sigma_1 \right), \label{eq:lax.pair.B}
\end{align}
where $W$ is as in \eqref{eq:sg.potential.simplified}. Given $A,B$ we consider the system of ODE's 
\begin{equation}\label{eq:lax.pair.A.B}
\left\{
\begin{gathered}
	\partial_{x} \Phi  = A \Phi \\ 
	\partial_{y} \Phi  = B \Phi
\end{gathered}
\right\}
\end{equation}
for $\Phi \in C_\textrm{loc}^{\infty}(\RR^{2};\CC^{2})$. A crucial observation, as we now explain, is that these equations encode \eqref{eq:simplified.SG}. The compatibility of the two equations for $\Phi$ (i.e., consistency with $\partial^2_{x,y} \Phi = \partial^2_{y,x} \Phi$) requires that
\[
(\partial_y A) \Phi + AB\Phi = (\partial_x B) \Phi + B A \Phi.
\]
As such, for the equations for $\Phi$ to be compatible, it must hold that
\[
\partial_y A - \partial_x B = [B,A]. 
\]
The fundamental link with \eqref{eq:simplified.SG} is contained in the following straightforward calculation.

\begin{lemm}\label{lemm:Lax-pair-compat}
	For $\lambda \in \CC\setminus\{0\}$ and $A,B$ as in \eqref{eq:lax.pair.A}, \eqref{eq:lax.pair.B},
	\[ \partial_y A - \partial_x B = [B,A] \iff \Delta u = \sin u. \]
\end{lemm}

It is useful to reformulate the Lax pair equations in a more invariant manner as follows. For $A,B\in C^\infty_{\textrm{loc}}(\RR^2;M(\CC,2))$ we can define a connection $\nabla$ on the trivial $2$-dimensional complex vector bundle $E\to \RR^2$ by
\begin{equation}\label{eq:lax.connection}
 \nabla \Phi = d\Phi - (A dx + Bdy) \wedge \Phi
\end{equation}
for $\Phi \in \Gamma(E) = C^\infty_\textrm{loc}(\RR^2;\CC^2)$. Note that $\nabla \Phi = 0$ is equivalent to \eqref{eq:lax.pair.A.B}. 
\begin{rema}
It is interesting to observe that $\nabla$ is a flat connection if and only if the compatibility conditions are satisfied, but we will not need this fact here. 
\end{rema}
\begin{defi}
	We say that $\Phi \in \Gamma(E)$ is \emph{parallel with respect to (the connection defined by data) $(u,\lambda) \in C^\infty_\textrm{loc}(\RR^2)\times (\CC\setminus \{0\})$} if $\nabla \Phi \equiv 0$ with respect to the connection $\nabla$ defined via  \eqref{eq:lax.connection} and $A$, $B$ as in \eqref{eq:lax.pair.A}, \eqref{eq:lax.pair.B}.
\end{defi}

We will use the following rotation lemma several times in the sequel. 
\begin{lemm}[Rotating parallel sections] \label{lemm:rotate.connection.lax.pair}
	Suppose that $\Phi \in \Gamma(E)$ is parallel with respect to $(u,\lambda) \in C^\infty_{\textrm{loc}}(\RR^2) \times (\CC\setminus\{0\})$. For $\theta \in \RR$ define $\phi : \CC\to\CC$ by $\phi(z) = e^{i\theta}z$. Then, $\phi^* \Phi$ is parallel with respect to $(\phi^*u,e^{i\theta}\lambda)$. 
\end{lemm}
\begin{proof} 
In terms of complex coordinates $z=x+iy$, $\bar z = x-iy$ on $\RR^2 = \CC$,
\begin{align} \label{eq:lax.pair.complex.coord}
	& A dx + Bdy \\
	&   = \frac i 4 \left( \lambda \sigma_3 dz - \lambda^{-1} \sigma_3 d\bar z + \lambda^{-1} W(u) \sigma_3 d\bar z - 2 (\partial_z u) \sigma_2 dz - \lambda^{-1} W'(u) \sigma_1 d\bar z\right) \nonumber 
\end{align}
where $dz=dx+idy$, $d\bar z=dx-idy$, $\partial_z = \tfrac 12 (\partial_x-i\partial_y)$. Thus, 
\begin{multline*}
	0 = \nabla \Phi = d\Phi - \frac i 4 \Big( \lambda \sigma_3 dz - \lambda^{-1} \sigma_3 d\bar z + \lambda^{-1} W(u) \sigma_3 d\bar z \\
	- 2 (\partial_z u) \sigma_2 dz - \lambda^{-1} W'(u) \sigma_1 d\bar z\Big) \wedge \Phi = 0.
\end{multline*}
Pulling this back by $\phi$ we obtain
\begin{multline*}
	d\phi^*\Phi - \frac i 4 \Big( \lambda \sigma_3 e^{i\theta} dz - \lambda^{-1} \sigma_3 e^{-i\theta} d\bar z + W(\phi^* u) \sigma_3 e^{-i\theta} d\bar z \\
	- 2 (\partial_z (\phi^*u)) \sigma_2 dz - \lambda^{-1} W'(\phi^* u)  \sigma_1 e^{-i\theta} d\bar z\Big) \wedge \phi^* \Phi = 0.
\end{multline*}
This completes the proof. 
\end{proof}

\subsubsection{Parallel sections with respect to the trivial solution} We now consider $u\equiv 0$ (or $u\equiv 2\pi$) and study the equation $\nabla \Phi =0$. This will correspond to the asymptotic behavior of $\Phi$ away from the ends of an entire solution $u$ that is regular at infinity. In this case, we clearly have
\begin{equation} \label{eq:SG.AB.0}
A^0 = \frac i 4 K(\lambda) \sigma_3,  \;  B^0 = - \frac 1 4 J(\lambda) \sigma_3
\end{equation}
for the trivial solution's connection coefficients $A^0$, $B^0$. (The superscript $0$ indicates we've set $u \equiv 0$.) Note, now, that
\begin{equation*} 
	\Phi^0(x, y, \lambda) := \exp((\tfrac i 4(K(\lambda)x-\tfrac 1 4 J(\lambda)y) \sigma_3)
\end{equation*}
i.e.,
\begin{align}
	\Phi^0(x,y,\lambda) 
		& = \begin{pmatrix} e^{\frac i 4 K(\lambda) x - \frac 1 4 J(\lambda)y} & 0 \\ 0 &  e^{-\frac i 4 K(\lambda) x + \frac 1 4 J(\lambda)y} \end{pmatrix} \label{eq:defn.jost.trivial.u} \\ 
		& =: \begin{pmatrix} \vert & \vert \\
 				\Phi^0_{+,1}(x,y,\lambda) & \Phi^0_{-,2}(x,y,\lambda) \\
 				\vert & \vert \end{pmatrix} \nonumber
\end{align}
is the unique matrix valued solution to $\nabla \Phi^0 =0$ (i.e., its columns $\Phi^0_{+,1}$, $\Phi^0_{-,2}$ are parallel with respect to $(0, \lambda)$) normalized so that $\Phi^0(0,0) = \Id$. Note that this choice of solutions satisfies:
\begin{align}
	\lim_{x\to +\infty} e^{\frac{-iK(\lambda)}{4}x} \Phi^0_{+,1}(x,y) 
		& = e^{-\frac{J(\lambda)}{4}y}\left( \begin{matrix} 1 \\ 0 \end{matrix}\right), \label{eq:defn.jost.trivial.u.p1} \\
	\lim_{x\to-\infty} e^{\frac{iK(\lambda)}{4}x} \Phi^0_{-,2}(x,y) 
		& = e^{\frac{J(\lambda)}{4}y}\left( \begin{matrix} 0 \\ 1 \end{matrix}\right), \label{eq:defn.jost.trivial.u.m2}
\end{align}
which is also how the subscripts $_{+,1}$, $_{-,2}$ are meant to be interpreted.

\begin{rema} \label{rema:SG.KL.S1}
	We could restrict to $\lambda \in \SS^1 \subset \CC \setminus \{0\}$ in all that follows, in which case $K$ and $J$ in \eqref{eq:SG.KL} would take the more geometric form
	\begin{equation} \label{eq:SG.KL.S1}
		K(\lambda) = \lambda - \bar \lambda = 2 i \operatorname{Im} \lambda, \; J(\lambda) = \lambda + \bar \lambda = 2 \operatorname{Re} \lambda, \; \lambda \in \SS^1.
	\end{equation}
	For $\lambda = q + ip \in \SS^1$, 
	\begin{equation} \label{eq:defn.jost.trivial.u.S1}
		\Phi^0(x,y,q+ip) = \begin{pmatrix} e^{-(px + qy)/2} & 0 \\ 0 & e^{(px + qy)/2} \end{pmatrix}.
	\end{equation}
	These formulas are geometrically simpler and worth keeping in mind. We will adhere to the Liu--Wei approach of working on $\CC \setminus \{0\} \supset \SS^1$ in order to prevent any confusion in the process of referring to their work. (Liu--Wei need to work on $\CC \setminus \{0\} \supset \SS^1$ for  other results.) 
\end{rema}

\subsubsection{Parallel sections for the heteroclinic solution} 

In the previous section we found solutions to $\nabla \Phi = 0$ for the trivial solution $u\equiv 0$. In this section we instead consider $u(x,y) = \HH(x) = 4\arctan e^x$. Note that the asymptotic directions of $u$ are $v_1=(0,1),v_2=(0,-1)$.

In the sequel it will be important that we have a linearly independent set of solutions to $\nabla \Phi =0$ with respect to the data $(\HH(x),\lambda)$. Use $u(x,y) = \HH(x)$ to define the connection coefficients $A,B$ via \eqref{eq:lax.pair.A}, \eqref{eq:lax.pair.B}, and note that they are related to the trivial solution's connection coefficients in \eqref{eq:SG.AB.0} which we now denote $A^0, B^0$, by:
\begin{equation} \label{eq:SG.AB.H}
	A^{\HH} = A^0 + \frac i 4 \Delta^0_{\HH}, \; B^{\HH} = B^0 + \frac 1 4 \Delta^0_{\HH},
\end{equation}
where
\begin{align} \label{eq:SG.AB.H.Delta}
	\Delta^0_{\HH} 
		& := \lambda^{-1} W(\HH) \sigma_3 - \HH' \sigma_2 - \lambda^{-1} W'(\HH) \sigma_1 \nonumber \\
		& = 2 (\lambda^{-1} (\sech^2 x) \sigma_3 - (\sech x) \sigma_2 + \lambda^{-1} (\sech x) (\tanh x) \sigma_1).
\end{align}
In the above, we used the easily verifiable facts 
\begin{align*}
	\HH'(x) & = 2 (\sech x), \\
	W(\HH(x)) & = \tfrac12 (\HH'(x))^2 = 2 (\sech x)^2, \\
	W'(\HH(x)) & = \HH'' = - 2 (\sech x)(\tanh x),
\end{align*}
that follow from \eqref{eq:sg.heteroclinic} and \eqref{eq:simplified.SG} on $\HH(x)$, $x \in \RR$. 

By examining some of the expressions given in \cite{LiuWei}, we are able to explicitly write down such a set of solutions. It is possible to (partially) justify the reason for positing the given expressions by using the inverse scattering transform (cf.\ \cite[Lemma 5.9, Step 1]{LiuWei}), but here we will simply take the expressions for granted (with no justification as to their origin) and then check that they do indeed define parallel sections.

We thus declare, for $\lambda \neq -i$,
\begin{align} \label{eq:defn.jost.heteroclinic.u}
	\Phi^{\HH} 
		& := \Phi^0 + \frac{i}{\lambda+i} \left( (\tanh x) \sigma_3 - (\sech x) \sigma_1 - \Id \right) \Phi^0 \\
		& =: \begin{pmatrix} \vert & \vert \\ \Phi^{\HH}_{+,1} & \Phi^{\HH}_{-,2} \\ \vert & \vert \end{pmatrix}. \nonumber
\end{align}
where $\Phi^0$ is as in \eqref{eq:defn.jost.trivial.u}. As the notation suggests:

\begin{lemm} \label{lemm:heteroclinic-jost} 
	For $\lambda \neq -i$, the columns of $\Phi^{\HH}$ are parallel with respect to $(\HH(x), \lambda)$ and their Wronskian is $\det \Phi^{\HH}(x, y, \lambda) = \frac{\lambda-i}{\lambda+i}$ for all $(x, y) \in \RR^2$.
\end{lemm}

Note that this choice of solutions satisfies (cf. \eqref{eq:defn.jost.trivial.u.p1}, \eqref{eq:defn.jost.trivial.u.m2}) for  $\lambda \neq -i$:
\begin{align}
	\lim_{x\to +\infty} e^{\frac{-iK(\lambda)}{4}x} \Phi^{\HH}_{+,1}(x,y, \lambda) 
		& = e^{-\frac{J(\lambda)}{4}y}\left( \begin{matrix} 1 \\ 0 \end{matrix}\right), \label{eq:defn.jost.heteroclinic.u.p1} \\
	\lim_{x\to-\infty} e^{\frac{iK(\lambda)}{4}x} \Phi^{\HH}_{-,2}(x,y, \lambda) 
		& = e^{\frac{J(\lambda)}{4}y}\left( \begin{matrix} 0 \\ 1 \end{matrix}\right). \label{eq:defn.jost.heteroclinic.u.m2}
\end{align}

\begin{proof}
	For obvious reasons we will abbreviate:
	\[ S := \sech x, \; T := \tanh x. \]
	By rewriting the definition of $\Phi^{\HH}$, using the fact that $\det \Phi^0 = 1$ from \eqref{eq:defn.jost.trivial.u}, and the hyperbolic trigonometry identity $S^2 + T^2 = 1$ we have:
	\[ \Phi^{\HH} (\Phi^0)^{-1} = \frac{\lambda}{\lambda+i} \Id + \frac{i}{\lambda+i} \begin{pmatrix} T & -S \\ -S & -T \end{pmatrix} \implies \det \Phi^{\HH} = \frac{\lambda-i}{\lambda+i} \]
	as desired. 
	
	Next, we compute $\partial_x \Phi^{\HH} - A^{\HH} \Phi^{\HH}$ using the fact that $\partial_x \Phi^0 = A^0 \Phi^0$ and the calculus identities $\partial_x S = -ST$, $\partial_x T = S^2$:
	\begin{align*}
		& \partial_x \Phi^{\HH} - A^{\HH} \Phi^{\HH} \\
		& = \partial_x \Big( \Phi^0 + \frac{i}{\lambda+i} (T\sigma_3 - S\sigma_1 - \Id) \Phi^0 \Big) \\
		& \qquad - \Big( A^0 + \frac{i}{4} \Delta^0_{\HH} \Big) \Big( \Phi^0 + \frac{i}{\lambda+i} (T\sigma_3 - S\sigma_1 - \Id) \Phi^0 \Big) \\
		& = A^0 \Phi^0 + \frac{i}{\lambda+i} \partial_x (T\sigma_3 - S\sigma_1 - \Id) \Phi^0 + \frac{i}{\lambda+i} (T\sigma_3 - S \sigma_1 - \Id) A^0 \Phi^0 \\
		& \qquad - \Big( A^0 + \frac{i}{4} \Delta^0_{\HH} \Big) \Big( \Phi^0 + \frac{i}{\lambda+i} (T\sigma_3 - S \sigma_1 - \Id) \Phi^0 \Big) \\
		& = \frac{i}{\lambda+i} (S^2 \sigma_3 + ST \sigma_1) \Phi^0  + \frac{i}{\lambda+i} (T\sigma_3 - S \sigma_1 - \Id) A^0 \Phi^0 - \frac{i}{4} \Delta^0_{\HH} \Phi^0 \\
		& \qquad - \frac{i}{\lambda+i} A^0 (T \sigma_3 - S \sigma_1 - \Id) \Phi^0 + \frac{1}{4(\lambda+i)} \Delta^0_{\HH} (T\sigma_3 - S \sigma_1 - \Id) \Phi^0.
	\end{align*}
	Recalling \eqref{eq:SG.AB.0} and the spinor formulas we find that
	\[ A^0 \sigma_1 - \sigma_1 A^0 = \frac{i}{2} K(\lambda) \sigma_3 \sigma_1 = -\frac{1}{2} K(\lambda) \sigma_2, \]
	while the $A^0 \sigma_3, \sigma_3 A^0$ and $A^0 \Id, \Id A^0$ cancel out (the former using \eqref{eq:SG.AB.0}, the latter trivially). Using \eqref{eq:SG.AB.H.Delta} and our expressions for $W(\HH)$, $\HH'$, $W'(\HH)$:
	\begin{align*}
		& \partial_x \Phi^{\HH} - A^{\HH} \Phi^{\HH} \\
		& = \Big[ \frac{i}{\lambda+i} (S^2 \sigma_3 + ST \sigma_1) - \frac{i}{2(\lambda+i)} S K(\lambda) \sigma_2 - \frac{i}{2} (\lambda^{-1} S^2 \sigma_3 - S \sigma_2 + \lambda^{-1} ST \sigma_1) \\
		& \qquad + \frac{1}{2(\lambda+i)} (\lambda^{-1} S^2 \sigma_3 - S \sigma_2 + \lambda^{-1} ST \sigma_1)(T\sigma_3 - S\sigma_1 - \Id) \Big] \Phi^0.
	\end{align*}
	We show the expression inside the bracket vanishes by collecting terms according to $\sigma_1$, $\sigma_2$, $\sigma_3$, $\Id$, using the spinor formulas. The coefficients are:
	\begin{align*}
		\Big( \cdots \Big)_{\sigma_1} 
			& = \frac{i}{\lambda+i} ST - \frac{i}{2\lambda} ST - \frac{i}{2(\lambda+i)} ST - \frac{1}{2\lambda(\lambda+i)} ST \\
			& = \frac{i}{2} ST \Big( \frac{1}{\lambda+i} - \frac{1}{\lambda} + \frac{i}{\lambda(\lambda+i)} \Big) = 0, \\
		\Big( \cdots \Big)_{\sigma_2} 
			& = - \frac{i}{2(\lambda+i)} SK(\lambda) + \frac{i}{2} S - \frac{i}{2\lambda(\lambda+i)} S^3 + \frac{1}{2(\lambda+i)} S \\
			& \qquad - \frac{i}{2 \lambda(\lambda+i)} ST^2 \\
			& = \frac{i}{2} S \Big( - \frac{\lambda}{\lambda+i} + \frac{1}{\lambda(\lambda+i)} + 1 - \frac{1}{\lambda(\lambda+i)} - \frac{i}{\lambda+i} \Big) = 0, \\
		\Big( \cdots \Big)_{\sigma_3} 
			& = \frac{i}{\lambda+i} S^2 - \frac{i}{2\lambda} S^2 - \frac{1}{2\lambda(\lambda+i)} S^2 - \frac{i}{2(\lambda+i)} S^2 \\
			& = \frac{i}{2} S^2 \Big( \frac{2}{\lambda+i} - \frac{1}{\lambda} + \frac{i}{\lambda(\lambda+i)} - \frac{1}{\lambda+i} \Big) = 0, \\
		\Big( \cdots \Big)_{\Id}
			& = \frac{1}{2\lambda (\lambda+i)} S^2 T - \frac{1}{2\lambda (\lambda+i)} S^2 T = 0.
	\end{align*}
	Thus, $\partial_x \Phi^{\HH} = A^{\HH} \Phi^{\HH}$. The computation for $\partial_y \Phi^{\HH} = B^{\HH} \Phi^{\HH}$ is analogous and easier, since $\HH(x)$ is independent of $y$.
\end{proof}

Note that when $\lambda=i$, the columns of our matrix-valued solution are linearly dependent (though neither is identically zero) and thus do not span the ODE solution space. Nonetheless, they do span the space of solutions that will be of interest to us:

\begin{coro} \label{coro:one.bound.state}
	If $\Phi \in C^\infty_{\textrm{loc}}(\RR^2, \CC^2)$ is parallel with respect to $(\HH(x), i)$ and either 
	\[ e^{-x/2} \Phi(x,y) \in L^\infty(\RR^2) \text{ or } e^{x/2} \Phi(x, y) \in L^\infty(\RR^2), \]
	then $\Phi$ is a multiple of the (linearly dependent) columns of $\Phi^{\HH}(x, y, i)$.
\end{coro}
\begin{proof}
	Suppose that $e^{x/2} \Phi(x, y, i) \in L^\infty(\RR^2)$. It follows from the tracelessness of $A^{\HH}$, $B^{\HH}$ (see \eqref{eq:SG.AB.0}, \eqref{eq:SG.AB.H}, \eqref{eq:SG.AB.H.Delta}) that the Wronskian
	\[ (x, y) \mapsto \det (\Phi, \Phi^{\HH}_{+,1})(x, y, i) \]
	is independent of $(x, y) \in \RR^2$. (See, e.g., Lemma \ref{lemm:jost.wronskian}.) However,
	\[ e^{x/2} \Phi(x, y) \in L^\infty(\RR^2) \implies |\Phi(x, y)| \leq C e^{-x/2} \text{ for all } (x, y) \in \RR^2, \]
	while \eqref{eq:defn.jost.heteroclinic.u.p1} and $K(i) = 2i$ imply
	\[ \lim_{x \to \infty} e^{x/2} \Phi^{\HH}_{+,1}(x, y, i) = e^{-iy/2} \begin{pmatrix} 1 \\ 0 \end{pmatrix}. \]
	In particular, holding $y$ fixed and sending $x \to \infty$ shows $\det(\Phi, \Phi^{\HH}_{+,1}) = 0$, so $\Phi$ is a multiple of $\Phi^{\HH}_{+,1}$. The argument is analogous when $e^{-x/2} \Phi \in L^\infty(\RR^2)$ instead, except we use $\Phi^{\HH}_{-,2}$ rather than $\Phi^{\HH}_{+,1}$.
\end{proof}

We need one final piece of notation:
\begin{align}
\mathring H_{\pm} & : = \{\lambda \in \CC : \pm \imag \lambda >0\} \label{eq:mathring.Hpm}\\
\check H_{\pm} & : = \{\lambda \in \CC \setminus\{0\} : \pm \imag \lambda \geq 0\}\label{eq:check.Hpm}
\end{align}

\begin{coro}\label{coro:jost.soln.bds.line.trivial}
Let $\lambda = q + i p \in (\SS^1 \cap \mathring H_+) \setminus \{i\}$. Suppose that $\Phi$ is parallel with respect to $(\HH(x),\lambda)$ and also satisfies
\[
\sup_{(x,y) \in \ell} |\Phi(x,y)| < \infty
\]
where $\ell$ is the line defined by
\[
\ell = \{(x,y) \in \RR^2 : (-q, p) \cdot (x, y) = 0\}. 
\]
Then, $\Phi \equiv 0$. 
\end{coro}
\begin{proof} 
By Lemma \ref{lemm:heteroclinic-jost}, the columns of our $\Phi^{\HH}$ from \eqref{eq:defn.jost.heteroclinic.u} span the ODE solution space when $\lambda \not = i$, so we can write
\begin{equation} \label{eq:jost.soln.bds.line.trivial.phi}
	\Phi = c_{1}\Phi_{+,1}^\HH(\cdot, \cdot, \lambda) + c_2 \Phi_{-,2}^\HH(\cdot, \cdot, \lambda)
\end{equation}
for some $c_1,c_2\in\CC$, where $\Phi^{\HH}_{+,1}$, $\Phi^{\HH}_{-,2}$ are as in \eqref{eq:defn.jost.heteroclinic.u}. Observe that the line $\ell$ can be parametrized by $\gamma(t) = (tp,tq)$ and 
\[
\frac{iK(\lambda)}{4} t p - \frac{J(\lambda)}{4} t q = - \frac t 2,
\]
so
\begin{equation} \label{eq:jost.soln.bds.line.trivial.phi0}
	\Phi^{0}(tp, tq, q+ip) = \begin{pmatrix} e^{-t/2} & 0 \\ 0 & e^{t/2} \end{pmatrix}
\end{equation}
for all $t \in \RR$. Moreover, by inspecting \eqref{eq:defn.jost.heteroclinic.u},
\begin{align*}
	& \Phi^{\HH}(tp, tq, q+ip) (\Phi^0(tp, tq, q+ip))^{-1} \\
	& = \frac{\lambda}{\lambda+i} \Id + \frac{i}{\lambda+i} (\tanh pt) \sigma_3 - \frac{i}{\lambda+i} (\sech pt) \sigma_1 \\
	& \to \frac{\lambda}{\lambda+i} \Id \pm \frac{i}{\lambda+i} \sigma_3 \text{ as } |t| \to \infty,
\end{align*}
so together with \eqref{eq:jost.soln.bds.line.trivial.phi0} this implies
\[ \Phi^{\HH}_{+,1}(tp, tq, q+ip) \to \infty \text{ as } t \to -\infty, \]
\[ \Phi^{\HH}_{-,2}(tp, tq, q+ip) \to \infty \text{ as } t \to +\infty. \]
Now evaluating \eqref{eq:jost.soln.bds.line.trivial.phi} along $\gamma(t)$ with $t \to -\infty$ yields $c_1 = 0$, and $t \to +\infty$ yields $c_2 = 0$. This completes the proof. 
\end{proof}

\subsubsection{Analyzing the flipped heteroclinic solution}  

The following result corresponds to a well-known symmetry/gauge invariance of the Lax pair (cf.\ \cite[Lemma 5.3]{LiuWei}). 

\begin{lemm}\label{lemm:flip.lax.pair.pm}
Suppose that $\lambda \in \check{H}_+$, $\nabla$ is the connection with respect to $(\HH(x),\lambda)$, and $\tilde\nabla$ is the connection with respect to $(-\HH(x),\lambda)$. If $\phi(z) := -z$, then
	\[ \nabla \Phi \equiv 0 \iff \tilde \nabla \tilde \Phi \equiv 0 \text{ for } \tilde\Phi := i \sigma_2 \phi^*\Phi. \]
\end{lemm}
\begin{proof}
We compute, using \eqref{eq:lax.pair.complex.coord}, the evenness of $W$, and oddness of $W'$:
\begin{align*}
	(\phi^{-1})^*\tilde\nabla \tilde\Phi 
	& = (\phi^{-1})^* d\tilde\Phi - \frac i 4 (\phi^{-1})^* \Big( \lambda \sigma_3 dz - \lambda^{-1} \sigma_3 d\bar z + \lambda^{-1} W(-\HH) d \bar z \\
	& \qquad \qquad \qquad \qquad + 2 (\partial_z \HH) \sigma_2 dz - \lambda^{-1} W'(-\HH) \sigma_1 d\bar z \Big)\wedge(\phi^{-1})^* \tilde\Phi\\
	& = i \sigma_2 d\Phi + \frac i 4  \Big(  \lambda \sigma_3 dz - \lambda^{-1} \sigma_3 d\bar z + \lambda^{-1} W(\HH) \sigma_3 d\bar z \\
	& \qquad \qquad \qquad - 2 (\partial_z \HH) \sigma_2 dz - \lambda^{-1} W'(\HH) \sigma_1 d\bar z \Big) \wedge i \sigma_2 \Phi \\
	& = i\sigma_2 \Big( d\Phi + \frac i 4 \Big( \lambda \sigma_2 \sigma_3 \sigma_2 dz - \lambda^{-1} \sigma_2 \sigma_3 \sigma_3 d \bar z + \lambda^{-1} W'(\HH)\sigma_2 \sigma_3 \sigma_2 d\bar z \\
	& \qquad \qquad \qquad - 2 (\partial_z \HH) \sigma_2 dz - \lambda^{-1} W'(\HH)   \sigma_2 \sigma_1\sigma_2 d\bar z \Big)   \wedge  \Phi  \Big)\\
	& = i\sigma_2 \Big( d\Phi - \frac i 4 \Big(  \lambda  \sigma_3  dz - \lambda^{-1} \sigma_3 d\bar z + \lambda^{-1} W'(\HH) \sigma_3 d\bar z \\
	& \qquad \qquad \qquad - 2 (\partial_z \HH) \sigma_2 dz - \lambda^{-1} W'(\HH) \sigma_1  d\bar z \Big) \wedge\Phi \Big)  \\
& = i \sigma_2 \nabla \Phi.
\end{align*}
This completes the proof. 
\end{proof}

This allows us to ``flip'' Corollaries \ref{coro:one.bound.state} and \ref{coro:jost.soln.bds.line.trivial} as follows. 

\begin{coro}\label{coro:one.bound.state.flipped}
If $\Phi \in C^\infty_\textrm{loc}(\RR^2;\CC^2)$ is parallel with respect to $(-\HH(x), i)$ and either
\[ e^{-x/2} \Phi(x,y) \in L^\infty(\RR^2) \text{ or } e^{x/2} \Phi(x,y) \in L^\infty(\RR^2), \]
then $\Phi$ is a multiple of the (linearly dependent) columns of $\sigma_2 \Phi^{\HH}(-x, -y, i)$.
\end{coro}
\begin{proof}
By Lemma \ref{lemm:flip.lax.pair.pm}, $\sigma_2 \phi^* \Phi$ is parallel with respect to $(\HH(x),i)$ provided $\phi(z) = -z$. The result follows from Corollary \ref{coro:one.bound.state}.
\end{proof}

\begin{coro}\label{coro:jost.soln.bds.line.trivial.flipped}
Let $\lambda = q + i p \in \SS^1 \cap \mathring H_+ \setminus \{i\}$. Suppose that $\Phi$ is parallel with respect to $(-\HH(x),\lambda)$ and also satisfies
\[
\sup_{(x,y) \in \ell} |\Phi(x,y)| < \infty
\]
where $\ell$ is the line defined by
\[
\ell = \{(x,y) \in \RR^2 : (-q, p) \cdot (x, y) = 0\}. 
\]
Then, $\Phi \equiv 0$. 
\end{coro}
\begin{proof}
By Lemma \ref{lemm:flip.lax.pair.pm}, $\sigma_2 \phi^* \Phi$ is parallel with respect to $(\HH(x),\lambda)$. Moreover, $\phi$ maps $\ell$ to itself. Thus, 
\[
\sup_{(x,y) \in \ell} |\phi^*\Phi(x,y)| < \infty.
\]
Hence, the assertion follows from Corollary \ref{coro:jost.soln.bds.line.trivial}. 
\end{proof}

\subsubsection{Jost solutions} In the previous sections we found solutions to $\nabla \Phi = 0$ for the trivial solution ($u\equiv 0$) and the heteroclinic solution ($u = \HH(x)$). We now discuss the general case. This section follows the arguments given in \cite[Lemma 5.2]{LiuWei} very closely. 

Consider an entire solution of \eqref{eq:simplified.SG} which is regular at infinity and none of whose asymptotic directions are $\pm (1,0)$. Recall the definition of $\check H_\pm$ in \eqref{eq:check.Hpm}. 

\begin{prop}[Existence of Jost solutions]\label{prop:Jost}
Fix $\lambda \in \check H_+$. There exist $\Phi_{+,1},\Phi_{-,2} \in \Gamma(E)$ that are parallel with respect to $(u,\lambda)$ and so that, for all $y \in \RR$ fixed (cf. \eqref{eq:defn.jost.trivial.u.p1}, \eqref{eq:defn.jost.trivial.u.m2}, \eqref{eq:defn.jost.heteroclinic.u.p1}, \eqref{eq:defn.jost.heteroclinic.u.m2}),
\begin{align*}
	\lim_{x\to +\infty} e^{\frac{-iK(\lambda)}{4}x} \Phi_{+,1}(x,y) & = e^{-\frac{J(\lambda)}{4}y}\left( \begin{matrix} 1 \\ 0 \end{matrix}\right), \\
	\lim_{x\to-\infty} e^{\frac{iK(\lambda)}{4}x} \Phi_{-,2}(x,y) & = e^{\frac{J(\lambda)}{4}y}\left( \begin{matrix} 0 \\ 1 \end{matrix}\right).
\end{align*}
Moreover, 
\[ e^{-\frac{iK(\lambda)}{4}x + \frac{J(\lambda)}{4}y} \Phi_{+,1} \text{ and } e^{\frac{iK(\lambda)}{4} x - \frac{J(\lambda)}{4}y} \Phi_{-,2} \in L^\infty(\RR^2). \]
\end{prop}
\begin{proof}
We consider $\Phi_{-,2}$ since the argument for $\Phi_{+,1}$ is the same. 

We first fix $y=y_0$ and seek a solution to
\[ \partial_x \Phi(x, y_0) = A(x, y_0) \Phi(x, y_0), \; x \in \RR, \] 
\[ \lim_{x\to-\infty} e^{\frac{iK(\lambda)}{4} x} \Phi = (0,1)^T. \]
Write $\Psi(x,y) := e^{\frac{iK(\lambda)}{4}x} \Phi(x,y)$ so in coordinates $(\psi_1, \psi_2) := \Psi$ the evolution $\partial_x\Phi = A \Phi$ becomes 
\[
\begin{cases}
\partial_x \psi_1 = (A_{11} + \tfrac{iK(\lambda)}{4}) \psi_1 + A_{12} \psi_2, \\
\partial_x \psi_2 = A_{21} \psi_1 + (A_{22}+ \tfrac{iK(\lambda)}{4}) \psi_2. 
\end{cases}
\]
Define 
\[
\tilde A(s,y_0) : = A(s,y_0) - \frac{iK(\lambda)}{4} \sigma_3
\]
so that $\tilde A\to 0$ exponentially fast away from the ends of $u$; cf.\ \eqref{eq:exp.decay.entire}. We can thus rewrite the equations for $(\psi_1,\psi_2)$ as
\[
\begin{cases}
\partial_x (e^{-\frac{iK(\lambda)}{2}x} \psi_1) = e^{-\frac{iK(\lambda)}{2}x} (\tilde A_{11}  \psi_1 + \tilde A_{12} \psi_2), \\
\partial_x \psi_2 = \tilde A_{21} \psi_1 + \tilde A_{22}  \psi_2.
\end{cases}
\]
These equations, along with the limiting assumption are equivalent to the integral equations (where the dependence of $\tilde A_{ij}$ on $y_0$ is suppressed):
\[
\begin{cases}
\psi_1(x) = \int_{-\infty}^x \exp(\tfrac{iK(\lambda)}{2} (x-s)) (\tilde A_{11}  \psi_1 + \tilde A_{12} \psi_2)(s) \,  ds\\
\psi_2(x) = 1 + \int_{-\infty}^x   (\tilde A_{21}  \psi_1 + \tilde A_{22} \psi_2)(s) \,  ds. 
\end{cases}
\]
We now solve this integral equation with Picard iteration using, crucially, that $\re(iK(\lambda)) \leq 0$ when $\lambda \in \check{H}_+$. Our estimates will involve
\[
Q(x,y) : = \sum_{i,j=1}^2 \int_{-\infty}^x |\tilde A_{ij}(s,y)| ds, \; Q_* := \Vert Q \Vert_{L^\infty(\RR^2)}.
\]
(Note that $Q \in L^\infty(\RR^2)$ thanks to \eqref{eq:exp.decay.entire}.)

We start our Picard iteration with $(\psi^{(0)}_1,\psi^{(0)}_2) = (0,1)$. Inductively set
\[
\begin{cases}
\psi_1^{(n)}(x) = \int_{-\infty}^x \exp(\tfrac{iK(\lambda)}{2} (x-s)) (\tilde A_{11}  \psi_1^{(n-1)} + \tilde A_{12} \psi_2^{(n-1)})(s) \,  ds\\
\psi_2^{(n)}(x) = 1 + \int_{-\infty}^x   (\tilde A_{21}  \psi_1^{(n-1)} + \tilde A_{22} \psi_2^{(n-1)})(s) \,  ds.
\end{cases}
\]
For example, we have
\[
\begin{cases}
\psi_1^{(1)}(x) = \int_{-\infty}^x \exp(\tfrac{iK(\lambda)}{2} (x-s))   \tilde A_{12} (s) \,  ds\\
\psi_2^{(1)}(x) = 1 + \int_{-\infty}^x    \tilde A_{22}(s)  \,  ds.
\end{cases}
\]
We claim that, for every $x \in \RR$,
\[
| \psi_1^{(n)}(x) - \psi_1^{(n-1)}(x) |, \; | \psi_2^{(n)}(x) - \psi_2^{(n-1)}(x) | \leq \frac{Q(x, y_0)^{n}}{n!}. 
\]
This clearly holds when $n=1$ since we can bound (using $\re(iK(\lambda)) \leq 0$)
\[
|\psi_1^{(1)}(x)|, \; |\psi_2^{(1)}(x) - 1| \leq Q(x, y_0).
\]
In general, this follows inductively: 
\begin{align*}
|\psi_1^{(n)}(x) - \psi_1^{(n-1)}(x)| & \leq \int_{-\infty}^x  (|\tilde A_{11}|  + |\tilde A_{12}|)(s) \frac{Q(s, y_0)^{n-1}}{(n-1)!}  \,  ds\\
& \leq  \frac{1}{(n-1)!} \int_{-\infty}^x (\tfrac{d}{ds} Q(s, y_0)) Q(s, y_0)^{n-1}  \,  ds\\
& =  \frac{1}{n!} \int_{-\infty}^x \frac{d}{ds} Q(s, y_0)^{n}  \,  ds = \frac{Q(x, y_0)^n}{n!}.
\end{align*}
The estimate for $\psi_2$ is identical. Using this, we find that $(\psi_1^{(n)},\psi_2^{(n)})$ converges uniformly to $(\psi_1,\psi_2)$ solving the integral equation (and thus the differential equation).  We also note that we have established the bounds
\begin{equation}\label{eq:jost.exist.psi.Q.bds}
|\psi_1(x)| \leq \exp(Q(x)) - 1, \; |\psi_2(x)|\leq \exp(Q(x)),
\end{equation}
so combining the integral equation with the given bounds shows that 
\[
\lim_{x\to-\infty} (\psi_1(x),\psi_2(x)) = (0,1). 
\]
Thus, $\Phi = e^{-\frac{iK(\lambda)}{4}x} (\psi_1,\psi_2)^T$ solves $\partial_x \Phi = A \Phi$ (for $y=y_0$ fixed) with 
\[
\lim_{x\to-\infty} e^{\frac{iK(\lambda)}{4}x} \Phi(x) = (0,1)^T
\]
and satisfies $|\Phi(x)| \leq e^{Q_*} \exp( - \tfrac 14  \re(iK(\lambda)) x)$.

It is standard to show that when allowing $y_0$ to vary, we obtain $\Phi(x,y) \in C^\infty_{\textrm{loc}}(\RR^2;\CC^2)$ solving $\partial_x \Phi = A\Phi$ and satisfying the same bound 
\[
|\Phi(x,y)| \leq e^{Q_*} \exp( - \tfrac 14  \re(iK(\lambda)) x).
\]
By differentiating the integral equation for $(\psi_1,\psi_2)$, we have that 
\[
\lim_{x\to-\infty} e^{\frac{iK(\lambda)}{4} x} \partial_y \Phi(x,y) = (0,0)^T.
\]
In particular (since $\re(iK(\lambda))\leq 0$) 
\[
\lim_{x\to- \infty} e^{\frac{iK(\lambda)}{4} x} (\partial_y \Phi(x,y) - B\Phi(x,y)) = - \tfrac{J(\lambda)}{4} (0,-1)^T. 
\]
Note that $\Phi(x,y)$ does \textit{not} solve the $y$-equation, but we claim that 
\[
\Phi_{-,2}(x,y) := e^{\frac{J(\lambda)}{4} y} \Phi(x,y)
\]
does. Indeed, using $\partial_x \Phi = A\Phi$, we have
\begin{multline}\label{eq:y-equation-from-x-equation-and-infty}
\partial_x(\partial_y \Phi_{-,2} - B \Phi_{-,2}) \\
= A (\partial_y\Phi_{-,2} - B\Phi_{-,2}) + \underbrace{(\partial_y A-\partial_x B + AB -BA)}_{=0}\Phi_{-,2} 
\end{multline}
by Lemma \ref{lemm:Lax-pair-compat}. Moreover, by the limiting behavior of $\Phi$, we find that
\[
\lim_{x\to -\infty} e^{\frac{iK(\lambda)}{4}x } (\partial_y \Phi_{-,2} - B \Phi_{-,2}) = (0,0)^T. 
\]
Using this in conjunction with \eqref{eq:y-equation-from-x-equation-and-infty}, we can argue similarly to the construction of $(\psi_1,\psi_2)$ to conclude that $\partial_y \Phi_{-,2} - B\Phi_{-,2}$ vanishes identically. It also satisfies $\partial_x \Phi_{-,2} - A \Phi_{-,2} \equiv 0$ too, since $\Phi$ does. Thus $\nabla \Phi_{-,2} \equiv 0$ as claimed. This completes the proof.
\end{proof}

We have the following important (but simple) result.
\begin{lemm} \label{lemm:jost.wronskian}
The Wronskian
\[ \det (\Phi_{+,1}(x,y,\lambda),\Phi_{-,2}(x,y,\lambda)) \]
is independent of $(x,y)$.
\end{lemm}
\begin{proof}
It follows from Jacobi's formula for the derivative of the determinant, and \eqref{eq:lax.pair.A.B} that
\begin{align*}
\frac{\partial}{\partial x} \det (\Phi_{+,1}(x,y,\lambda),\Phi_{-,2}(x,y,\lambda)) & = (\tr A )\det (\Phi_{+,1}(x,y,\lambda),\Phi_{-,2}(x,y,\lambda))\\
\frac{\partial}{\partial y} \det (\Phi_{+,1}(x,y,\lambda),\Phi_{-,2}(x,y,\lambda)) & = (\tr B  )\det (\Phi_{+,1}(x,y,\lambda),\Phi_{-,2}(x,y,\lambda))
\end{align*}
and is easy to check that $\tr A = \tr B =0$.
\end{proof}

We thus define the $(x,y)$-independent quantity
\[ a(\lambda) : = \det (\Phi_{+,1}(x,y,\lambda),\Phi_{-,2}(x,y,\lambda)), \; \lambda \in \check{H}_+. \]
The function $a(\lambda)$ will play a key role in the proof of Theorem \ref{theo:SG.cone}. We will be able to use $a(\lambda)$ to relate the behavior of Jost solutions for an arbitrary entire solution that is regular at infinity to the behavior of the Jost solutions far out along an end of the solution, since $a(\lambda)$ is independent of $(x,y)$. (Note that $\Phi_{+,1}, \Phi_{-,2}$ do depend on $(x,y)$, in a nontrivial way.)

\subsubsection{Ends correspond to bound states}

This section is inspired by the proof of \cite[Lemmas 5.8, 5.9]{LiuWei}, but arguments are somewhat different, if only superficially.

We continue to assume, as in the previous section, that $u$ is an entire solution of \eqref{eq:simplified.SG} which is regular at infinity, and none of whose asymptotic directions $\{v_1,\dots,v_{2m}\}$ are $\pm (1,0)$. Using the standard identification of $\RR^2$ and $\CC$, we will consider the asymptotic directions as lying in $\SS^1\subset \CC$. 

Recall that for $\lambda \in \check{H}_+$ the Jost solutions $\Phi_{-,2}$ and $\Phi_{+,1}$ exist and satisfy the conclusions of Proposition \ref{prop:Jost}. Recall also that $a(\lambda) : = \det(\Phi_{-,2},\Phi_{+,1})$ is independent of $(x,y) \in \RR^2$. In this section we will identify the asymptotic directions of $u$ with solutions to $a(\lambda) = 0$ in a manner that will allow us to prove Theorem \ref{theo:SG.cone}. 

We first make several useful definitions. Define the set 
\[
\fB : = \{ \lambda \in \SS^1 \cap \mathring H_+ : a(\lambda) = 0\}.
\]
Define a function 
\[
\sigma : \{1,\dots,2m\} \to \{\pm 1\} \text{ so that } \sigma(j) v_{j} \in \SS^1 \cap \mathring H_+ \text{ for all } j \in \{ 1, \ldots, 2m \}.
\]
Also define
\[ R : \CC \to\CC, \; R(x+iy) = -x+iy \]
to be the reflection across the $y$-axis. (Note that this reflection already appeared implicitly in Corollaries \ref{coro:jost.soln.bds.line.trivial}, \ref{coro:jost.soln.bds.line.trivial.flipped}.) The following two propositions summarize the main ingredients needed to prove Theorem \ref{theo:SG.cone}.

\begin{prop}\label{prop:ends.to.bd.state}
$R(\{\sigma(1)v_1, \ldots, \sigma(2m) v_{2m} \}) \subset \fB$. 
\end{prop}

\begin{prop}\label{prop:bd.state.to.end}
$R(\{ \pm \lambda : \lambda \in \fB \}) \subset \{ v_1,\dots,v_{2m}\}$. 
\end{prop}

We prove these below, but first we observe how Theorem \ref{theo:SG.cone} follows from them.

\begin{proof}[Proof of Theorem \ref{theo:SG.cone}]
For an asymptotic direction $v_j$, 
\[ R(\sigma(j)v_j) \in \fB \]
by Proposition \ref{prop:ends.to.bd.state}. Proposition \ref{prop:bd.state.to.end} implies that
\[ -v_j = R(  -\sigma(j) R(\sigma(j)v_j)) \]
is an asymptotic direction. Since the asymptotic directions $v_1,\dots,v_{2m}$ are distinct, this completes the proof. 
\end{proof}
\begin{proof}[Proof of Proposition \ref{prop:ends.to.bd.state}]
Fix $j \in \{1,\dots,2m\}$ and consider $\sigma(j)v_j = (-q,p)$ for $p>0$. Denote $\lambda : = q + ip \in \SS^1 \cap \mathring H_+$.

\begin{claim}
	There exist $(x_k,y_k) \in u^{-1}(\pi)$ with $(x_k,y_k) \to \infty$ and
	\[
		u(x + x_k, y + y_k) \to \pm \HH(px+qy) \text{ in } C^\infty_\textrm{loc}(\RR^2)
	\] 
	as $k \to \infty$.
\end{claim}
\begin{proof}[Proof of claim]
 To find such a sequence, we note that Propositions \ref{prop:HT.theory} and \ref{prop:tangent.cone.infty.AC} yield $(x_k,y_k) \in u^{-1}(\pi)$ with 
\[
\frac{(x_k,y_k)}{|(x_k,y_k)|} \to v_j \text{ as } k \to \infty.
\]
Then, thanks to Proposition \ref{prop:asymp.dir.behavior} (see also Remark \ref{rema:sg.potential}), we know that up to passing to a subsequence, $u(x+x_k,y+y_k)$ limits to a function $u_\infty$ that is either (1) $u_\infty \equiv 0,2\pi$ or (2) $u_\infty(x, y) = \HH(p_\infty x + q_\infty y + \eta_\infty)$ for $(p_\infty,q_\infty)\in\SS^1,\eta_\infty\in\RR$. Since the limit has $u_\infty(0,0) = \pi$, we must be in case (2) and $\eta_\infty = 0$. By Proposition \ref{prop:asymp.dir.behavior} again, we have $(p_\infty,q_\infty) = \pm(p,q)$. 
\end{proof}

In what follows, we will assume that
\[ \lim_k u(x+x_k,y+y_k) = \HH(px+qy) \]
and will explain the modifications when the limit is $-\HH(px+qy)$ at the end of the proof. By \eqref{eq:ends.deviate.asymptot.dir.bd.dist}, we can assume that
\[
px_k + qy_k \to -2\kappa \text{ as } k \to \infty,
\]
for some $\kappa \in \RR$. Then,
\begin{equation}\label{eq:shift.jost.end.eta.limit}
\frac{iK(\lambda)}{4} x_k - \frac{J(\lambda)}{4} y_k \to \kappa \text{ as } k \to \infty.
\end{equation}
Now consider the Jost solutions $\Phi_{+,1},\Phi_{-,2}$ for the original data $(u, \lambda)$, as constructed by Proposition \ref{prop:Jost}, which guarantees that 
\begin{align*}
|\Phi_{+,1}(x,y) | & \leq C e^{\re(\frac{iK(\lambda)}{4} x - \frac{J(\lambda)}{4}y)}, \\
|\Phi_{-,2}(x,y)| & \leq C e^{\re(-\frac{iK(\lambda)}{4} x + \frac{J(\lambda)}{4}y)},
\end{align*}
for some $C < \infty$. Combined with \eqref{eq:shift.jost.end.eta.limit}, we find that, as $k \to \infty$,
\begin{align*}
|\Phi_{+,1}(x+x_k,y+y_k) | & \leq C e^{\kappa + o(1)} e^{\re(\frac{iK(\lambda)}{4} x - \frac{J(\lambda)}{4}y)}, \\
|\Phi_{-,2}(x+x_k,y+y_k)| & \leq C e^{-\kappa + o(1)} e^{\re(-\frac{iK(\lambda)}{4} x + \frac{J(\lambda)}{4}y)}.
\end{align*}
One can absorb the $e^{\pm\kappa + o(1)}$ factors into the constant $C$, and pass to these sections to subsequential $C^\infty_\textrm{loc}(\RR^2;\CC^2)$ limits $\hat\Phi_{+,1}, \hat \Phi_{-,2}$ that are parallel with respect to the limiting data ($\HH(px+qy),\lambda)$, and satisfy:
\begin{align}
|\hat \Phi_{+,1}(x,y) | & \leq C  e^{\re(\frac{iK(\lambda)}{4} x - \frac{J(\lambda)}{4}y)}\label{eq:ends.to.bd.est.lim.1}, \\
|\hat \Phi_{-,2}(x,y)| & \leq C   e^{\re(-\frac{iK(\lambda)}{4} x + \frac{J(\lambda)}{4}y)} \label{eq:ends.to.bd.est.lim.2}.
\end{align}
Note that
\begin{equation} \label{eq:ends.to.bd.state.det}
	\det(\hat \Phi_{+,1}, \hat \Phi_{-,2}) = \det(\Phi_{+,1}, \Phi_{-,2})
\end{equation}
due to the pointwise convergence and Lemma \ref{lemm:jost.wronskian}.

Consider the rotation $\phi(z) = (p+iq)z$. It is easy to see that
\[ (p + iq) \lambda = i \]
and
\[ \phi^*u_\infty = \HH(x), \]
the latter using, e.g., that $\phi(x,y) = (px-qy,qx+py)$ in real coordinates $(x, y)$ on $\CC$ and that $p (px-qy) + q(qx+py) = x$. By Lemma \ref{lemm:rotate.connection.lax.pair}, $\phi^*\hat\Phi_{+,1}$, $\phi^*\hat\Phi_{-,2}$ are parallel with respect to $(\HH(x),i)$. However, \eqref{eq:ends.to.bd.est.lim.1}, \eqref{eq:ends.to.bd.est.lim.2} imply  
\begin{align*}
|\phi^*\hat \Phi_{+,1}(x,y) | & \leq C  e^{-\frac x 2}, \\
|\phi^*\hat \Phi_{-,2}(x,y)| & \leq C   e^{\frac x 2},
\end{align*}
on $\RR^2$, since 
\[
\phi^* \left( - \frac{iK(\lambda)}{4} x + \frac{J(\lambda)}{4} y \right) = -\frac{iK(\lambda)}{4} (px-qy) + \frac{J(\lambda)}{4}(qx+py) = x.
\]
Corollary \ref{coro:one.bound.state} then implies that 
\[
\det(\phi^*\hat \Phi_{+,1}, \phi^*\hat \Phi_{-,2}) \equiv 0,
\]
and thus $\det(\Phi_{+,1}, \Phi_{-,2}) = 0$ by \eqref{eq:ends.to.bd.state.det}, so $\lambda \in \fB$ by definition of $\fB$.

If $u_\infty = - \HH(px+qy)$ we can use an identical argument, except at the very last step we apply Corollary \ref{coro:one.bound.state.flipped}. This completes the proof. 
\end{proof}

\begin{proof}[Proof of Proposition \ref{prop:bd.state.to.end}]
Fix $\lambda = q+ip \in \fB$, and note that $p>0$ since, by definition, $\fB \subset \mathring{H}_+$. We will show that
\[ R(\lambda) = (-q,p) \in \{v_1,\dots,v_{2m}\} \]
and describe below how to show that $-R(\lambda) = (q,-p) \in\{v_1,\dots,v_{2m}\}$.

Because $\lambda \in \cB$, the Jost solutions $\Phi_{+,1}, \Phi_{-,2}$ of Proposition \ref{prop:Jost} with data $(u, \lambda)$ satisfy $\det(\Phi_{+,1},\Phi_{-,2}) = 0$. Hence, $\Phi_{-,2} = c \Phi_{+,1}$ for some $c \in \CC\setminus\{0\}$ (it cannot hold that $c=0$ thanks to the asymptotic conditions of the Jost solutions as proven in Proposition \ref{prop:Jost}). 

Set $\Phi = \Phi_{-,2}$. By the bounds from Proposition \ref{prop:Jost}, we have 
\begin{equation}\label{eq:end.to.bd.state.Jost.bds}
(e^{\re(\frac{iK(\lambda)}{4} x - \frac{J(\lambda)}{4}y)} + e^{\re(-\frac{iK(\lambda)}{4} x + \frac{J(\lambda)}{4}y)}) |\Phi(x,y)| \in L^\infty(\RR^2)
\end{equation}
Observe that 
\begin{equation}\label{eq:end.to.bd.state.exp.to.line}
\frac{iK(\lambda)}{4} x - \frac{J(\lambda)}{4}y = - \frac{px+qy}{2},
\end{equation}
Note that for $\eta \in \RR $ fixed, the supremum
\[
\sup_{(-q,p) \cdot (x,y)=\eta} |\Phi(x,y)| 
\]
is attained at some point along the line $(-q,p) \cdot (x,y)=\eta$. Indeed, we can parametrize the line by $\gamma(t) := (pt,qt + \tfrac \eta p)$ so $p\gamma_1(t) + q\gamma_2(t) = p^2 t + q^2 t + \tfrac q p \eta = t +\tfrac q p \eta$, and combining \eqref{eq:end.to.bd.state.Jost.bds} with \eqref{eq:end.to.bd.state.exp.to.line} (and $p>0$) we get $|\Phi(\gamma(t))| \to 0$ as $t\to\pm\infty$, which verifies our assertion.

Hence, for any $\eta_k\to \infty$, there is $(x_k,y_k) \in \RR^2$ so that 
\[
\eta_k = (-q,p) \cdot (x_k, y_k)
\]
and
\[
0 < \mu_k : = |\Phi(x_k,y_k)| = \sup_{-qx+py=\eta_k} |\Phi(x,y)| < \infty.
\]
Set $\Phi_k(x,y) : = \mu_k^{-1} \Phi(x+x_k,y+y_k)$ and $u_k(x,y) = u(x+x_k,y+y_k)$. Passing to a subsequence, $\Phi_k$ and $u_k$ limit to $\tilde\Phi$ and $\tilde u$ in $C^\infty_\textrm{loc}$. By Proposition \ref{prop:asymp.dir.behavior}, we have that either $\tilde u \equiv 0,2\pi$ or $\tilde u= \pm \HH(\tilde px+\tilde q y+\tilde\eta)$ for some $(\tilde p,\tilde q) \in \SS^1$  $\tilde\eta \in \RR$. Also, $\tilde \Phi$ is parallel with respect to the limiting data $(\tilde u, \lambda)$, and 
\[
|\tilde \Phi(0,0)| = \max_{-qx+py=0} |\tilde\Phi(x,y)| = 1. 
\]
If $\tilde u \equiv 0$ or $\tilde u \equiv 2\pi$, we have that
\[
\tilde\Phi(x,y) =   \left( \begin{matrix} e^{\frac i 4 K(\lambda) x - \frac 1 4 J(\lambda)y} & 0 \\ 0 &  e^{-\frac i 4 K(\lambda) x + \frac 1 4 J(\lambda)y} \end{matrix} \right) \bv
\]
for some $\bv \in \RR^2\setminus\{0\}$ by \eqref{eq:defn.jost.trivial.u}. However, \eqref{eq:end.to.bd.state.exp.to.line} forces $\tilde \Phi$ to be unbounded (e.g., along one end of the line $(-q,p) \cdot (x,y) = 0$). This is a contradiction.

Thus, we find that $\tilde u = \pm \HH(\tilde p x + \tilde q y +\tilde\eta)$. We can translate to ensure  that $\tilde \eta = 0$ and 
\[
\max_{(-q,p) \cdot (x,y) =0} |\tilde\Phi(x,y)| = 1.
\]
(The maximum may no longer be attained at the origin.) By \eqref{eq:ends.deviate.asymptot.dir.bd.dist},
\[
\limsup_{k\to\infty} |(x_k,y_k)\cdot(\tilde p,\tilde q)| < \infty. 
\]
Thus, $(\tilde p,\tilde q) \neq \pm (-q,p)$ since $(x_k,y_k)\cdot(-q,p) = \eta_k \to \infty$.

Replacing $(\tilde p,\tilde q)$ by $(-\tilde p,-\tilde q)$ if necessary (this just changes the $\pm$ in front of $\HH(\tilde px+\tilde q y)$) we can assume that 
\begin{equation}\label{eq:bd.state.to.ends.tilde.pq.dot.prod.pos}
p\tilde p + q\tilde q > 0. 
\end{equation}
We now rotate the data. Set $\phi(z) := (\tilde p + i\tilde q)z$. Note that
 \[ \phi^* \tilde u(x,y) = \pm\HH(x) \]
and, by \eqref{eq:bd.state.to.ends.tilde.pq.dot.prod.pos},
\[
\mu : = (\tilde p + i\tilde q) \lambda =  \tilde p q - \tilde q p + i(\tilde p p + \tilde q q) \in \SS^1 \cap \mathring{H}_+.
\]
By Lemma \ref{lemm:rotate.connection.lax.pair}, $\phi^*\tilde \Phi$ is parallel with respect to $(\pm\HH(x),\mu)$. Furthermore, we observe that for $t \in\RR$,
\[
|\phi^*\tilde \Phi((p\tilde p+q\tilde q)t,(\tilde p q - \tilde q p)t)| = |\tilde\Phi(p t,qt)| \leq \max_{(-q,p) \cdot (x,y)=0} |\tilde\Phi(x,y)| = 1,
\]
so (recalling \eqref{eq:bd.state.to.ends.tilde.pq.dot.prod.pos}) if we define $\ell$ to be the line
\[
\ell : = \{ (x,y)\in\RR^2 :  (-\tilde p q + \tilde q p, p\tilde p + q \tilde q) \cdot (x, y) = 0 \} ,
\]
we see that 
\[
\sup_{(x,y) \in \ell} |\phi^*\Phi(x,y)| \leq 1. 
\]
By Corollary \ref{coro:jost.soln.bds.line.trivial} (if $\phi^* \tilde u(x,y) = +\HH(x)$) and Corollary \ref{coro:jost.soln.bds.line.trivial.flipped} (if $\phi^* \tilde u(x,y) = -\HH(x)$), we find that $\mu=i$ since $\phi^*\tilde\Phi \not \equiv 0$. Returning to the definition of $\mu$, we find that $(\tilde p,\tilde q) = \pm (p,q)$. By \eqref{eq:bd.state.to.ends.tilde.pq.dot.prod.pos}, we find that $(\tilde p,\tilde q) = (p,q)$.

We claim that $(-q,p) \in \{v_1,\dots,v_{2m}\}$. This will establish the claim (up to showing that $(q,-p)$ is also an asymptotic direction, which we do below). Returning to the limiting procedure used to find $\tilde u$, Proposition \ref{prop:asymp.dir.behavior} implies that up to passing to a subsequence, there is $j\in \{1,\dots,2m\}$ so that
\[
\frac{(x_k,y_k)}{|(x_k,y_k)|} \to v_j 
\]
as $k\to\infty$ and moreover $(p,q)\cdot v_j = 0$. This implies that $v_j \in\{ \pm (-q,p)\}$. The remaining issue is to show that the sign is determined by the geometric setup used above. Recalling that
\[
\eta_k = (-q,p)\cdot(x_k,y_k),
\]
and we chose $\eta_k\to \infty$, we find that
\[
0 \leq (-q,p) \cdot v_j = \pm 1,
\]
so we find that the sign must have been ``$+$,'' i.e. $v_j = (-q,p)$. Thus, we find that $(-q,p) \in \{v_1,\dots,v_{2m}\}$ as desired. 

Finally, we observe that if we had chosen $\eta_k\to-\infty$, the exact same argument as above would apply, except in the end we would find that $v_j = (q,-p)$. This completes the proof.
\end{proof}

\section{Immersed geodesics representing the $p$-widths}\label{sec:immersed.p.width} 

Fix $(M^2,g)$ a closed Riemannian manifold and $X$ a cubical subcomplex of $I^k$ with double cover $\pi : \tilde X\to X$, as in Section \ref{sec:AP.AC}. Recall the definition of phase transition critical set from Section \ref{subsec:dey}.

\begin{prop}\label{prop:phase.tran.crit.immersed}
If $V \in \bC_{\textnormal{PT}}(\tilde \Pi)$ and we work with the sine-Gordon double-well potential \eqref{eq:SG-potential}, then there exist primitive closed geodesics $\sigma_1,\dots,\sigma_N$  (repetitions allowed) so that 
\[
V = \sum_{j=1}^N \bv(\sigma_j,\mathbf{1}_{\sigma_j}). 
\]
\end{prop}
\begin{proof}
By definition of $\bC_{\textnormal{PT}}(\tilde \Pi)$ there is a sequence of min-max critical points $\{(u_i,\eps_i)\}_{i=1}^\infty \subset C^\infty(M) \times (0,\infty)$ so that $\eps_i\to 0$ and $V_{\eps_i}[u_i]\rightharpoonup V$. By Propositions \ref{prop:AC.min.max} and \ref{prop:AC-AP}, we have that 
\[
\limsup_i (\Index_{\eps_i}(u_{\eps_i}) + E_{\eps_i}(u_i)) < \infty. 
\]
The assertion now follows from Theorem \ref{theo:SG-lim}. 
\end{proof}
This immediately implies Theorem \ref{theo:immersed.geo.min.max} as follows:
\begin{proof}[Proof of Theorem \ref{theo:immersed.geo.min.max}]
Invoke Lemma \ref{lemm:P.p.m} to find cubical subcomplexes $X_i$ with bounded dimension and homotopy classes $\Pi_i$ containing $\bF$-continuous $p$-sweepouts so that $\lim_i \bL_{\textrm{AP}}(\Pi_i) = \omega_p(M,g)$. By Propositions \ref{prop:AC.min.max}, \ref{prop:AC-AP} and \ref{prop:phase.tran.crit.immersed}, there exist primitive closed geodesics $\sigma_{i,1},\dots,\sigma_{i,N_i}$ (repetition allowed) so that:
\[
\bL_{\textrm{AP}}(\Pi_i) = \sum_{j=1}^{N_i} \length_g(\sigma_{i,j}). 
\]
Note that $N_i$ is uniformly bounded above because $\length_g(\sigma_{i,j}) \geq 2\inj (M,g)$ for all $j$. Thus, we can pass to a subsequence and take a limit of the $\sigma_{i,j}$ yielding the desired geodesics. The exact statement of Theorem \ref{theo:immersed.geo.min.max} follows by grouping together geodesics that are not geometrically distinct.
\end{proof}

\section{The space of geodesic networks}\label{sec:geo.net}

In this section, we prove a bumpy metrics theorem for stationary geodesic networks under a length constraint. The unconstrained version of this result has been independently proven (in all codimensions) by Staffa \cite{Staffa:bumpy.geodesic.nets}. 

\subsection{The strata}

We fix a smooth closed 2-dimensional manifold M and a $k \in \NN$, $k \geq 3$. 

\begin{defi} \label{defi:geodesic.network.metric}
	We write
	\[ \met^k(M) := \{ C^k \text{ metrics on } M \}, \]
	and, for $g \in \met^k(M)$, $\eps > 0$,
	\[ \met^k(M, g, \eps) := \{ g' \in \met^k(M) : \Vert g' - g \Vert_{C^k} < \eps \}, \]
	where the $C^k$ norms are computed with respect to a smooth background metric on $M$ that can be fixed throughout the paper. 
\end{defi}

\begin{defi} \label{defi:geodesic.network.stationary.lambda}
	For $g \in \met^k(M)$, $\Lambda > 0$, we define:
	\[ \cS^\Lambda(g) := \{ S \in \cI\cV_1(M) \text{ is } g\text{-stationary and } \# \sing S + \Vert S \Vert(M, g) < \Lambda \}. \]
\end{defi}

We intend to prove a stratification theorem for $\cS^\Lambda(g)$ and a trichotomy theorem that results from this stratification. We rely substantially on the work of Allard--Almgren \cite{AllardAlmgren:1varifold}, according to which elements of $\cS^\Lambda(g)$ look like networks of geodesics and  the closure of $\cS^\Lambda(g)$ (in the varifold topology) is contained in (see Proposition \ref{prop:geodesic.network.stationary.varifold.limit} on $\# \sing S$ control under limits):
\[ \bar \cS^\Lambda(g) := \{ S \in \cI\cV_1(M) \text{ is } g\text{-stationary and } \# \sing S + \Vert S \Vert(M, g) \leq \Lambda \}. \]
In what follows, we proceed to construct good models for $\cS^\Lambda(g)$.

\begin{defi}[Graph structure] \label{defi:geodesic.network.graph.structure}
	We say $G = (V, E, \omega)$ is a graph structure if $G$ is a finite undirected weighted simple graph with vertices $V$, edges $E$, and edge weights $\omega : E \to \NN^*$. For $v \in V$, $E_v \subset V$ denotes the set of vertices that are joined to $v$, and $\deg_G v := \# E_v$ is the degree of $v$ in the graph.
\end{defi}

Graph structures will model the ``topology'' of elements of $\cS^\Lambda(g)$. One complication in trying to model elements of $\cS^\Lambda(g)$ with graphs is that pairs of points on a manifold can often be connected by multiple geodesic segments, and Definition \ref{defi:geodesic.network.graph.structure} does not distinguish between these segments. Our solution to this is to subdivide our edges by introducing auxiliary vertices so that each edge in this subdivided graph structure can be taken to be the unique length-minimizing geodesic segment between its endpoints. 

\begin{defi}[Graph immersion] \label{defi:geodesic.network.immersion}
	Let $G$ be a graph structure, $g \in \met^k(M)$. We call $\mathbf{p} \in M^V$ an immersion of $G$ in $(M,g)$ provided it satisfies:
	\begin{enumerate}
		\item[(I$_1$)] $\mathbf{p}$ is injective, and
		\item[(I$_2$)] $\{u,v\} \in E \implies \dist_g(\mathbf{p}(u), \mathbf{p}(v)) < \inj(M,g)$ .
	\end{enumerate}
	The space of all immersions of $G$ in $(M, g)$ is denoted $\imm_g(G, M)$.
\end{defi}

\begin{defi}[Varifold associated to immersed graph]
	Given a graph structure $G$, $g \in \met^k(M)$, and $\mathbf{p} \in \imm_g(G, M)$, we define $\iota_{g}(G, \mathbf{p}) \in \cI \cV_1(M)$ to be the integral 1-varifold associated with $(G, \mathbf{p})$, i.e.:
	\[ \iota_{g}(G, \mathbf{p}) := \sum_{\{u,v\} \in E} \mathbf{v}(\sigma_g(\mathbf{p}(u), \mathbf{p}(v)), \omega(\{u,v\}) \mathbf{1}_{\sigma_g(\mathbf{p}(u), \mathbf{p}(v))}), \]
	where the notation $\mathbf{v}(\cdot, \cdot)$ is as in \cite[Chapter 4]{Simon83} and $\sigma_g(p, q)$ denotes the closed $g$-length minimizing geodesic segment from $p$ to $q$ in $(M, g)$ provided $0 < \dist_g(p, q) < \inj(M, g)$ (which immersions satisfy).
\end{defi}

Note that $\iota_{g}(G, \cdot)$ is far from injective on $\imm_g(G, M)$. We proceed to make various canonical choices to control this gauge freedom.

\begin{defi}[$Q$-subdivided graph structure] \label{defi:geodesic.network.subdivision}
	A graph structure $G$ is $Q$-subdivided for some fixed $Q \in \NN$, $Q \geq 2$, provided it satisfies:
	\begin{itemize}
		\item[(S$_1$)] If $u_1, \ldots, u_k \in V$ are distinct vertices with degree $2$, $u_0$, $u_{k+1} \in V$ are vertices (not necessarily distinct) with degree $\neq 2$, and
			\[ \{ u_0, u_1 \}, \{ u_1, u_2 \}, \ldots, \{ u_{k-1}, u_k \}, \{ u_k, u_{k+1} \} \in E, \]
			then $k=Q$. Such configurations will be called ``chains'' in $(V, E)$.
		\item[(S$_2$)] If $u_0, u_1, \ldots, u_k \in V$ are distinct vertices with degree $2$, and
			\[ \{ u_0, u_1 \}, \{ u_1, u_2 \}, \ldots, \{ u_{k-1}, u_k \}, \{ u_k, u_0 \} \in E, \]
			then $k=Q$. Such configurations will be called ``cycles'' in $(V, E)$.
	\end{itemize}
\end{defi}

\begin{defi}[Graph embedding] \label{defi:geodesic.network.embedding}
	Suppose $G$ is a $Q$-subdivided graph structure for some $Q \in \NN$, $Q \geq 2$, and $g \in \met^k(M)$. We call $\mathbf{p} \in M^V$ an embedding of $G$ in $(M, g)$ provided it satisfies:
	\begin{enumerate}
		\item[(E$_1$)] $\mathbf{p} \in \imm_g(G, M)$.
		\item[(E$_2$)]   for all $\{ u, v \} \neq \{u', v'\} \in E$,
		\[   \sigma_g(\mathbf{p}(u), \mathbf{p}(v)) \cap \operatorname{int} \sigma_g(\mathbf{p}(u'), \mathbf{p}(v')) =  \emptyset. \]
	\end{enumerate}
	We call $\mathbf{p}$ a \textit{balanced} embedding if it additionally satisfies:
	\begin{enumerate}
		\item[(E$_3$)] If $u \in V$ has $E_u = \{v, v'\}$, then $\dist_g(\mathbf{p}(u), \mathbf{p}(v)) = \dist_g(\mathbf{p}(u), \mathbf{p}(v'))$.
	\end{enumerate}
	The space of embeddings of $G$ in $(M, g)$ is denoted $\emb_g(G, M)$, and the space of balanced embeddings of $G$ in $(M, g)$ is denoted $\cB \emb_g(G, M)$.
\end{defi}

Note that $\imm_g(G, M)$ and $\emb_g(G, M)$ are open in the product topology on $M^V$, and $\cB \emb_g(G, M)$ is relatively closed in $\emb_g(G, M)$.

\begin{lemm}[Structure of stationary embeddings] \label{lemm:geodesic.network.embedding.structure}
	Let $G$ be a $Q$-subdivided graph structure, $Q \in \NN$, $Q \geq 2$, $g \in \met^k(M)$, $\mathbf{p} \in \emb_g(G, M)$ be so that $\iota_g(G, \mathbf{p})$ is $g$-stationary.
	\begin{enumerate}
		\item \begin{enumerate}
				\item There exists an $\eps > 0$ such that for every $u \in V$, $\{u', v'\} \in E$
					\[  B^g_\eps(\mathbf{p}(u)) \cap \sigma_g(\mathbf{p}(u'), \mathbf{p}(v')) \neq \emptyset \implies u \in \{u', v'\}. \]
				\item   $\sing \iota_g(G, \mathbf{p}) = \{ \mathbf{p}(u) : u \in V, \; \deg_G u \neq 2 \}$.
				\item For $u \in V$ with $\deg_G u = 2$, $\omega(\{u,v\})$ is independent of $v \in E_u$.
			\end{enumerate}
		\item Suppose $G$ is connected and contains a vertex of degree $\neq 2$. Then:
			\begin{enumerate}
				\item $\reg \iota_g(G, \mathbf{p})$ contains no closed connected components.
				\item Each $u \in V$ with $\deg_G u = 2$ is a $u_i$ with $i \in \{ 1, \ldots, Q \}$ in a chain as in Definition \ref{defi:geodesic.network.subdivision}; the chain is unique. Each $u \in V$ with $\deg_G u \neq 2$ is a $u_i$ with $i \in \{ 0, Q+1 \}$ in a chain as in Definition \ref{defi:geodesic.network.subdivision}; this chain need not be unique. Thus,
					\[ \qquad \qquad \quad \# V = \# \sing \iota_g(G, \mathbf{p}) + Q \cdot \# \{ \text{components of } \reg \iota_g(G, \mathbf{p}) \} \]
					in the configuration of 1 vertex per singular point and $Q$ vertices along each regular segment.
				\item Each $e \in E$ is of them form $\{ u_i, u_{i+1} \}$ with $i \in \{ 0, \ldots, Q \}$ in a chain as in Definition \ref{defi:geodesic.network.subdivision}; the chain is unique. Thus,
					\[ \#E = (Q+1) \cdot \# \{\text{components of } \reg \iota_g(G, \mathbf{p})\} \]
					in the configuration of $Q+1$ edges along each regular segment.
				\item For all $\{u, v\} \in E$,
					\[ \omega(\{u,v\}) = \Theta^1(\iota_g(G, \mathbf{p}), \cdot) \]
					evaluated at any interior point of $\sigma_g(\mathbf{p}(u), \mathbf{p}(v))$.
			\end{enumerate}
		\item Suppose $G$ is connected and only contains degree $2$ vertices. Then:
			\begin{enumerate}
				\item $\reg \iota_g(G, \mathbf{p})$ consists of a single closed geodesic loop.
				\item Each $u \in V$ is a $u_i$ with $i \in \{ 0, \ldots, Q \}$ in cycle as in  Definition \ref{defi:geodesic.network.subdivision}; the cycle is unique. Thus,
					\[ \# V = Q+1. \]
				\item Each $e \in E$ is of the form $\{ u_i, u_{i+1} \}$ (indices taken mod $Q+1$) with $i \in \{ 0, \ldots, Q \}$ in cycle as in Definition \ref{defi:geodesic.network.subdivision}. Thus,
					\[ \# E = Q+1. \]
				\item For all $\{u, v\} \in E$,
					\[ \omega(\{u,v\}) = \Theta^1(\iota_g(G, \mathbf{p}), \cdot) \]
					evaluated at any point on $\supp \iota_g(G, \mathbf{p})$.
			\end{enumerate}
		\item If $G' \subset G$ is a connected component, then $\mathbf{p} \vert {G'} \in \emb_g(G', M)$ and is  balanced if $\mathbf{p}$ is.
		\item If $G'$, $G'' \subset G$ are distinct connected components, then
				\[ \supp \iota_g(G', \mathbf{p} \vert {G'}) \cap \supp \iota_g(G'', \mathbf{p} \vert {G''}) = \emptyset. \]
	\end{enumerate}
\end{lemm}
\begin{proof}
	(1)(a). If this failed with $\eps \to 0$, then by the finiteness of $V$, $E$, there would exist fixed $u \in V$, $\{u', v'\} \in E$ with
	\[ u \not \in \{u', v'\} \text{ and } \mathbf{p}(u) \in \sigma_g(\mathbf{p}(u'), \mathbf{p}(v')). \]
	In view of (E$_1$) (specifically (I$_1$)), this in turn implies that
	\[ u \not \in \{u', v'\} \text{ and } \mathbf{p}(u) \in \operatorname{int} \sigma_g(\mathbf{p}(u'), \mathbf{p}(v')), \]
	which violates (E$_2$).
	
	(1)(b), (1)(c) We first claim that
	\begin{equation} \label{eq:geodesic.network.embedding.structure.1.a.sing.in.p}
		\sing \iota_g(G, \mathbf{p}) \subset \{ \mathbf{p}(u) : u \in V \}.
	\end{equation}
	Take $p \in \sing \iota_g(G, \mathbf{p})$. There exists $\{u',v'\} \in E$ so that $p \in \sigma_g(\mathbf{p}(u'), \mathbf{p}(v'))$, otherwise $p \not \in \supp \iota_g(G, \mathbf{p}) \supset \sing \iota_g(G, \mathbf{p})$. Now if
	\[ p \in \sigma_g(\mathbf{p}(u'), \mathbf{p}(v')) \setminus \operatorname{int} \sigma_g(\mathbf{p}(u'), \mathbf{p}(v')) = \{ \mathbf{p}(u'), \mathbf{p}(v') \}, \]
	then \eqref{eq:geodesic.network.embedding.structure.1.a.sing.in.p} follows. So let us rule out
	\[ p \in \operatorname{int} \sigma_g(\mathbf{p}(u'), \mathbf{p}(v')). \]
	In this case there has to exist another $\{u, v\} \in E$ such that $p \in \sigma_g(\mathbf{p}(u), \mathbf{p}(v))$, otherwise $\iota_g(G, \mathbf{p})$ would locally equal the smooth curve $\operatorname{int} \sigma_g(\mathbf{p}(u), \mathbf{p}(v))$ near $p$, contradicting $p \in \sing \iota_g(G, \mathbf{p})$. But the existence of $\{u, v\}$ means 
	\[ p \in \sigma_g(\mathbf{p}(u), \mathbf{p}(v)) \cap \operatorname{int} \sigma_g(\mathbf{p}(u'), \mathbf{p}(v')), \]
	in violation of (E$_2$). This completes the proof of \eqref{eq:geodesic.network.embedding.structure.1.a.sing.in.p}.
	
	Given \eqref{eq:geodesic.network.embedding.structure.1.a.sing.in.p}, it remains to show that $\mathbf{p}(u)$ is singular if and only if $\deg_G u \neq 2$. Let $u \in V$. By  part (1)(a), there exists a ball around $\mathbf{p}(u)$ such that the only segments entering the ball are of the form $\sigma_g(\mathbf{p}(u), \mathbf{p}(v))$, $v \in E_u$. Then, the stationarity condition (\cite[(1)]{AllardAlmgren:1varifold}) applied to $\mathbf{p}(u)$ is
		\begin{equation} \label{eq:geodesic.network.embedding.structure.sum.tangents}
			\sum_{v \in E_u} \omega(\{u,v\}) \tau_{u,v}(\mathbf{p}(u)) = 0;
		\end{equation}
		here $\tau_{u,v}$ is the unit normal along $\sigma_g(\mathbf{p}(u), \mathbf{p}(v))$ oriented from $\mathbf{p}(u)$ to $\mathbf{p}(v)$. 
		
		Suppose $\deg_G u = 2$, and write $E_u = \{v_1, v_2\}$. Then \eqref{eq:geodesic.network.embedding.structure.sum.tangents} implies that $\tau_{u,v_1}(\mathbf{p}(u))$, $\tau_{u,v_2}(\mathbf{p}(u))$ are linearly dependent, so the curves $\sigma_g(\mathbf{p}(u), \mathbf{p}(v_1))$ and $\sigma_g(\mathbf{p}(u), \mathbf{p}(v_2))$ join smoothly at $\mathbf{p}(u)$ and $\mathbf{p}(u) \in \reg \iota_g(G, \mathbf{p})$. Now, $\omega(\{u,v_1\}) = \omega(\{u,v_2\})$ follows from \eqref{eq:geodesic.network.embedding.structure.sum.tangents} and the fact that $\tau_{u,v_1}(\mathbf{p}(u))$, $\tau_{u,v_2}(\mathbf{p}(u)$ have unit length.
		
		Suppose $\deg_G u \neq 2$. Then no two of the segments $\sigma_g(\mathbf{p}(u), \mathbf{p}(v))$, where $v \in E_u$, may overlap; otherwise, we contradict either the injectivity of $\mathbf{p}$ (Definition \ref{defi:geodesic.network.immersion}'s (I$_1$)) or Definition \ref{defi:geodesic.network.embedding}'s (E$_2$). 
		
		As a result, no two of the tangent vectors in \eqref{eq:geodesic.network.embedding.structure.sum.tangents} coincide, so the varifold tangent $\VarTan(\iota_g(G, \mathbf{p}), \mathbf{p}(u))$ isn't supported on a line, so $\mathbf{p}(u) \in \sing \iota_g(G, \mathbf{p})$.
		
	(2)(a). By construction, $\supp \iota_g(G, \mathbf{p})$ is connected. If $\reg \iota_g(G, \mathbf{p})$ contained a closed connected component, then it would have to coincide with $\supp \iota_g(G, \mathbf{p})$. This contradicts $\sing \iota_g(G, \mathbf{p}) \neq \emptyset$ by part (1)(b).
	
	(2)(b), (2)(c). It is well-known that the degree-2 vertices form a subgraph whose connected components form cycles or chains, i.e., (S$_1$)'s or (S$_1$)'s. By part (2)(a), only chains may occur, since cycles give rise to closed connected components of $\reg \iota_g(G, \mathbf{p})$ by part (1)(b). By Definition \ref{defi:geodesic.network.embedding} (E$_2$) and part (1)(b), distinct chains trace out distinct components of $\reg \iota_g(G, \mathbf{p})$. The rest follows from Definition \ref{defi:geodesic.network.graph.structure} (S$_1$).
	
	(2)(d). This follows from the definition of $\iota_g(G, \mathbf{p})$, part (1)(b), and constancy.
	
	(3)(a). The result follows from the fact that $\supp \iota_g(G, \mathbf{p})$ must be closed and connected, while $\sing \iota_g(G, \mathbf{p}) = \emptyset$ by part (1)(b).
	
	(3)(b), (3)(c). It is well-known that a connected graph with only degree-2 vertices must be a cycle. So, our graph is as in Definition \ref{defi:geodesic.network.graph.structure} (S$_2$).
	
	(3)(d). This follows from the definition of $\iota_g(G, \mathbf{p})$, part (1)(b), and constancy.
	
	(4). Trivial.
	
	(5). Definition \ref{defi:geodesic.network.embedding}'s (E$_2$) guarantees that distinct connected components have non-intersecting images.
\end{proof}

\begin{coro} \label{coro:geodesic.network.embedding.structure}
	Fix $Q \in \NN$, $Q \geq 2$, and $g \in \met^k(M)$. Suppose, for $i=1$, $2$, $G_i$ is a $Q$-subdivided graph structure and $\mathbf{p}_i \in \emb_g(G_i, M)$, and that
	\[ \iota_g(G_1, \mathbf{p}_1) = \iota_g(G_2, \mathbf{p}_2) \text{ is } g\text{-stationary}. \]
	Then, there exists a graph isomorphism $\varphi : G_1 \to G_2$ satisfying:
	\begin{enumerate}
		\item $\iota_g(G_1', (\mathbf{p}_2 \circ \varphi) \vert {G_1'}) = \iota_g(G_1', \mathbf{p}_1 \vert {G_1'})$ for every component $G_1' \subset G_1$.
		\item Suppose $G_1' \subset G_1$ is a component containing a degree $\neq 2$ vertex and that $\mathbf{p}_1$, $\mathbf{p}_2$ are both balanced. Then, $\mathbf{p}_2 \circ \varphi \equiv \mathbf{p}_1$ on $V(G_1')$.
		\item Suppose $G_1' \subset G_1$ is a component containing only degree $2$ vertices and that $\mathbf{p}_1$, $\mathbf{p}_2$ are both balanced. Then, there exists a parallel tangent vector field $\tau$ on $\iota_g(G_1', \mathbf{p}_1 \vert {G_1'})$ such that $\mathbf{p}_2 \circ \varphi \equiv \exp_{\mathbf{p}_1} \tau$ on $V(G_1')$ and $\Vert \tau \Vert < 2 \Vert \iota_g(G_1', \mathbf{p}_1 | G_1')(M, g)\Vert/(Q+1)$.
	\end{enumerate}
\end{coro}

\begin{defi} \label{defi:geodesic.network.manifold.slice}
	Let $\Lambda > 0$, $g \in \met^k(M)$, and $G$ be a $Q$-subdivided graph structure with any $Q \in \NN$, $Q \geq 2$. We define:
	\[ \cS^\Lambda_{G}(g) := \cS^\Lambda(g) \cap \{ \iota_g(G, \mathbf{p}) : \mathbf{p} \in \emb_g(G, M) \}.  \]
\end{defi}

The following theorem tells us that our definitions capture $\cS^\Lambda(g)$ well as long as $Q$ is large, and that we may restrict to $\cB \emb_g(G, M) \subset \emb_g(G, M)$.

\begin{theo}[Finite stratification theorem] \label{theo:geodesic.network.finite.stratification}
	Suppose $g \in \met^k(M)$, $\Lambda > 0$, and $Q \in \NN$ satisfy
	\begin{equation} \label{eq:geodesic.network.finite.stratification.Q}
		Q \cdot \inj(M, g) > \Lambda.
	\end{equation}
	There exists a finite set $\cG = \cG(\Lambda, Q)$ of $Q$-subdivided graph structures with
	\begin{equation} \label{eq:geodesic.network.finite.stratification.bounds}
		\# \cG \leq C(\inj(M, g), \Lambda, Q),
	\end{equation}
	so that if $\cB \emb_g(\cG, M) := \cup_{G \in \cG} (\{G\} \times \cB \emb_g(G, M))$, then one can construct a map $j^{\Lambda}_Q : \cS^\Lambda(g) \to \cB \emb_g(\cG, M)$ satisfying
	\[ \iota_{g} \circ j^\Lambda_Q \equiv \operatorname{Id} \text{ on } \cS^\Lambda(g). \]
	Thus, $\cS^\Lambda(g) = \cup_{G \in \cG} \cS^\Lambda_G(g)$, i.e., $\cS^\Lambda(g)$ can be generated by balanced embeddings of finitely many $Q$-subdivided graphs structures.
\end{theo}
\begin{proof}
	It will suffice to define $j^\Lambda_Q$. To that end, fix some $S \in \cS^\Lambda(g)$. 
	
	\textbf{Step 1} (constructing $G$ and $\mathbf{p}$). It follows from \cite[Section 3]{AllardAlmgren:1varifold} that $S \in \cS^{\Lambda}(g)$ must be supported on a \textit{finite} geodesic network in $(M, g)$. We will construct $G = (V, E, \omega)$ and $\mathbf{p} \in \cB \emb_g(G, M)$ one component of $\reg S$ at a time. Let $C_\alpha$, $\alpha \in A$, be the connected components of $\reg S$. 
	\begin{itemize}
		\item Suppose $\bar C_\alpha \cap \sing S = \emptyset$. Then $C_\alpha$ is a loop with density $\theta_\alpha \in \NN^*$ by \cite[Section 3]{AllardAlmgren:1varifold}. Take $Q+1$ equidistant points $p^\alpha_0, \ldots, p^\alpha_Q$ along $C_\alpha$. From \eqref{eq:geodesic.network.finite.stratification.Q} and the equidistance of the $p_i$ along $C_\alpha$, 
		\[ \dist_{C_\alpha,g}(p^\alpha_i, p^\alpha_{i+1}) < Q^{-1} \Vert S \Vert(M,g) < Q^{-1} \Lambda < \inj(M, g), \]
		\[ \implies \dist_g(p^\alpha_i, p^\alpha_{i+1}) = \dist_{C_\alpha,g}(p^\alpha_i, p^\alpha_{i+1}) < \inj(M,g). \]
		Construct $G_\alpha := (V_\alpha, E_\alpha, \omega_\alpha)$ with
		\begin{align*}
			V_\alpha & := \{ p^\alpha_0, \ldots, p^\alpha_Q \}, \\
			E_\alpha & := \{ \{p^\alpha_0, p^\alpha_1\}, \ldots, \{ p^\alpha_{Q-1}, p^\alpha_Q\}, \{ p^\alpha_Q, p^\alpha_0 \} \}, \\
			\omega_\alpha & := \theta_\alpha \text{ on } E_\alpha, 
		\end{align*}
		and set $\mathbf{p}_\alpha(p^\alpha_i) := p^\alpha_i$  for all $i = 0, \ldots, Q$. 
		\item Suppose $\bar C_\alpha \cap \sing S \neq \emptyset$. Then $C_\alpha$ is a segment with density $\theta_\alpha \in \NN^*$ and $\bar C_\alpha \setminus C_\alpha = \{ s^\alpha_1, s^\alpha_2 \}$ for  $s^\alpha_1, s^\alpha_2 \in \sing S$ (not necessarily distinct) by \cite[Section 3]{AllardAlmgren:1varifold}. Consider $Q+2$ equidistant points $p^\alpha_0, \ldots, p^\alpha_{Q+1}$ along $\bar C_\alpha$, with $p^\alpha_0
= s^\alpha_1$, $p^\alpha_{Q+1} = s^\alpha_2$. From \eqref{eq:geodesic.network.finite.stratification.Q} and the equidistance of the $p_i$ along $S$, we get similarly to the bullet point above that
		\[ \dist_g(p^\alpha_i, p^\alpha_{i+1}) = \dist_{\bar C_\alpha,g}(p^\alpha_i, p^\alpha_{i+1}) < \inj(M,g). \]
		Construct $G_\alpha := (V_\alpha, E_\alpha, \omega_\alpha)$ so that
		\begin{align*}
			V_\alpha & := \{ p^\alpha_0, \ldots, p^\alpha_{Q+1} \}, \\
			E_\alpha & := \{ \{p^\alpha_0, p^\alpha_1\}, \ldots, \{p^\alpha_Q, p^\alpha_{Q+1}\} \}, \\
			\omega_\alpha & := \theta_\alpha \text{ on } E_\alpha, 
		\end{align*}
		and set $\mathbf{p}_\alpha(p^\alpha_i) := p^\alpha_i$ for all $i = 0, \ldots, Q+1$. 
	\end{itemize}
	After doing this for all $\alpha \in A$, then define $V := \cup_{\alpha \in A} V_\alpha \subset M$, $E := \cup_{\alpha \in A} E_\alpha$, $\omega := \cup_{\alpha \in A} \omega_\alpha$, $\mathbf{p} := \cup_{\alpha \in A} \mathbf{p}_\alpha$, where points that  correspond to the same point on $M$ are obviously identified in these unions. By construction, $G$ is a subdivided graph structure and $\mathbf{p} \in \cB \emb_g(G, M)$. 
	
	\textbf{Step 2} (controlling the graph size). To show that a finite set $\cG$ of graphs $G$ will be sufficient, with \eqref{eq:geodesic.network.finite.stratification.bounds} holding, it suffices to bound:
	\begin{equation} \label{eq:geodesic.network.stratification.size.bound}
		\# V + \max_E \omega \leq C(\inj(M, g_0), \Lambda, Q).
	\end{equation}
	The mass bound $\Vert S \Vert(M, g) < \Lambda$ for $S \in \cS^{\Lambda}(g)$ gives the density bound
	\begin{equation} \label{eq:geodesic.network.stratification.density}
		\Theta^1(S, \cdot) \leq \Theta_0 = \Theta_0(\inj(M, g), \Lambda)
	\end{equation}
	by the monotonicity formula for stationary 1-varifolds (\cite[Section 2]{AllardAlmgren:1varifold}).  
	
	Our bound on $\max_E \omega$ follows from \eqref{eq:geodesic.network.stratification.density} and Lemma  \ref{lemm:geodesic.network.embedding.structure}'s (2)(d), (3)(d). To bound $\# V$, Lemma   \ref{lemm:geodesic.network.embedding.structure}'s (2)(a), (3)(a) say we need to bound $\# \sing S$ and the number of components of $\reg S$. Certainly, $\# \sing S < \Lambda$ when $S \in \cS^{\Lambda}(g)$. Connected components of $\reg S$ are of two types: loops and segments. The number of closed loops is bounded from above in the desired form in view of the mass bound $\Vert S \Vert(M, g) < \Lambda$ and the monotonicity formula. Finally, since segments start and end at singular points and $\# \sing S < \Lambda$, and the number of distinct segments joining the same pair of singular points is controlled by the density of $S$'s vertices, the result follows from \eqref{eq:geodesic.network.stratification.density} again.
\end{proof}

\subsection{Manifold structure of strata and a trichotomy}

We direct the reader to \cite[II.1]{Lang:fundamentals.dg} for information on Banach manifolds.

\begin{theo} [Manifold structure theorem, cf. {\cite[Theorem 2.1]{White:bumpy.old}}] \label{theo:geodesic.network.manifold}
	Let $\Lambda > 0$, $G$ be a $Q$-subdivided graph structure with $Q$ satisfying \eqref{eq:geodesic.network.finite.stratification.Q}, 
	\[ \cS^\Lambda_G := \{ (g, S) : g \in \met^k(M), \; Q \cdot \inj(M, g) > \Lambda, \; S \in \cS^\Lambda_G(g) \}, \]
	and suppose $\pi^\Lambda_G : \cS^\Lambda_G \to \met^k(M)$ is the projection $(g, S) \mapsto g$. Then, there exists an atlas for $\cS^\Lambda_G$ with respect to which:
	\begin{enumerate}
		\item $\cS^\Lambda_G$ is a second countable\footnote{\cite[Theorem 2.1]{White:bumpy.old} discusses separability but second countability holds too, e.g., via a simple metrizability verification \cite{White:bumpy.separable}. The stronger conclusion of second countability is necessary for Sard--Smale applications  \cite{Smale:sard}.} $C^{k-1}$ Banach manifold.
		\item $\pi^\Lambda_{G}$ is $C^{k-1}$ and Fredholm, with Fredholm index zero.
		\item For every $(g, S) \in \cS^\Lambda_G$,
		\begin{align*}
			& (g, S) \text{ is a singular point for } \pi^\Lambda_G \\
			& \qquad \iff S \text{ has a not-everywhere-tangential} \\
			& \qquad \qquad \text{ stationary varifold Jacobi field in } (M,g).
		\end{align*}
		(See Definition \ref{defi:geodesic.network.jacobi.field.geom}.) 
		\item The set of regular values of $\pi^\Lambda_{G}$ is comeager (``Baire generic'') in $\met^k(M)$.
	\end{enumerate}
\end{theo}

\begin{rema} \label{rema:geodesic.network.manifold.regularity}
	While we did not explicitly indicate so in the notation, $\cS^\Lambda_G$ also depends on the number $k \in \NN$, $k \geq 3$, albeit in a mild manner. If we were to write $\cS^{\Lambda,k}_G$ rather than $\cS^\Lambda_G$ for the space in Theorem \ref{theo:geodesic.network.manifold}, then it is straightforward to check directly from the definition that for every $k' \geq k$
	\begin{equation} \label{eq:geodesic.network.manifold.regularity.S}
		\cS^{\Lambda,k'}_G = \cS^{\Lambda,k}_G \cap (\met^{k'}(M) \times \cI\cV_1(M)).
	\end{equation}
	Likewise, if the projection to $\met^k(M)$ is denoted by $\pi^{\Lambda,k}_G$ and the set of its regular values by $\cR^{\Lambda,k}_G \subset \met^k(M)$, then for every $k' \geq k$
	\begin{equation} \label{eq:geodesic.network.manifold.regularity.R}
		\cR^{\Lambda,k}_G = \cR^{\Lambda,k'}_G \cap \met^{k'}(M),
	\end{equation}
	by \eqref{eq:geodesic.network.manifold.regularity.S} and Theorem \ref{theo:geodesic.network.manifold}'s (3).
\end{rema}

If we denote
\[ \met(M) := \{ C^\infty \text{ metrics on } M\}, \]
and endow it with the usual $C^\infty$ topology, then abstract arguments imply: 

\begin{coro} \label{coro:geodesic.network.manifold.generic.smooth}
	The set of regular values of $\pi^\Lambda_G$ that are also in $\met(M)$ is comeager (``Baire generic'') in $\met(M)$.
\end{coro}

 This follows, e.g., from \cite[Lemma 6.2]{Staffa:bumpy.geodesic.nets} with $\mathcal{M}^k = \met^k(M)$, $\mathcal{M}^\infty = \met(M)$, $\mathcal{N}^k =$ regular values of $\pi^\Lambda_G$ in $\met^k(M)$ (i.e., $\cR^{\Lambda,k}_G$), and $\mathcal{N}^\infty =$ regular values of $\pi^\Lambda_G$ in $\met(M)$. The key is that $\cN^{k'} = \cN^k \cap \cM^{k'}$ for $k' \geq k$ by \eqref{eq:geodesic.network.manifold.regularity.R} and that $\cM^\infty \subset \cM^k$ is dense. See also \cite[Theorem 2.10]{White:bumpy.new}.

The proof of the main theorem is postponed until the next subsection. We owe the following definition:

\begin{defi}[Stationary varifold Jacobi field] \label{defi:geodesic.network.jacobi.field.geom}
	Let $g \in \met^k(M)$, $\Lambda > 0$, $S \in \cS^\Lambda(g)$. A stationary varifold Jacobi field along $S$ is a section $J$ of $TM$ along $\supp S$ with the following properties:
	\begin{enumerate}
		\item[(J$_1$)] $J \in C^0(\supp S)$;
		\item[(J$_2$)] $J \in C^2(\reg S)$ and satisfies the Jacobi equation along $\reg S$; and,
		\item[(J$_3$)] for all $u \in V$,
			\[ \sum_\tau \omega_\tau \nabla_\tau^\perp J(\mathbf{p}(u)) = 0, \]
			where the sum is over all unit vectors $\tau$ in the support of the tangent cone of $S$ at $\mathbf{p}(u)$, $\omega_\tau$ is the density of the cone in the direction $\tau$, and $\nabla_\tau^\perp$ is the unit speed covariant derivative in the $\tau$ direction, projected onto $\{ \tau \}^\perp$.
	\end{enumerate}
	A stationary varifold Jacobi field is said to be:
	\begin{enumerate}
		\item[(J$_4$)] (\textit{not-})\textit{everywhere-tangential} if $J^\perp \equiv 0$ (resp. $J^\perp \not \equiv 0$) along $\reg S$, where $\perp$ denotes the projection onto the normal bundle $N(\reg S)$.
	\end{enumerate}
\end{defi}

It will be more convenient for us to work with discrete Jacobi fields that are only defined along the $\sigma_g(\mathbf{p}(u), \mathbf{p}(v))$, $\{u, v\} \in E$. The following lemma helps us go back and forth between the continuous and discrete settings:

\begin{lemm} \label{lemm:geodesic.network.jacobi.field}
	Take $g \in \met^k(M)$, $Q \in \NN$, $Q \geq 2$, a $Q$-subdivided graph structure $G$, $\mathbf{p} \in \emb_g(G, M)$ such that $S := \iota_g(G, \mathbf{p})$ is $g$-stationary, and a collection of Jacobi fields $\{ J_{u,v} : \{u,v\} \in E \}$ on the segments $\sigma_g(\mathbf{p}(u), \mathbf{p}(v))$, $\{u,v\} \in E$, with $J_{u,v} = J_{v,u}$. Assume our collection of Jacobi fields satisfies:
	\begin{enumerate}
		\item[(J$_1$')] for all $u \in V$, $J_{u,v}(\mathbf{p}(u))$ is independent of $v \in E_u$;
		\item[(J$_2$')] for all $u \in V$ with $\deg_G u = 2$, $\sum_{v \in E_u} \nabla_{u,v}^T J_{u,v}(\mathbf{p}(u)) = 0$, where $T$ is the projection to the tangent bundle of $\sigma_g(\mathbf{p}(u), \mathbf{p}(v))$;
		\item[(J$_3$')] for all $u \in V$, $\sum_{v \in E_u} \omega(\{u,v\}) \nabla_{u,v}^\perp J_{u,v}(\mathbf{p}(u)) = 0$, where $\nabla_{u,v}$ denotes the unit speed covariant derivative along $\sigma_g(\mathbf{p}(u), \mathbf{p}(v))$ oriented from $\mathbf{p}(u)$ to $\mathbf{p}(v)$ and $\perp$ denotes the projection onto the normal bundle along $\sigma_g(\mathbf{p}(u), \mathbf{p}(v))$.
	\end{enumerate}
	Then, there exists a stationary varifold Jacobi field $J$ along $S$ such that:
	\begin{enumerate}
		\item[(1')] $J \equiv J_{u,v}$ along $\sigma_g(\mathbf{p}(u), \mathbf{p}(v))$ for all $\{u,v\} \in E$.
	\end{enumerate}
	If the collection of Jacobi fields only satisfies (J$_1$'), (J$_3$'), then there exists a stationary varifold Jacobi field $J$ along $S$ such that
	\begin{enumerate}
		\item[(1'')] $J^\perp \equiv J_{u,v}^\perp$ along $\sigma_g(\mathbf{p}(u), \mathbf{p}(v))$ for all $\{u,v\} \in E$, where $\perp$ denotes projection onto the normal bundle of $\sigma_g(\mathbf{p}(u), \mathbf{p}(v))$;
		\item[(2'')] $J(\mathbf{p}(u)) = J_{u,v}(\mathbf{p}(u))$ for all $\{u,v\} \in E$ with $\deg_G u \neq 2$;
		\item[(3'')] $J(\mathbf{p}(u_0)) = J_{u_0,u_1}(\mathbf{p}(u_0))$ on components as in Definition \ref{defi:geodesic.network.subdivision}'s (S$_2$).
	\end{enumerate}
\end{lemm}
\begin{proof}
	We start with assuming (J$_1$'), (J$_2$'), (J$_3$'). Obviously, (J$_1$') implies Definition \ref{defi:geodesic.network.jacobi.field.geom}'s (J$_1$) for the concatenation $J$ of $(J_{u,v})_{\{u,v\} \in E}$. Moreover, (J$_2$') implies that $J$ is $C^1$ on $\reg S$, and thus at least $C^2$ on $\reg S$ by relying on ODE existence and uniqueness applied to the Jacobi equation. Finally, it is clear that (J$_3$') implies (J$_3$). This completes the proof of (1').
	
	Now assume the weaker hypotheses (J$_1$'), (J$_3$'). Note that since the discrepancy in (J$_2$') is due to a tangential term, it is easy to find everywhere-tangential Jacobi fields $\{ \tilde J_{u,v} : \{u,v\} \in E \}$ so that $\tilde J_{u,v} = \tilde J_{v,u}$ and
	\begin{itemize}
		\item $\{ \tilde J_{u,v} : \{ u,v \} \in E \}$ satisfies (J$_1$'),
		\item $\tilde J_{u,v}(\mathbf{p}(u)) = 0$ whenever $u \in V$ has $\deg_G u \neq 2$,
		\item $\tilde J_{u_0, u_1}(\mathbf{p}(u_0)) = 0$ on all components as in Definition \ref{defi:geodesic.network.subdivision}'s (S$_2$),
		\item $\{ J_{u,v} + \tilde J_{u,v} : \{u,v\} \in E \}$ satisfies (J$_2$').
	\end{itemize}
	Then, $\{ J_{u,v} + \tilde J_{u,v} : \{u,v\} \in E \}$ clearly satisfies (J$_1$'), (J$_2$'), (J$_3$'), and the lemma follows by applying the previous case to $\{ J_{u,v} + \tilde J_{u,v} : \{u,v\} \in E \}$.
\end{proof}

Our work culminates in the following trichotomy theorem:

\begin{theo}[Trichotomy theorem] \label{theo:geodesic.network.trichotomy}
	Suppose $g \in \met^k(M)$ and $\Lambda > 0$. Then, at least one of the following is true:
	\begin{enumerate}
		\item there exists a not-everywhere-tangential stationary varifold Jacobi field along some $S \in \cS^\Lambda(g)$, or
		\item $(M, g)$ contains a sequence of simple closed geodesics of length $< \Lambda$ converging to another such with multiplicity $\geq 2$, or
		\item for every $0 < \Lambda' < \Lambda$, the compact subset $\bar \cS^{\Lambda'}(g) \subset \cS^{\Lambda}(g)$ has no $N$-th iterated limit points, where $N = N(\inj(M, g), \Lambda)$.
	\end{enumerate}
\end{theo}

Note that alternative (2) can only hold if $M$ is non-orientable (e.g., $\RR P^2$). When it does hold, it only does so non-generically by \cite{White:bumpy.new} (we will not need this, however).

This theorem will be a consequence of the following technical lemma:

\begin{lemm} \label{lemm:geodesic.network.trichotomy}
	Suppose $g$, $\Lambda$, $\Lambda'$ are as in Theorem \ref{theo:geodesic.network.trichotomy}, $G$ is a $Q$-subdivided graph structure with a $Q \in \NN$ such that \eqref{eq:geodesic.network.finite.stratification.Q} holds with $g$, $\Lambda$, $Q$, and that $\{ \mathbf{p}_i \}_{i=1}^\infty \subset \cB \emb_g(G,M)$ satisfy 
	\begin{equation} \label{eq:geodesic.network.trichotomy.distinct}
		\iota_g(G, \mathbf{p}_i) \neq \iota_g(G, \mathbf{p}_j) \text{ whenever } i \neq j, \text{ and}
	\end{equation}
	\begin{equation} \label{eq:geodesic.network.trichotomy.convergence}
		\bar \cS^{\Lambda'}(g) \ni \iota_{g}(G, \mathbf{p}_i) \rightharpoonup S_\infty \text{ as } i \to \infty.
	\end{equation}
	Then, at least one of the following is true:
	\begin{enumerate}
		\item $g$ is a singular value of $\pi^{\Lambda}_G$, or
		\item there are closed and connected $C_i \subset \reg \iota_g(G, \mathbf{p}_i)$ subsequentially converging (as $i \to \infty$) with multiplicity $\geq 2$, or
		\item $j^{\Lambda}_Q(S_\infty) =: (G_\infty, \mathbf{p}_\infty)$ has $\# V(G_\infty) < \# V(G)$.
	\end{enumerate}
\end{lemm}
\begin{proof}
	We know from Proposition \ref{prop:geodesic.network.stationary.varifold.limit} that \eqref{eq:geodesic.network.trichotomy.convergence} implies
	\begin{equation} \label{eq:geodesic.network.trichotomy.supp.set}
		\supp S_\infty = \lim_i \supp S_i
	\end{equation}
	and
	\begin{equation} \label{eq:geodesic.network.trichotomy.sing.set}
		\sing S_\infty \subset \lim_i \sing S_i
	\end{equation}
	in the Hausdorff sense. These readily imply that
	\begin{equation} \label{eq:geodesic.network.trichotomy.sing}
		\# \sing S_\infty \leq \# \sing S_i \text{ for large } i,
	\end{equation}
	and that connected components of $\supp S_i$ and $\reg S_i$ do not disconnect as $i \to \infty$. Therefore, 
	\begin{equation} \label{eq:geodesic.network.trichotomy.reg}
		\# \{ \text{components of } \reg S_\infty \} \leq \# \{ \text{components of } \reg S_i \} \text{ for large } i,
	\end{equation}	
	and the convergence of components comes in three forms:
	\begin{itemize}	
		\item[(a)] A singular component of $\supp S_i$ converging to a smooth component of $\supp S_\infty$ Let $S_i'$ be the restriction of our varifold $S_i$ to the singular component we're studying, and let $G_i' \subset G$ be the corresponding graph component. By passing to a subsequence, we may assume that $G_i'$ is independent of $i$ and label it $G'$. Let $G_\infty'$ be the corresponding connected component of $G_\infty$. By Lemma \ref{lemm:geodesic.network.embedding.structure}'s (2)(a), (3)(a):
			\begin{align} \label{eq:geodesic.network.trichotomy.1.v}
				\# V(G') - \# V(G_\infty') 
					& = \# \sing S_i' - 1  \\
					& \qquad + Q \cdot (\# \{ \text{components of } \reg S_i' \} - 1) \nonumber \\
					& \geq 0 + Q \cdot 1 = Q, \nonumber
			\end{align}
			where we used the fact that $\reg S_i'$ has to have $\geq 2$ components.
		\item[(b)] A smooth component of $\supp S_i$ converging to a smooth component of $\supp S_\infty$. If $G' \subset G$, $G_\infty' \subset G_\infty$ are defined analogously, then by Lemma \ref{lemm:geodesic.network.embedding.structure}'s (3)(a):
			\begin{equation} \label{eq:geodesic.network.trichotomy.2.v}
				\# V(G') = \# V(G_\infty').
			\end{equation}
		\item[(c)] A singular component of $\supp S_i$ converging to a singular component of $\supp S_\infty$. If $G' \subset G$, $G_\infty' \subset G_\infty$ are defined analogously, then by Lemma \ref{lemm:geodesic.network.embedding.structure}'s (2)(a):
			\begin{align} \label{eq:geodesic.network.trichotomy.3.v}
				\# V(G') - \# V(G_\infty') 
					& = \# \sing S_i' - \# \sing S_\infty'  \\
					& \qquad + Q \cdot (\# \{ \text{components of } \reg S_i' \} \nonumber \\
					& \qquad \qquad - \# \{ \text{components of } \reg S_\infty' \}) \geq 0, \nonumber
			\end{align}
			where the inequality follows as with \eqref{eq:geodesic.network.trichotomy.sing}, \eqref{eq:geodesic.network.trichotomy.reg}, since Proposition \ref{prop:geodesic.network.stationary.varifold.limit} applies to components as well.
	\end{itemize}
	It follows from \eqref{eq:geodesic.network.trichotomy.sing}, \eqref{eq:geodesic.network.trichotomy.reg}, \eqref{eq:geodesic.network.trichotomy.1.v}, \eqref{eq:geodesic.network.trichotomy.2.v}, \eqref{eq:geodesic.network.trichotomy.3.v}, and Lemma \ref{lemm:geodesic.network.embedding.structure}'s (2)(a), (3)(a), that
	\[ \# V(G_\infty) \leq \# V(G). \]
	Let us assume that alternatives (2) and (3) of the lemma both fail. Then, 
	\begin{equation} \label{eq:geodesic.network.trichotomy.equality}
		\# V(G_\infty) = \# V(G)
	\end{equation}
	and adding \eqref{eq:geodesic.network.trichotomy.1.v}, \eqref{eq:geodesic.network.trichotomy.2.v}, \eqref{eq:geodesic.network.trichotomy.3.v} over all converging components, we deduce:
	\begin{itemize}
		\item Case (a) never occurs and case (b) never occurs  more than once for each cycle $G_\infty' \subset G_\infty$ (or we would have had a vertex drop of $\geq Q+1$ elements in the limit, contradicting  \eqref{eq:geodesic.network.trichotomy.3.v}, \eqref{eq:geodesic.network.trichotomy.equality}).
		\item Consequently, each smooth component of $\supp S_\infty$ must be the limit of precisely one smooth component of $\supp S_i$ and the convergence of supports holds with multiplicity one because of the failure of (2).
		\item Case (c) must always occur with equality in \eqref{eq:geodesic.network.trichotomy.3.v}. Thus, each singular component of $\supp S_\infty$ must be the limit of precisely one singular component of $\supp S_i$, and their singular points and regular segments must be in bijection and thus converge with multiplicity one.
	\end{itemize}
	After passing to a subsequence (not relabeled), define
	\[ \bar{\mathbf{p}}_\infty := \lim_i \mathbf{p}_i \in M^{V(G)}. \]
	\begin{claim}
		$\bar{\mathbf{p}}_\infty \in \cB \emb_g(G, M)$ and $\iota_g(G, \bar{\mathbf{p}}_\infty) = \iota_g(G_\infty, \mathbf{p}_\infty)$.
	\end{claim}
	\begin{proof}[Proof of claim]
		Definition \ref{defi:geodesic.network.immersion}'s (I$_1$) and Definition \ref{defi:geodesic.network.embedding}'s (E$_2$) holds because singular points converge to singular points and regular points converge to regular points, all with multiplicity-one. Definition \ref{defi:geodesic.network.immersion}'s (I$_2$) and Definition \ref{defi:geodesic.network.embedding}'s (E$_3$) hold as in Step 1 of the proof of Theorem \ref{theo:geodesic.network.finite.stratification}. This completes the proof that, $\bar{\mathbf{p}}_\infty \in \cB \emb_g(G, M)$. 
		
		By construction, the singular and regular parts of $\iota_g(G, \bar{\mathbf{p}}_\infty)$, $\iota_g(G_\infty, \mathbf{p}_\infty)$ coincide. The varifold densities coincide from the multiplicity-one convergence and Lemma \ref{lemm:geodesic.network.embedding.structure}'s (2)(d) and (3)(d). This completes the proof.
	\end{proof}

	As a consequence of the claim and Corollary \ref{coro:geodesic.network.embedding.structure}'s (1), we may redefine
	\[ G_\infty := G \text{ and  } \mathbf{p}_\infty := \bar{\mathbf{p}}_\infty. \]
	Next, we redefine $\mathbf{p}_i$ on cyclic components $G' \subset G$ as in Definition \ref{defi:geodesic.network.subdivision}'s (S$_2$). In the notation $V(G') = \{ u_0,  \ldots, u_Q \}$, redefine $\mathbf{p}_i(u_0)$ as being the unique point satisfying
	\begin{equation} \label{eq:geodesic.network.stratification.loop.redefined}
		\mathbf{p}_i(u_0) \in \supp S_i \cap N_{\mathbf{p}_\infty(u_0)} \supp S_\infty
	\end{equation}
	while being $o(1)$ close to $\mathbf{p}_\infty(u_0)$ as $i \to \infty$. (Here, $N$ denotes the normal line, locally.) Then, redefine $\mathbf{p}_i(u_j)$, $j \in \{ 1, \ldots, Q \}$, as being the corresponding equidistant points along the component of $\supp S_i$ traced by $G'$. Note that $\mathbf{p}_i \in \cB \emb_g(G, M)$ still and that $\iota_g(G, \mathbf{p}_i)$ is unchanged.
	
	We can finally proceed to construct a nontrivial Jacobi field along $S_\infty$. For all sufficiently large $i = 1, 2, \ldots$, $\{ u, v \} \in E$, there exists a unique Jacobi field $J^{(i)}_{u,v}$ along $\sigma_g(\mathbf{p}_\infty(u), \mathbf{p}_\infty(v))$ with the boundary conditions
	\begin{equation} \label{eq:geodesic.network.finite.stratification.jf.bc}
		\begin{gathered}
		J^{(i)}_{u,v}(\mathbf{p}_\infty(u)) = (\exp^g_{\mathbf{p}_\infty(u)})^{-1} (\mathbf{p}_i(u)), \\
		J^{(i)}_{u,v}(\mathbf{p}_\infty(v)) = (\exp^g_{\mathbf{p}_\infty(v)})^{-1} (\mathbf{p}_i(v)),
		\end{gathered}
	\end{equation}
	and further satisfies, for a uniform $c_0 \in \RR$,
	\begin{equation} \label{eq:geodesic.network.finite.stratification.jf.bound}
		\Vert J^{(i)}_{u,v} \Vert_{C^3} \leq c_0 \max_{z \in \{u,v\}} \dist_g(\mathbf{p}_i(z), \mathbf{p}_\infty(z)),
	\end{equation}
	in view of \eqref{eq:geodesic.network.finite.stratification.Q}, which guarantees we are working strictly away from $\inj(M,g)$. Moreover, by elementary Jacobi field analysis we know that the unit tangent vectors
	\begin{itemize}
		\item $\tau^{(i)}_{u,v}$ along $\sigma_g(\mathbf{p}_i(u), \mathbf{p}_i(v))$ oriented from $\mathbf{p}_i(u)$ to $\mathbf{p}_i(v)$, and
		\item $\tau^{(\infty)}_{u,v}$ along $\sigma_g(\mathbf{p}_\infty(u), \mathbf{p}_\infty(v))$ oriented from $\mathbf{p}_\infty(u)$ to $\mathbf{p}_\infty(v)$,
	\end{itemize}
	satisfy
	\begin{multline} \label{eq:geodesic.network.finite.stratification.jf.tangent}
		d (\exp^g_{\mathbf{p}_\infty(u)})^{-1} \tau^{(i)}_{u,v}(\mathbf{p}_i(u)) - \tau^{(\infty)}_{u,v}(\mathbf{p}_\infty(u)) \\
		= \nabla_{u,v}^\perp J^{(i)}_{u,v}(\mathbf{p}_\infty(u)) + O(\Vert J^{(i)}_{u,v} \Vert_{C^2}^2),
	\end{multline}
	where $\nabla^\perp_{u,v}$ is the unit speed covariant derivative along $\sigma_g(\mathbf{p}_\infty(u), \mathbf{p}_\infty(v))$, oriented as before, and projected to the normal bundle. Now define
	\[ \lambda_i := \max_{z \in V} \dist_g(\mathbf{p}_i(z), \mathbf{p}_\infty(z)). \]
	It is straightforward to show that
	\[ c_1^{-1} \lambda_i \leq \dist_{\mathbf{F}}(\iota_g(G, \mathbf{p}_i), \iota_g(G_\infty, \mathbf{p}_\infty)) \leq c_1 \lambda_i \]
	for a uniform $c_1 \in \RR$, so, by \eqref{eq:geodesic.network.trichotomy.distinct} and \eqref{eq:geodesic.network.trichotomy.convergence},
	\[ 0 < \lambda_i \to 0 \text{ as } i \to \infty. \]
	From \eqref{eq:geodesic.network.finite.stratification.jf.bc}, \eqref{eq:geodesic.network.finite.stratification.jf.bound}, and Arzel\`a--Ascoli, after passing to a subsequence,
	\[ \lambda_i^{-1} J^{(i)}_{u,v} \to J_{u,v} \text{ in } C^2, \]
	where $J_{u,v}$ is a Jacobi field along $\sigma_g(\mathbf{p}_\infty(u), \mathbf{p}_\infty(v))$.
	
	We will complete our proof of this lemma by showing that $\{ J_{u,v} :  \{u,v\} \in E \}$ concatenate to a stationary varifold Jacobi field via Lemma \ref{lemm:geodesic.network.jacobi.field}. We have
	\begin{equation} \label{eq:geodesic.network.finite.stratification.jf.c0}
		J_{u,v}(\mathbf{p}_\infty(u)) = J_{u,v'}(\mathbf{p}_\infty(u)) \text{ whenever } v, v' \in E_u
	\end{equation}
	by \eqref{eq:geodesic.network.finite.stratification.jf.bc}. This implies Lemma  \ref{lemm:geodesic.network.jacobi.field}'s (J$_1$'). Next, the stationarity of $S_i$ and $S_\infty$ implies that
	\begin{align*}
		0 & = \sum_{v \in E_u} \omega_i(\{u,v\}) \tau^{(i)}_{u,v}(\mathbf{p}_i(u)) \\
		0 & = \sum_{v \in E_u} \omega_\infty(\{u,v\}) \tau^{(\infty)}_{u,v}(\mathbf{p}_\infty(u)).
	\end{align*}
	Apply $d (\exp^g_{\mathbf{p}_\infty(u)})^{-1}$ to the first equation and then subtract the second from it, using $\omega_i = \omega_\infty$ from the convergence. Invoke \eqref{eq:geodesic.network.finite.stratification.jf.tangent}, divide by $\lambda_i$, and send $i \to \infty$ to deduce
	\begin{equation} \label{eq:geodesic.network.finite.stratification.jf.c1.normal}
		0 = \sum_{v \in E_u} \omega_\infty(\{u,v\}) \nabla_{u,v}^\perp J_{u,v}(\mathbf{p}_\infty(u)).
	\end{equation}
	This is Lemma \ref{lemm:geodesic.network.jacobi.field}'s (J$_3$'). The fact that, whenever $\deg_G u = 2$,
	\begin{equation} \label{eq:geodesic.network.finite.stratification.jf.c1.tangential}
		\sum_{v \in E_u} \nabla_{u,v}^T J_{u,v}(\mathbf{p}_\infty(u)) = 0
	\end{equation}
	follows from that tangential derivatives of Jacobi fields measure infinitesimal changes in length, and Definition \ref{defi:geodesic.network.embedding}'s (E$_3$) balancing applying to each of $\mathbf{p}_i$, $\mathbf{p}_\infty$. This gives Lemma \ref{lemm:geodesic.network.jacobi.field}'s (J$_2$'), completing our proof that $\{ J_{u,v} : \{u,v\}\in E \}$ concatenates to a stationary varifold Jacobi field $J$ along $S$.
	
	It remains to prove that $J$, or equivalently $\{ J_{u,v} : \{u,v\}\in E \}$, isn't everywhere tangential. Suppose, for the sake of contradiction, that
	\begin{equation} \label{eq:geodesic.network.finite.stratification.jf.tangent.everywhere}
		J^\perp \equiv J_{u,v}^\perp \equiv 0 \text{ along } \sigma_g(\mathbf{p}(u), \mathbf{p}(v)) \text{ for all } \{u,v\} \in E,
	\end{equation}
	where $\perp$ is as before. Take $u \in V$ with $\deg_G u \neq 2$. Then $\mathbf{p}_\infty(u) \in \sing S_\infty$ by Lemma \ref{lemm:geodesic.network.embedding.structure}'s (1)(b), so $\operatorname{span}\{ \nu_{u,v}(\mathbf{p}_\infty(u)) : v \in E_u \} = \Tan_{\mathbf{p}_\infty(u)} M$ by the stationarity condition of $S_\infty$ evaluated at $\mathbf{p}_\infty(u)$. Thus, \eqref{eq:geodesic.network.finite.stratification.jf.tangent.everywhere} forces:
	\[ J_{u,v}(\mathbf{p}_\infty(u)) = 0 \text{ for all } u \in V, \; \deg_G u \neq 2, \; v \in E_u, \]
	\[ \implies J(\mathbf{p}(u)) = 0 \text{ for all } u \in V, \; \deg_G u \neq 2. \]
	Together with the known fact that the tangential portion of Jacobi fields is a linear function and Lemma \ref{lemm:geodesic.network.embedding.structure}'s (2), this implies
	\begin{equation} \label{eq:geodesic.network.finite.stratification.jf.zero.1}
		J \equiv 0 \text{ on all non-closed components of } \reg S.
	\end{equation}
	
	Now take $u = u_0$ from a component $G'' \subset G$ as in Definition \ref{defi:geodesic.network.subdivision}'s (S$_2$) and  \eqref{eq:geodesic.network.stratification.loop.redefined} above. By construction, $J^T_{u_0,v}(\mathbf{p}_\infty(u_0)) = 0$ for both $v \in E_{u_0}$, and thus $J(\mathbf{p}(u_0)) = 0$ by \eqref{eq:geodesic.network.finite.stratification.jf.tangent.everywhere}. Reusing the fact that the tangential portion of Jacobi fields is a linear function, now along the geodesic loop around $u_0$:
	\begin{equation} \label{eq:geodesic.network.finite.stratification.jf.zero.2}
		J \equiv 0 \text{ on all closed components of } \reg S.
	\end{equation}
	Together, \eqref{eq:geodesic.network.finite.stratification.jf.zero.1} and  \eqref{eq:geodesic.network.finite.stratification.jf.zero.2} imply that $J \equiv 0$. This contradicts the definition of $\lambda_i$ and \eqref{eq:geodesic.network.finite.stratification.jf.bc}. Thus, $J$ isn't everywhere tangential, as claimed. The fact that $g$ is a singular value of $\pi^\Lambda_G$ now follows from Theorem \ref{theo:geodesic.network.manifold}.
\end{proof}

\begin{proof}[Proof of Theorem \ref{theo:geodesic.network.trichotomy}]
	Assume that alternatives (1) and (2) both fail. Take $Q \in \NN$ so that \eqref{eq:geodesic.network.finite.stratification.Q} holds with $g$, $\Lambda$, $Q$, and set $N$ to be the maximum number of vertices among all graphs in $\cG(\Lambda, Q)$, which is of the desired form due to \eqref{eq:geodesic.network.stratification.size.bound}. 
	
	We claim that, for every $k \in \{ 0, 1, \ldots, N \}$,
	\begin{equation} \label{eq:geodesic.network.trichotomy.induction}
		S \in \operatorname{Lim}^{(k)}(\bar \cS^{\Lambda'}(g)), \; j^{\Lambda}_Q(S) =: (G, \mathbf{p}) \implies \# V(G) \leq N-k,
	\end{equation}
	where $\operatorname{Lim}^{(k)}(\bar \cS^{\Lambda'}(g))$ denotes the set of $k$-th iterated limit points of $\bar \cS^{\Lambda'}(g)$. (By definition, $\operatorname{Lim}^{(0)}(\bar \cS^{\Lambda'}(g)) = \bar \cS^{\Lambda'}(g)$.) This will complete the proof. 
	
	We prove \eqref{eq:geodesic.network.trichotomy.induction} by induction on $k$. The base case ($k=0$) is trivial from our definition of $N$ and Theorem \ref{theo:geodesic.network.manifold}. So we assume \eqref{eq:geodesic.network.trichotomy.induction} holds for some $k \in \{ 0, 1, \ldots, N-1\}$, and we prove it for $k+1$. Pick $S_\infty \in \operatorname{Lim}^{(k+1)}(\bar \cS^{\Lambda'}(g))$, which by definition means that there exist $S_i \in \operatorname{Lim}^{(k)}(\bar \cS^{\Lambda'}(g))$ with $S_i \rightharpoonup S_\infty$. Suppose $j^{\Lambda}_Q(S_i) =: (G_i, \mathbf{p}_i)$. Since $\{ G_i \}_{i=1}^{\infty} \subset \cG(\Lambda, Q)$ is finite, we pass to a subsequence (not relabeled) along which $G_i \equiv G \in \cG$. By \eqref{eq:geodesic.network.trichotomy.induction},
	\begin{equation} \label{eq:geodesic.network.trichotomy.induction.implies}
		\# V(G) \leq N-k.
	\end{equation}
	We now apply Lemma \ref{lemm:geodesic.network.trichotomy}. Since the first two alternatives of the lemma are assumed false (the first by way of Theorem \ref{theo:geodesic.network.manifold}'s (3)), Lemma \ref{lemm:geodesic.network.trichotomy} implies that, for $j^{\Lambda}_Q(S_\infty) =: (G_\infty, \mathbf{p}_\infty)$,
	\[ \# V(G_\infty) \leq \# V(G) - 1 \leq N-k-1, \]
	where the last inequality used was \eqref{eq:geodesic.network.trichotomy.induction.implies}. The induction is complete and the theorem follows.
\end{proof}

\subsection{Proof of Theorem \ref{theo:geodesic.network.manifold}} 

Fix $(g_0, \iota_g(G, \mathbf{p}_0)) \in \cS^\Lambda_G$ with $\mathbf{p}_0 \in \cB \emb_g(G, M)$, as we allowed to do by Theorem \ref{theo:geodesic.network.finite.stratification}.
	
Using the continuous dependence of the injectivity radius and the exponential map on $\met^k(M)$ for $k \geq 2$ (recall, we have $k \geq 3$), choose a smooth background metric $\mathfrak{g}$ on $M$, $\mathfrak{p} := \mathbf{p}_0$, and an $\eps > 0$ so that:
\begin{equation} \label{eq:geodesic.network.manifold.small.frakg}
	\mathfrak{g} \in \met^k(M, g_0, \tfrac12 \eps),
\end{equation}
\begin{equation} \label{eq:geodesic.network.manifold.small.inj}
	\eps < \inj(M, \mathfrak{g}) \text{ and } \eps < \min_{\substack{u' \in V(G')\\u'' \in V(G'')}} \dist_{\mathfrak{g}}(\mathfrak{p}(u'), \mathfrak{p}(u'')),
\end{equation}
where $G'$, $G''$ range over all distinct connected components of $G$, and finally
\begin{multline} \label{eq:geodesic.network.manifold.small.imm}
	g \in \met^k(M, \mathfrak{g}, \eps), \; \mathbf{p} \in B^{\mathfrak{g}}_{\eps}(\mathfrak{p}) \\
	\implies \mathbf{p} \in \emb_g(G, M) \text{ and } \\
	\# (V \setminus V_{\reg}) + \Vert \iota_g(G, \mathbf{p}) \Vert(M, g) < \Lambda,
\end{multline}
where $V_{\reg} := \{ u \in V : \deg_G u = 2 \}$, $B^{\mathfrak{g}}_\eps(\mathfrak{p}) := \times_{u \in V} B^{\mathfrak{g}}_\eps(\mathfrak{p}(u))$. Now define
\[ \cL^0 : \met^k(M, \mathfrak{g}, \eps) \times B^{\mathfrak{g}}_{\eps}(\mathfrak{p}) \to \RR \]
as follows: 
\[ \cL^0(g, \mathbf{p}) := \Vert \iota_g(G, \mathbf{p}) \Vert(M, g) = \sum_{\{u,v\} \in E} \omega(\{u,v\}) \length_g(\sigma_g(\mathbf{p}(u), \mathbf{p}(v))). \]
	
Notice that $\cL^0$ is well-defined by \eqref{eq:geodesic.network.manifold.small.imm}. At the moment, the domain of $\cL^0$ is a $C^k$ Banach manifold. We will prefer to work with functionals on open subsets of Banach spaces, so we seek to replace $B^{\mathfrak{g}}_\eps(\mathfrak{p})$ with
\[ B_\eps(\mathbf{0}) := \times_{u \in V} (B_\eps(0) \subset \Tan_{\mathfrak{p}(u)} M). \]
To do so, we will use the $C^\infty$ exponential map
\[ \mathfrak{e}(\mathbf{p})(u) := \exp^{\mathfrak{g}}_{\mathfrak{p}(u)} \mathbf{p}(u), \; \mathbf{p} \in B_\eps(\mathbf{0}), \; u \in V, \]
which is a $C^\infty$ diffeomorphism $\mathfrak{e} : B_\eps(\mathbf{0}) \to B^{\mathfrak{g}}_\eps(\mathfrak{p})$ in view of \eqref{eq:geodesic.network.manifold.small.inj}. We often occasionally write $\mathfrak{e} \mathbf{p}$ in place of $\mathfrak{e}(\mathbf{p})$, just like one normally doesn't use parentheses in the exponential map. This map allows us to work with:
\[ \cL : \met^k(M, \mathfrak{g}, \eps) \times B_\eps(\mathbf{0}) \to \RR, \; \cL(g, \mathbf{p}) := \cL^0(g, \mathfrak{e}\mathbf{p}). \]

\begin{lemm} \label{lemm:geodesic.network.manifold.L.regularity}
	$\cL : \met^k(M, \mathfrak{g}, \eps) \times B_\eps(\mathbf{0}) \to \RR$ is a $C^{k-1}$ Banach map.
\end{lemm}
\begin{proof}
	This is clear from the definition of $\cL$ and the dependence on $\sigma_g$, which requires one derivative on $g$ due to the Christoffel symbols.
\end{proof}

We compute the rate of change of $\cL(g, \mathbf{p})$ in $\mathbf{p}$. Write $Y$ for the direct sum of tangent spaces that $B_\eps(\mathbf{0})$ is inside of. Fix $g \in \met^k(M, \mathfrak{g}, \eps)$, $\mathbf{p} \in B_\eps(\mathbf{0})$, and $\mathbf{q} \in Y$. The first variation formula in Riemannian geometry gives
\begin{equation} \label{eq:geodesic.network.manifold.first.variation.p}
	\Big[ \tfrac{d}{dt} \cL(g, \mathbf{p} + t \mathbf{q}) \Big]_{t=0} = - \sum_{u \in V} \langle \sum_{v \in E_u} \omega(\{u,v\}) \tau^{g}_{u,v}(\mathfrak{e}\mathbf{p}(u)), \mathfrak{e}\mathbf{p}(u)_* \big[ \mathbf{q}(u) \big] \rangle_{g}
\end{equation}
where:
\begin{itemize}
	\item $\mathfrak{e}\mathbf{p}(u)_* = d(\exp^{\mathfrak{g}}_{\mathfrak{p}(u)})_{\mathbf{p}(u)}$, and 
	\item $\tau^{g}_{u,v}$ is the unit tangent vector along $\sigma_g(\mathfrak{e}\mathbf{p}(u), \mathfrak{e}\mathbf{p}(v))$, taken with respect to $g$ and oriented from $\mathfrak{e}\mathbf{p}(u)$ to $\mathfrak{e}\mathbf{p}(v)$.
\end{itemize}
The ``formal gradient'' of $\cL(g, \mathbf{p})$ with respect to $\mathbf{p}$ is the unique function
\begin{equation} \label{eq:geodesic.network.manifold.H}
	\cH : \met^k(M, \mathfrak{g}, \eps) \times B_{\eps}(\mathbf{0}) \to Y
\end{equation}
determined by
\begin{equation} \label{eq:geodesic.network.manifold.H.variational}
	\Big[ \tfrac{d}{dt} \cL(g, \mathbf{p} + t \mathbf{q}) \Big]_{t=0} =: \langle \cH(g, \mathbf{p}), \mathbf{q} \rangle_{\mathfrak{e}\mathbf{p}^* g};
\end{equation}
the right hand side's inner product is that induced on $Y$ by the embedding $\mathbf{p}$ and $g$, i.e. the right hand side is short for
\[ \sum_{u \in V} \langle \mathfrak{e}\mathbf{p}(u)_* \big[ \cH(g, \mathbf{p})(u) \big], \mathfrak{e}\mathbf{p}(u)_* \big[ \mathbf{q}(u) \big] \rangle_g. \]
In any case, \eqref{eq:geodesic.network.manifold.first.variation.p} implies
\begin{equation} \label{eq:geodesic.network.manifold.H}
	\mathfrak{e}\mathbf{p}(u)_* \big[ \cH(g, \mathbf{p})(u) \big] = - \sum_{v \in E_u} \omega(\{u,v\}) \tau^g_{u,v}(\mathfrak{e}\mathbf{p}(u)).
\end{equation}
The following justifies our introduction of the functional $\cL$:

\begin{lemm} \label{lemm:geodesic.network.manifold.equivalence}
	For $(g, \mathbf{p}) \in \met^k(M, \mathfrak{g}, \eps) \times B_\eps(\mathbf{0})$:
	\begin{equation} \label{eq:geodesic.network.manifold.H.equivalence}
		\iota_g(G, \mathfrak{e} \mathbf{p}) \text{ is } g\text{-stationary} \iff \cH(g, \mathbf{p}) = \mathbf{0}.
	\end{equation}
\end{lemm}
\begin{proof}
	Let $u \in V$. By Lemma \ref{lemm:geodesic.network.embedding.structure} (1)(a), there exists a ball around $\mathfrak{e} \mathbf{p}(u)$ so that the only segments entering the ball are $\sigma_g(\mathfrak{e} \mathbf{p}(u), \mathfrak{e} \mathbf{p}(v))$, $v \in E_u$. Then, the stationarity condition (\cite[(1)]{AllardAlmgren:1varifold}) shows
	\begin{equation} \label{eq:geodesic.network.manifold.stationarity}
		\cH(g, \mathbf{p})(u) = 0 \iff \delta(\iota_g(G, \mathfrak{e} \mathbf{p}))|_{\mathfrak{e} \mathbf{p}(u)} = 0
	\end{equation}
	by way of \eqref{eq:geodesic.network.manifold.H}. Since the segments joining the $\mathfrak{e}\mathbf{p}(u)$ are $g$-geodesics, it follows that the left hand side of \eqref{eq:geodesic.network.manifold.stationarity} is true for all $u \in V$ if and only if the right hand side vanishes throughout $M$.
\end{proof}

In order to make our definitions more useful toward establishing Theorem \ref{theo:geodesic.network.manifold}, we need to eliminate some of the gauge freedom in them. Every $u \in V_{\reg}$ has $E_u = \{ n_1(u), n_2(u) \}$. We distinguish these two neighbors by introducing a total order $<$ on $V$ and requiring that $n_1(u) < n_2(u)$. For all $u \in V_{\reg}$ we fix
\begin{multline*}
	\mathring{B}^{\mathfrak{g}}_\eps(\mathfrak{p}(u)) := \text{1-dimensional } \mathfrak{g}\text{-geodesic segment of } \mathfrak{g} \text{-length } 2 \eps \\
	\text{ centered at } \mathfrak{p}(u) \text{ and transverse to } \sigma_{\mathfrak{g}}(\mathfrak{p}(n_1(u)), \mathfrak{p}(n_2(u))),
\end{multline*}
e.g., the normal to $\sigma_{\mathfrak{g}}(\mathfrak{p}(n_1(u)), \mathfrak{p}(n_2(u)))$ at $\mathfrak{p}(u)$ intersected with $B^{\mathfrak{g}}_\eps(\mathfrak{p}(u))$, then we can require that $\eps > 0$ is also small enough that, in addition to \eqref{eq:geodesic.network.manifold.small.frakg}, \eqref{eq:geodesic.network.manifold.small.inj}, \eqref{eq:geodesic.network.manifold.small.imm}, we also have
\begin{multline} \label{eq:geodesic.network.manifold.small.transverse}
	g \in \met^k(M, \mathfrak{g}, \eps), \; \mathbf{p} \in B^{\mathfrak{g}}_\eps(\mathfrak{p}), \; u \in V_{\reg}, \; \mathbf{p}(u) \in \mathring{B}^{\mathfrak{g}}_\eps(\mathfrak{p}(u)) \\
	\implies \sigma_g(\mathfrak{e} \mathbf{p}(u), \mathfrak{e} \mathbf{p}(n_1(u))), \; \sigma_g(\mathfrak{e} \mathbf{p}(u), \mathfrak{e} \mathbf{p}(n_2(u))) \text{ are transverse to } \mathring{B}^{\mathfrak{g}}_\eps(\mathfrak{p}(u)).
\end{multline}
Then, we write
\[ \mathring{B}^{\mathfrak{g}}_\eps(\mathfrak{p}) := \Big( \times_{u \in V_{\reg}} \mathring{B}^{\mathfrak{g}}_{\eps}(\mathfrak{p}(u)) \Big) \times \Big( \times_{u \in V \setminus V_{\reg}} B^{\mathfrak{g}}_{\eps}(\mathfrak{p}(u)) \Big) \subset B^{\mathfrak{g}}_\eps(\mathfrak{p}), \]
and analogously define 
\[ \mathring{B}_\eps(\mathbf{0}) := (\times_{u \in V_{\reg}} \mathring{B}_\eps(0) \subset \Tan_{\mathfrak{p}(u)} M) \times ( \times_{u \in V \setminus V_{\reg}} B_\eps(0) \subset \Tan_{\mathfrak{p}(u)} M), \]
where the $\mathring{B}_\eps(0)$'s are simply pullbacks of $\mathring{B}^{\mathfrak{g}}_\eps(\mathfrak{p}(u))$, $u \in V_{\reg}$, under $\mathfrak{e}$.

Obviously, we may restrict $\cL$ and $\cH$ to the subset $\met^k(M, \mathfrak{g}, \eps) \times \mathring{B}_\eps(\mathbf{0})$, but it will also be important to restrict the target space of $\cH$ to the Banach subspace $\mathring{Y} \subset Y$ inside of which $\mathring{B}_\eps(\mathbf{0})$ lies. We proceed to define
\[ \mathring{\cH} : \met^k(M, \mathfrak{g}, \eps) \times \mathring{B}_\eps(\mathbf{0}) \to \mathring{Y} \]
as 
\begin{equation} \label{eq:geodesic.network.manifold.H.breve}
	\mathring{\cH}(g, \mathbf{p})(u) := \begin{cases}
		\cH(g, \mathbf{p})(u) & \text{ if } u \in V \setminus V_{\reg}, \\
		\proj^{\mathfrak{e} \mathbf{p}(u)^* g}_{\mathring{B}_\eps(0)}  \cH(g, \mathbf{p})(u) & \text{ if } u \in V_{\reg}. 
	\end{cases}
\end{equation}
We emphasize that $\proj^{\mathfrak{e}\mathbf{p}^* g}_{\mathring{B}_\eps(0)}$ projects onto the fixed $\mathring{B}_\eps(0)$ using the variable metric $g$. Our definitions of $\mathring{Y}$ and $\mathring{\cH}$ imply that, for all $\mathbf{q} \in \mathring{Y}$:
\begin{equation} \label{eq:geodesic.network.manifold.H.breve.projection}
	\langle \mathring{\cH}(g, \mathbf{p}), \mathbf{q} \rangle_{\mathfrak{e}\mathbf{p}^* g} = \langle \cH(g, \mathbf{p}), \mathbf{q} \rangle_{\mathfrak{e}\mathbf{p}^* g}
\end{equation}
and thus, by \eqref{eq:geodesic.network.manifold.H.variational},
\begin{equation} \label{eq:geodesic.network.manifold.H.breve.variational}
	\langle \mathring{\cH}(g, \mathbf{p}), \mathbf{q} \rangle_{\mathfrak{e}\mathbf{p}^* g} = \big[ \tfrac{d}{dt} \cL(g, \mathbf{p} + t\mathbf{q}) \big]_{t=0}.
\end{equation}

\begin{lemm}[cf. Lemma {\ref{lemm:geodesic.network.manifold.equivalence}}]  \label{lemm:geodesic.network.manifold.equivalence.breve}
	For $(g, \mathbf{p}) \in \met^k(M, \mathfrak{g}, \eps) \times \mathring{B}_\eps(\mathbf{0})$:
	\begin{equation} \label{eq:geodesic.network.manifold.H.equivalence.breve}
		\iota_g(G, \mathfrak{e} \mathbf{p}) \text{ is } g\text{-stationary} \iff \mathring{\cH}(g, \mathbf{p}) = \mathbf{0}.
	\end{equation}
\end{lemm}
\begin{proof}
	Using \eqref{eq:geodesic.network.manifold.H},  \eqref{eq:geodesic.network.manifold.small.transverse}, and \eqref{eq:geodesic.network.manifold.H.breve.projection} we have
	\begin{equation} \label{eq:geodesic.network.manifold.equivalence.breve}
		\cH(g, \mathbf{p}) = \mathbf{0} \iff \mathring{\cH}(g, \mathbf{p}) = \mathbf{0}.
	\end{equation}
	The result then follows from Lemma \ref{lemm:geodesic.network.manifold.equivalence}. 
\end{proof}

\begin{lemm}[cf. {\cite[Theorem 1.2 (1)]{White:bumpy.old}}] \label{lemm:geodesic.network.manifold.submersion}
	If $(g, \mathbf{p}) \in \met^k(M, \mathfrak{g}, \eps) \times \mathring{B}_\eps(\mathbf{0})$ satisfies $\mathring{\cH}(g, \mathbf{p}) = \mathbf{0}$, then $\mathring{\cH}$ is a $C^{k-1}$ submersion at $(g, \mathbf{p})$.
\end{lemm}

We direct the reader to \cite[II.2]{Lang:fundamentals.dg} for information on Banach submersions. The proof of Lemma \ref{lemm:geodesic.network.manifold.submersion} hinges on the following technical lemma:

\begin{lemm}[cf. {\cite[Theorem 1.1]{White:bumpy.old}}] \label{lemm:geodesic.network.manifold.linearization}
	$\mathring{\cH}$ is a $C^{k-1}$ map of Banach spaces and, when $\mathring{\cH}(g, \mathbf{p}) = \mathbf{0}$, the linearization
	\[ \mathring{J} = D_2 \mathring{\cH}(g, \mathbf{p}) : \mathring{Y} \to \mathring{Y} \]
	is a self-adjoint map with respect to the inner product $\mathfrak{e}\mathbf{p}^* g$ on $\mathring{Y}$.
\end{lemm}
\begin{proof}
	The fact that $\cH$ is $C^{k-1}$ on $\met^k(M, \mathfrak{g}, \eps) \times B_\eps(\mathbf{0})$ follows from \eqref{eq:geodesic.network.manifold.H} and that $\met^k(M, \mathfrak{g}, \eps) \ni g \mapsto \exp^g$ and the $g$-projection maps are $C^{k-1}$. 
	
	Now, fix $(g, \mathbf{p})$ with $\mathring{\cH}(g, \mathbf{p}) = \mathbf{0}$. By \eqref{eq:geodesic.network.manifold.H.breve.variational} we have
	\begin{align*}
		\langle \mathring{J} \mathbf{q}, \mathbf{r} \rangle_{\mathfrak{e}(\mathbf{p})^*g}
			& = \langle D_2 \mathring{\cH}(g, \mathbf{p})\{\mathbf{q}\}, \mathbf{r} \rangle_{\mathfrak{e}\mathbf{p}^*g} \\
			& = \big[ \langle \tfrac{d}{ds} \mathring{\cH}(g, \mathbf{p} + s \mathbf{q}), \mathbf{r} \rangle_{\mathfrak{e}\mathbf{p}^* g} \big]_{s=0} \\
			& = \big[ \tfrac{d}{ds} \langle \mathring{\cH}(g, \mathbf{p} + s \mathbf{q}), \mathbf{r} \rangle_{\mathfrak{e}(\mathbf{p} + s \mathbf{q})^* g} \big]_{s=0} \\
			& = \big[ \tfrac{\partial^2}{\partial s \partial t} \cL(g, \mathbf{p} + s \mathbf{q} + t \mathbf{r}) \big]_{s=t=0}.
	\end{align*}
	Since $\cL$ is $C^{k-1}$ (Lemma \ref{lemm:geodesic.network.manifold.L.regularity}) and $k \geq 3$, we may swap the order of differentiation in $s$ and $t$, and working backwards similarly we get $\langle \mathbf{q}, \mathring{J}\mathbf{r} \rangle_{\mathfrak{e}\mathbf{p}^* g}$. This gives the desired self-adjointness.
\end{proof}

\begin{proof}[Proof of Lemma \ref{lemm:geodesic.network.manifold.submersion}]
	Since $\mathring{\cH}$'s target space is finite dimensional and $\mathring{\cH}$ is $C^{k-1}$ by Lemma \ref{lemm:geodesic.network.manifold.linearization}, it suffices to prove that $D\mathring{\cH}(g, \mathbf{p})$ is surjective (\cite[Proposition II.2.3 (ii)]{Lang:fundamentals.dg}). For this, note that $D\mathring{\cH}(g, \mathbf{p})$ acts as
	\begin{align} \label{eq:geodesic.network.manifold.submersion.action}
		D\mathring{\cH}(g, \mathbf{p})\{ h, \mathbf{q} \} 
			& = D_1 \mathring{\cH}(g, \mathbf{p}) \{h\} + D_2 \mathring{\cH}(g, \mathbf{p}) \{\mathbf{q}\} \\
			& = D_1 \mathring{\cH}(g, \mathbf{p}) \{h\} + \mathring{J} \mathbf{q}, \nonumber
	\end{align}
	where $\mathring{J}$ is as in Lemma \eqref{lemm:geodesic.network.manifold.linearization}. Then, \eqref{eq:geodesic.network.manifold.submersion.action} and the self-adjointness of $J$ imply
	\[ \image D\mathring{\cH}(g, \mathbf{p}) \supset \image \mathring{J} = \mathring{K}^\perp, \]
	where 
	\[ \mathring{K} := \ker \mathring{J} \]
	and $\perp$ is taken with respect to the inner product $\mathfrak{e}\mathbf{p}^* g$ on $\mathring{Y}$. The surjectivity will then follow from \eqref{eq:geodesic.network.manifold.submersion.action} if we can show:
	\begin{claim}[cf. {\cite[p. 168, (3)]{White:bumpy.old}}] \label{clai:geodesic.network.manifold.submersion.full.img}
		If $\pi_{\mathring{K}} : \mathring{Y} \to \mathring{K}$ is the projection onto $\mathring{K}$ in $\mathring{Y}$ with respect to the inner product $\mathfrak{e}\mathbf{p}^* g$ on $\mathring{Y}$, then
		\[ \pi_{\mathring{K}} \circ D_1 \mathring{\cH}(g, \mathbf{p}) : \Tan_g \met^k(M, \mathfrak{g}, \eps) \to \mathring{K} \]
		is surjective.
	\end{claim}
	
	This claim will require some effort to prove, so we break up its proof in smaller claims. The theme is that we wish to understand the implications of $\mathbf{q} \in \mathring{K}$, i.e., $\mathring{J}\mathbf{q} = \mathbf{0}$. To that end, let us pick $\mathbf{q} \in \mathring{K}$. 
	
	It follows from elementary Jacobi field analysis (cf. \eqref{eq:geodesic.network.finite.stratification.jf.tangent}) and \eqref{eq:geodesic.network.manifold.H} that, with $D_2$ indicating differentiation with respect to $\mathbf{p}$,
	\[ D_2 \big( \mathfrak{e}\mathbf{p}_* \big[  \cH(g, \mathbf{p}) \big] \big)\{\mathbf{q}\}(u) = - \sum_{v \in E_u} \omega(\{u,v\}) \nabla_{u,v}^\perp J^{\mathbf{q}}_{u,v}(\mathfrak{e}\mathbf{p}(u)), \]
	where
	\begin{multline} \label{eq:geodesic.network.manifold.jacobi.field}
		J_{u,v}^{\mathbf{q}} \text{ is the unique Jacobi field along } \sigma_g(\mathfrak{e}\mathbf{p}(u), \mathfrak{e}\mathbf{p}(v)) \\
		\text{satisfying } J^{\mathbf{q}}_{u,v}(\mathfrak{e}\mathbf{p}(u)) = \mathfrak{e}\mathbf{p}(u)_* \big[ \mathbf{q}(u) \big], \; J^{\mathbf{q}}_{u,v}(\mathfrak{e}\mathbf{p}(v)) = \mathfrak{e}\mathbf{p}(u)_* \big[ \mathbf{q}(v) \big],
	\end{multline}
	and $\nabla_{u,v}$ is the unit speed covariant differentiation along $\sigma_g(\mathfrak{e}\mathbf{p}(u), \mathfrak{e}\mathbf{p}(v))$, oriented from $\mathfrak{e}\mathbf{p}(u)$ to $\mathfrak{e}\mathbf{p}(v)$, and $\perp$ is the projection onto the normal bundle with respect to $g$. (The existence and uniqueness of $J_{u,v}^{\mathbf{q}}$ are due to the fact that we are working below the injectivity radius.) 
	
	Since $\cH(g, \mathbf{p}) = \mathbf{0}$, we deduce
	\begin{equation} \label{eq:geodesic.network.manifold.J.geom}
		\mathfrak{e}\mathbf{p}(u)_* \big[  D_2 \cH(g, \mathbf{p})\{\mathbf{q}\}(u) \big] = - \sum_{v \in E_u} \omega(\{u,v\}) \nabla_{u,v}^\perp J^{\mathbf{q}}_{u,v}(\mathfrak{e}\mathbf{p}(u)),
	\end{equation}
	Therefore, for $u \in V \setminus V_{\reg}$:
	\begin{equation} \label{eq:geodesic.network.manifold.J.sing}
		\mathfrak{e}\mathbf{p}(u)_* \big[ (\mathring{J} \mathbf{q})(u) \big] = - \sum_{v \in E_u} \omega(\{u,v\}) \nabla_{u,v}^\perp J^{\mathbf{q}}_{u,v}(\mathfrak{e}\mathbf{p}(u)),
	\end{equation}
	while for $u \in V_{\reg}$:
	\begin{equation} \label{eq:geodesic.network.manifold.J.reg}
			\mathfrak{e}\mathbf{p}(u)_* \big[ (\mathring{J} \mathbf{q})(u) \big] = - \proj^{\mathfrak{e}\mathbf{p}(u)^* g}_{\mathring{B}_\eps(0)} \sum_{v \in E_u} \omega(\{u,v\}) \nabla_{u,v}^\perp J^{\mathbf{q}}_{u,v}(\mathfrak{e}\mathbf{p}(u)).
	\end{equation}
	
	\begin{claim} \label{clai:geodesic.network.manifold.jacobi.field.nontrivial}
		Suppose $\mathbf{q} \in \mathring{Y} \setminus \{\mathbf{0}\}$. Then, 
		\begin{equation} \label{eq:geodesic.network.manifold.jacobi.field.nontrivial}
			\mathfrak{e} \mathbf{p}(u)_* \big[ \mathbf{q}(u) \big]^\perp \neq 0 \text{ for some } \{ u, v \} \in E,
		\end{equation}
		where $\perp$ denotes projection onto the normal bundle of $\sigma_g(\mathfrak{e} \mathbf{p}(u), \mathfrak{e} \mathbf{p}(v))$ with respect to $g$.
	\end{claim}
	\begin{proof}[Proof of claim]
		We argue by contradiction, supposing that \eqref{eq:geodesic.network.manifold.jacobi.field.nontrivial} fails for all $\{u,v\} \in E$. Then, $\cH(g, \mathbf{p}) = \mathbf{0}$, \eqref{eq:geodesic.network.manifold.H.equivalence}, and Lemma \ref{lemm:geodesic.network.embedding.structure}'s (1)(b)-(c) imply
		\begin{equation} \label{eq:geodesic.network.manifold.q.bdry}
			\mathbf{q}(u) = 0 \text{ for all } u \in V \setminus V_{\reg}.
		\end{equation}
		The additional fact that $\mathring{B}_\eps(0) \ni q \mapsto (\mathfrak{e} \mathbf{p}(u)_* q)^\perp$ is an isomorphism when $u \in V_{\reg}$, as a consequence of \eqref{eq:geodesic.network.manifold.small.transverse}, gives
		\begin{equation} \label{eq:geodesic.network.manifold.q.intr}
			\mathbf{q}(u) = 0 \text{ for all } u \in V_{\reg}.
		\end{equation}
		The combination of \eqref{eq:geodesic.network.manifold.q.bdry}, \eqref{eq:geodesic.network.manifold.q.intr} contradicts $\mathbf{q} \neq \mathbf{0}$. 
	\end{proof}

	\begin{claim} \label{clai:geodesic.network.manifold.submersion.nondegenerate}
		Suppose $\mathbf{q} \in \mathring{Y}$ satisfies \eqref{eq:geodesic.network.manifold.jacobi.field.nontrivial}. Then, 
		\begin{equation} \label{eq:geodesic.network.manifold.submersion.nondegenerate}
			\big[ \tfrac{\partial^2}{\partial s \partial t} \cL(g(s), \mathbf{p} + t \mathbf{q}) \big]_{s=t=0} \neq \mathbf{0}
		\end{equation}		
		for at least one $C^1$ path $s \mapsto g(s)$ with $g(0) = g$.
	\end{claim}
	\begin{proof}[Proof of claim]
		By virtue of \eqref{eq:geodesic.network.manifold.jacobi.field.nontrivial}, there exists $\{u,v\} \in E$ along which the Jacobi field $J^{\mathbf{q}}_{u,v}$ of \eqref{eq:geodesic.network.manifold.jacobi.field} has a nontrivial normal component $(J^{\mathbf{q}}_{u,v})^\perp$. Set:
		\[ g(s) := (1 + sF) g, \]
		where $F : M \to \RR$ is $C^k$ is arbitrary for now but supported on an open set $U \subset M$ satisfying		
		\begin{equation} \label{eq:geodesic.network.manifold.submersion.U}
			\emptyset \neq U \cap (\supp \iota_g(G, \mathfrak{e} \mathbf{p})) \Subset \operatorname{int} \sigma_g(\mathfrak{e}\mathbf{p}(u), \mathfrak{e}\mathbf{p}(v)) \cap \{ (J^{\mathbf{q}}_{u,v})^\perp \neq 0 \}.
		\end{equation}
		By \eqref{eq:geodesic.network.manifold.submersion.U}, the chain rule, the variation formula for induced volume forms under changes of metric, the first variation formula combined with the geodesic nature of $\sigma_g(\mathfrak{e}\mathbf{p}(u), \mathfrak{e}\mathbf{p}(v))$, and integration by parts:
		\begin{align} \label{eq:geodesic.network.manifold.submersion.nondegenerate.computation}
			& \big[ \tfrac{\partial^2}{\partial t \partial s} \cL(g(s), \mathbf{p} + t \mathbf{q}) \big]_{s=t=0} \\
			& = \omega(\{u,v\}) \Big[ \tfrac{\partial}{\partial t} \tfrac{\partial}{\partial s} \int_{\sigma_{g(s)}(\mathfrak{e}(\mathbf{p}+t\mathbf{q})(u), \mathfrak{e}(\mathbf{p}+t\mathbf{q})(v))} d\ell_{g(s)} \Big]_{s=t=0} \nonumber \\
			& = \omega(\{u,v\}) \Big[ \tfrac{\partial}{\partial t} \int_{\sigma_{g}(\mathfrak{e}(\mathbf{p}+t\mathbf{q})(u), \mathfrak{e}(\mathbf{p}+t\mathbf{q})(v))} \tfrac12 F \, d\ell_{g} \Big]_{t=0} \nonumber \\
			& = \tfrac12 \omega(\{u,v\}) \int_{\sigma_g(\mathfrak{e} \mathbf{p}(u), \mathfrak{e} \mathbf{p}(v))} F \Div_{\sigma_g(\mathfrak{e} \mathbf{p}(u), \mathfrak{e} \mathbf{p}(v))} J^{\mathbf{q}}_{u,v} + dF(J^{\mathbf{q}}_{u,v}) \, d\ell_g \nonumber \\
			& = \tfrac12 \omega(\{u,v\}) \int_{\sigma_g(\mathfrak{e} \mathbf{p}(u), \mathfrak{e} \mathbf{p}(v))} \langle \nabla^\perp_g F, (J^{\mathbf{q}}_{u,v})^\perp \rangle \, d\ell_g, \nonumber
		\end{align}		
		where $\perp$ is the orthogonal projection to the normal bundle of $\sigma_g(\mathfrak{e}\mathbf{p}(u), \mathfrak{e}\mathbf{p}(v))$ with respect to $g$. 
		
		Pick any $F$ subject to \eqref{eq:geodesic.network.manifold.submersion.U} and extend it off $\sigma_g(\mathfrak{e}\mathbf{p}(u), \mathfrak{e}\mathbf{p}(v))$ in such a way so as to  additionally satisfy $\langle \nabla^\perp_g F, (J^{\mathbf{q}}_{u,v})^\perp \rangle_g \geq 0$ along  $\sigma_g(\mathfrak{e}\mathbf{p}(u), \mathfrak{e}\mathbf{p}(v))$, with strict inequality at some interior point. Since the curve is itself $C^{k+1}$, this is always feasible as long as $\supp F$ satisfies  \eqref{eq:geodesic.network.manifold.submersion.U}. Then,  \eqref{eq:geodesic.network.manifold.submersion.nondegenerate.computation} implies \eqref{eq:geodesic.network.manifold.submersion.nondegenerate} after swapping the $s$ and $t$ using $\cL$'s $C^2$ regularity (Lemma \ref{lemm:geodesic.network.manifold.L.regularity}, $k \geq 3$).
	\end{proof}
	
	\begin{claim} \label{clai:geodesic.network.manifold.submersion.degenerate}
		Suppose $\mathbf{q} \in \mathring{K} \cap (\image \pi_{\mathring{K}} \circ D_1 \mathring{\cH}(g, \mathbf{p}))^\perp$. Then
		\[ \big[ \tfrac{\partial^2}{\partial s \partial t} \cL(g(s), \mathbf{p} + t \mathbf{q}) \big]_{s=t=0} = \mathbf{0} \]
		for all $C^1$ paths $s \mapsto g(s)$ with $g(0) = g$.
	\end{claim}
	\begin{proof}[Proof of claim]
		By $\mathring{\cH}(g, \mathbf{p}) = \mathbf{0}$ and \eqref{eq:geodesic.network.manifold.H.breve.variational},
		\begin{align*}
			\big[ \tfrac{\partial^2}{\partial s \partial t} \cL(g(s), \mathbf{p} + t \mathbf{q}) \big]_{s=t=0} 
				& = \big[ \tfrac{d}{ds} \langle \mathring{\cH}(g(s), \mathbf{p}), \mathbf{q} \rangle_{\mathfrak{e}\mathbf{p}^* g(s)} \big]_{s=0} \\
				& = \langle \big[ \tfrac{d}{ds} \mathring{\cH}(g(s), \mathbf{p}) \big]_{t=0}, \mathbf{q} \rangle_{\mathfrak{e} \mathbf{p}^* g} \\
				& = \langle D_1 \mathring{\cH}(g, \mathbf{p}) \{\dot{g}(0)\}, \mathbf{q} \rangle_{\mathfrak{e} \mathbf{p}^* g} \\
				& = \langle \pi_{\mathring{K}} (D_1 \mathring{\cH}(g, \mathbf{p}) \{\dot{g}(0)\}), \mathbf{q} \rangle_{\mathfrak{e} \mathbf{p}^* g} = 0,
		\end{align*}
		as claimed.
	\end{proof}
	
	We finally arrive to:
	
	\begin{proof}[Proof of Claim \ref{clai:geodesic.network.manifold.submersion.full.img}]
		By Claims \ref{clai:geodesic.network.manifold.submersion.nondegenerate} and \ref{clai:geodesic.network.manifold.submersion.degenerate}, 
		\[
		\mathring{K} \cap (\image \pi_{\mathring{K}} \circ D_1 \mathring{\cH}(g, \mathbf{p}))^\perp = \{ \mathbf{0} \}.
		\]
		This concludes the proof of the claim.
	\end{proof}

	Our proof of Lemma \ref{lemm:geodesic.network.manifold.submersion} is complete.
\end{proof}

\begin{rema} \label{rema:geodesic.network.manifold.jacobi.field.construction}
	It follows from \eqref{eq:geodesic.network.manifold.jacobi.field}, \eqref{eq:geodesic.network.manifold.J.sing}, \eqref{eq:geodesic.network.manifold.J.reg} that
	\[ \mathbf{q} \in \mathring{K} \implies (J^{\mathbf{q}}_{u,v})_{\{u,v\} \in E} \text{ fulfills Lemma } \ref{lemm:geodesic.network.jacobi.field} \text{'s (J$_1$'), (J$_3$'),} \]
	Indeed: (J$_1$') follows from \eqref{eq:geodesic.network.manifold.jacobi.field}, and (J$_3$') follows from \eqref{eq:geodesic.network.manifold.J.sing}, \eqref{eq:geodesic.network.manifold.J.reg}, and the fact that $\mathring{B}_\eps(0) \ni q \mapsto (\mathfrak{e} \mathbf{p}(u)_* q)^\perp$ is an isomorphism when $u \in V_{\reg}$, due to \eqref{eq:geodesic.network.manifold.small.transverse}. Thus, by Lemma \ref{lemm:geodesic.network.jacobi.field}, there exists a stationary varifold Jacobi field along $S$ with the same normal components as $\{ J_{u,v} : \{ u,v \} \in E \}$. By Claim \ref{clai:geodesic.network.manifold.jacobi.field.nontrivial}, this Jacobi field is not-everywhere-tangential if and only if $\mathbf{q} \neq \mathbf{0}$.
\end{rema}

As a consequence of Lemma \ref{lemm:geodesic.network.manifold.submersion},
\[ \cS_{\textrm{par}} := \{ (g, \mathbf{p}) \in \met^k(M, \mathfrak{g}, \eps) \times \mathring{B}_\eps(\mathbf{0}) : \mathring{\cH}(g, \mathbf{p}) = \mathbf{0} \} \]
is a $C^{k-1}$ Banach submanifold of $\met^k(M, \mathfrak{g}, \eps) \times \mathring{B}_\eps(\mathbf{0})$ satisfying
\[ \Tan_{(g, \mathbf{p})} \cS_{\textrm{par}} = \ker D\mathring{\cH}(g, \mathbf{p}) \text{ for all } (g, \mathbf{p}) \in \cS_{\textrm{par}}. \]
We proceed to build our atlas for $\cS^\Lambda_G$ locally near $(g_0, \mathbf{p}_0)$, as in \cite[p. 179-180]{White:bumpy.old}.\footnote{Our $\cS_{\textrm{par}}$ corresponds to $\cS$ in the reference, and our $\cS^\Lambda_G$ to $\cM$.} Note that we may work with the convention $\mathbf{p}_0 \in B_\eps(\mathbf{0})$ in view of \eqref{eq:geodesic.network.manifold.small.frakg}, though it might happen that $\mathbf{p}_0 \in B_\eps(\mathbf{0}) \setminus \mathring{B}_\eps(\mathbf{0})$.

Our chart near our center point will be
\begin{equation} \label{eq:geodesic.network.manifold.chart}
	\varphi : \cS_{\textrm{par}} \to \cS^\Lambda_G, \; (g, \mathbf{p}) \mapsto (g, \iota_g(G, \mathfrak{e}\mathbf{p})).
\end{equation}
It is well-defined and bijects onto its image by \eqref{eq:geodesic.network.manifold.small.frakg}, \eqref{eq:geodesic.network.manifold.small.inj}, \eqref{eq:geodesic.network.manifold.small.imm}. Its image indeed contains $(g_0, \iota_g(G, \mathbf{p}_0))$ by \eqref{eq:geodesic.network.manifold.small.frakg} and the fact that we chose $\mathfrak{p} := \mathbf{p}_0$. Since the images of these charts cover $\cS_G^\Lambda$ as $(g_0, \mathbf{p}_0)$ varies (this is due to Theorem \ref{theo:geodesic.network.finite.stratification}), all that remains to note, for the manifold structure, is that the transition maps are $C^{k-1}$. This is done as in \cite[p. 179-180]{White:bumpy.old}, which we omit.

\begin{proof}[Proof of Theorem \ref{theo:geodesic.network.manifold}]
	(1) The Banach manifold structure was built above. The fact that $\cS^\Lambda_G$ is second countable will be a consequence of the following observation. By Definitions \ref{defi:geodesic.network.stationary.lambda}, \ref{defi:geodesic.network.manifold.slice}, and Theorem \ref{theo:geodesic.network.finite.stratification}:
	\begin{align*}
		\cS^\Lambda_G 
			& = \{ \iota_g(G, \mathbf{p}) : g \in \met^k(M), \; Q \cdot \inj(M, g) > \Lambda, \; \mathbf{p} \in \emb_g(G, M), \nonumber \\
			& \qquad \qquad \qquad \iota_g(G, \mathbf{p}) \text{ is } g\text{-stationary} \} \nonumber \\
			& = \{ \iota_g(G, \mathbf{p}) : g \in \met^k(M), \; Q \cdot \inj(M, g) > \Lambda, \; \mathbf{p} \in \cB \emb_g(G, M), \nonumber \\
			& \qquad \qquad \qquad \iota_g(G, \mathbf{p}) \text{ is } g\text{-stationary} \}.
	\end{align*}
	Therefore, $\cS^\Lambda_G$ can also be endowed with the topology induced by the subset
	\begin{multline*}
		P := \{ (g, \mathbf{p}) : g \in \met^k(M), \; Q \cdot \inj(M, g) > \Lambda, \; \mathbf{p} \in \cB \emb_g(G, M), \\
			\iota_g(G, \mathbf{p}) \text{ is } g\text{-stationary} \} \subset \met^k(M) \times M^V
	\end{multline*}
	under the quotient
	\[ [(g, \mathbf{p})] = [(g, \mathbf{p}')] \text{ in } P/\sim \iff \iota_g(G, \mathbf{p}) = \iota_g(G, \mathbf{p}') \text{ in } \cI \cV_1(M). \]
	Notice that the latter topology is second countable, since $\met^k(M)$, $M^V$ are themselves second countable, and thus so is their product and its subset $P$. Its quotient $P/\sim$ is second countable because, by Corollary \ref{coro:geodesic.network.embedding.structure} parts (2), (3), $\sim$ is a group action that acts by homeomorphisms, the group being that of graph isomorphisms of $G$ times one $\SS^1$ factor per cyclic component of $G$.
	
	\begin{claim} \label{clai:geodesic.network.manifold.second.countable.topology}
		The Banach manifold topology on $\cS^\Lambda_G$ coincides with the quotient topology.
	\end{claim}
	\begin{proof}
		Consider a convergent sequence $S_i \to S_\infty$ with respect to the manifold topology. By virtue of the atlas we defined, if we set $S_\infty = \iota_{g_\infty}(G, \mathbf{p}_\infty)$ with $\mathbf{p}_\infty \in \cB \emb_{g_\infty}(G, \mathbf{p}_\infty)$, manifold convergence implies $S_i = \iota_{g_i}(G, \mathbf{p}_i)$ with $\met^k(M) \ni g_i \to g_\infty$ in $\met^k(M)$, $\emb_{g_i}(G, M) \ni \mathbf{p}_i \to \mathbf{p}_\infty$ in $M^V$. By equidistancing vertices, it is easy to see that for large $i$ there is $\mathbf{p}_i' \in \cB \emb_{g_i}(G, M)$ with $\iota_{g_i}(G, \mathbf{p}_i') = S_i$ and $\mathbf{p}_i' \to \mathbf{p}_\infty$. Then $[(g_i, \mathbf{p}_i')] \to [(g_\infty, \mathbf{p}_\infty)]$ in $P/\sim$, so $S_i \to S_\infty$ with respect to the quotient topology.
		
		Conversely, consider a convergent sequence $S_i \to S_\infty$ with respect to the quotient topology. Then $S_\infty = (g_\infty, \mathbf{p}_\infty)$, $S_i = (g_i, \mathbf{p}_i)$, $g_i \to g_\infty$ in $\met^k(M)$, and $\cB \emb_{g_i}(G, M) \ni \mathbf{p}_i \to \mathbf{p}_\infty$ in $M^V$ after pulling back $\mathbf{p}_i$ by an isometry element. Recall that to construct our chart near $(g_\infty, \mathbf{p}_\infty)$, we relied on a nearby smooth $\mathfrak{g}$ and $\mathfrak{p} := \mathbf{p}_\infty$ and a direct sum of tangent spaces and transverse segments $\mathring{B}^{\mathfrak{g}}_\eps(\mathfrak{p})$ to $\mathfrak{p}$. It is straightforward to show that, for large $i$, there is $\mathbf{p}_i' \in \mathring{B}^{\mathfrak{g}}_\eps(\mathfrak{p})$ with $\iota_{g_i}(G, \mathbf{p}_i') = S_i$ and $\mathbf{p}_i' \to \mathbf{p}_\infty$. So, $S_i \to S_\infty$ in the manifold topology.
	\end{proof}

	(2), (3). These are the content of the lemma below viewed in the light of the chart $\varphi$ above, together with Remark \ref{rema:geodesic.network.manifold.jacobi.field.construction}.

\begin{lemm}[cf. {\cite[Theorem 1.2 (2)]{White:bumpy.old}}] \label{lemm:geodesic.network.manifold.fredholm}
	The projection
	\[ \mathring{\Pi} : \cS_{\textrm{par}} \to \met^k(M, \mathfrak{g}, \eps), \; (g, \mathbf{p}) \mapsto g, \]
	is a $C^{k-1}$ and Fredholm map with Fredholm index zero and
	\[ \ker D\mathring{\Pi}(g, \mathbf{p}) = \{0\} \times \mathring{K} \]
	where, as before, $\mathring{K} := \ker D_2 \cH(g, \mathbf{p})$.
\end{lemm}
\begin{proof}
	Note that $\mathring{\Pi}$ is the restriction of 
	\begin{equation*} 
		\Pi : \met^k(M, \mathfrak{g}, \eps) \times \mathring{B}_\eps(\mathbf{0}) \to \met^k(M, \mathfrak{g}, \eps), \; (g, \mathbf{p}) \mapsto g,
	\end{equation*}
	to the $C^{k-1}$ Banach manifold $\cS_{\textrm{par}}$, so the regularity of $\mathring{\Pi}$ follows from the regularity of $\cS_{\textrm{par}}$. We now check the Fredholm index zero property. To that end, for every $(g, \mathbf{p}) \in \cS_{\textrm{par}}$:
	\begin{align*}
		\ker D\mathring{\Pi}(g, \mathbf{p})
			& = \ker D(\Pi \vert \cS_{\textrm{par}})(g, \mathbf{p}) \\
			& = \ker \Pi \cap \Tan_{(g,\mathbf{p})} \cS_{\textrm{par}} \\
			& = (\{0\} \times \mathring{Y}) \cap \ker D\mathring{\cH}(g, \mathbf{p}) \\
			& = \{0\} \times \ker D_2 \mathring{\cH}(g, \mathbf{p}) = \{0\} \times \mathring{K}.
	\end{align*}
	It remains to show that the codimension of $\image D(\Pi \vert \cS_{\textrm{par}})(g, \mathbf{p})$ is $\dim \mathring{K}$. Indeed, using the self-adjointness of $\mathring{J}$ from Lemma \ref{lemm:geodesic.network.manifold.linearization}, and also \eqref{eq:geodesic.network.manifold.submersion.action}:
	\begin{align*}
		\image D\mathring{\Pi}(g, \mathbf{p})
			& = \image D(\Pi \vert \cS_{\textrm{par}})(g, \mathbf{p}) \\
			& = \Pi(\Tan_{(g, \mathbf{p})} \cS_{\textrm{par}}) \\
			& = \Pi(\ker D\mathring{\cH}(g, \mathbf{p})) \\
			& = \Pi \{ (h, \mathbf{q}) : D_1 \mathring{\cH}(g, \mathbf{p})\{h\} + \mathring{J}\mathbf{q} = \mathbf{0} \} \\
			& = \Pi \{ (h, \mathbf{q}) : (\pi_{\mathring{K}} \circ D_1 \mathring{\cH}(g, \mathbf{p}))\{h\} = \mathbf{0} \} \\
			& = \ker \pi_{\mathring{K}} \circ D_1 \mathring{\cH}(g, \mathbf{p}),
	\end{align*}
	which does indeed have codimension $\dim \mathring{K}$ in $\mathring{Y}$ by Claim \ref{clai:geodesic.network.manifold.submersion.full.img}.
\end{proof}

	(4). This is a consequence of the Sard--Smale theorem \cite{Smale:sard} and parts (1), (2).
\end{proof}

\subsection{Constrained genericity}

We will seek to apply a refinement of Theorem \ref{theo:geodesic.network.manifold}, where we study genericity over a refined moduli space of metrics, where the length of a certain finite subset of stationary integral 1-varifolds is held fixed. Rather than start from scratch, we obtain this as a corollary of what has already been shown. Fix $g_0 \in \met^k(M)$, $\Lambda > 0$, and $S^0_i \in \cS^{\Lambda}(g_0)$, $i = 1, \ldots, m$ such that
\begin{equation} \label{eq:geodesic.network.constrained.assumption.regular}
	\sing S^0_i = \emptyset \text{ and } \Theta^1(S_0^i) = 1 \text{ along } \supp S_0^i \text{ for all } i = 1, \ldots, m,
\end{equation}
\begin{multline} \label{eq:geodesic.network.constrained.assumption.generic}
	S^0_i, i = 1, \ldots, m, \text{ has no not-everywhere-tangential} \\
	\text{stationary varifold Jacobi fields in } (M, g_0),
\end{multline}
\begin{equation} \label{eq:geodesic.network.constrained.assumption.discrete}
	\supp S^0_i \cap \supp S^0_i \text{ is discrete whenever } i \neq j \in \{ 1, \ldots, m\}.
\end{equation}
In what follows, there is no loss of generality in assuming each $S^0_i$ is connected and multiplicity-one, so we proceed to do so. Fix $Q$ such that \eqref{eq:geodesic.network.finite.stratification.Q} holds with $\Lambda$, $Q$, and $g_0$, and suppose
\[ j^\Lambda_Q(S^0_i) =: (G_0, \mathbf{p}_i^0) \]
for a fixed cyclic graph $G_0$ whose edges have weight one. By Remark \ref{rema:geodesic.network.manifold.regularity} Theorem \ref{theo:geodesic.network.manifold} applies with $k=3$ and implies
\begin{equation} \label{eq:geodesic.network.constrained.generic.k3}
	(g_0, S^0_i) \text{ is a regular point of }  \pi^{\Lambda,3}_{G_0} \text{ for all } i = 1, \ldots, m,
\end{equation}
so there exist $\eps > 0$ and open neighborhood $\cU_i^3 \subset  \cS^{\Lambda,3}_{G_0}$ of $(g_0, S^0_i)$ such that, for each $i = 1, \ldots, m$,
\[ (\pi^{\Lambda,3}_{G_0} \vert {\cU_i^3}) : \cU_i^3 \to \met^3(M, g_0, \eps) \]
is a $C^2$ diffeomorphism. We further shrink $\eps > 0$ so \eqref{eq:geodesic.network.finite.stratification.Q} holds for all $g \in \met^3(M, g_0, \eps)$. Label the inverses of the diffeomorphisms above as
\[ \kappa_i : \met^3(M, g_0, \eps) \to \cU_i^3, \; i=1, \ldots, m, \]
(they are $C^2$, as the inverses of $C^2$ maps) and write
\[ \kappa_i =: (g, S_i(g)). \]
By construction, $S_i(g_0) = S_i^0$.

\begin{claim} \label{clai:geodesic.network.constrained.eps.k3}
	With $S_1(g), \ldots, S_m(g)$ in place of $S_1^0, \ldots, S_m^0$:
	\begin{itemize}
		\item \eqref{eq:geodesic.network.constrained.assumption.regular},  \eqref{eq:geodesic.network.constrained.assumption.generic} hold for all $g \in \met^3(M, g_0, \eps)$, and
		\item \eqref{eq:geodesic.network.constrained.assumption.discrete}  holds for all $g \in \met^3(M, g_0, \eps)$ if $\eps$ is sufficiently small.
	\end{itemize}
\end{claim}
\begin{proof}[Proof of claim]
	The persistence of \eqref{eq:geodesic.network.constrained.assumption.regular} across $g$ is a trivial consequence of, e.g., our construction of $S_i(g) \in \cS^{\Lambda,3}_{G_0}(g)$ and Lemma \ref{lemm:geodesic.network.embedding.structure}'s (3). The persistence of \eqref{eq:geodesic.network.constrained.assumption.generic} across $g$ is a consequence of Theorem \ref{theo:geodesic.network.manifold}'s (3) and that $\kappa_i$ is a diffeomorphism, so that its images are regular points of $\pi^{\Lambda,3}_{G_0}$. The persistence of \eqref{eq:geodesic.network.constrained.assumption.discrete} near $g = g_0$ follows from that \eqref{eq:geodesic.network.constrained.assumption.discrete} and the no-tangential-intersections nature of geodesics implies that intersections are transversal.
\end{proof}

In that follows, $\eps > 0$ is small enough for Claim \ref{clai:geodesic.network.constrained.eps.k3} to hold. Next, for every $i = 1, \ldots, m$, consider the corresponding total mass functional ($\cL$ in the notation of the previous section in coordinates) near $S^0_m$, pushed forward to the manifold and restrict to $\cU_i^3$:
\[ \lambda_i : \cU_i^3 \to \RR. \]
In the notation of Claim \ref{clai:geodesic.network.constrained.eps.k3}, $\lambda_i(g) = \length_g(\supp S_i(g))$. By Lemma \ref{lemm:geodesic.network.manifold.L.regularity}, $\lambda_i$ is $C^2$. 

So far, we have forced our setup to work with $k = 3$. We do, however, wish to prove a theorem for all $k \in \NN$, $k \geq 3$, and a \textit{uniform} choice of $\eps$ that is dictated from $k=3$. To that end, let us make the necessary adjustments. In all that follows, wherever we write $\cS^\Lambda_G$ and $\pi^\Lambda_G$ with no reference to regularity, it will be for regularity $k \in \NN$, $k \geq 3$, exactly as in the previous subsections.

It follows from Theorem \ref{theo:geodesic.network.manifold}, Remark \ref{rema:geodesic.network.manifold.regularity}, and \eqref{eq:geodesic.network.constrained.generic.k3}, that
\begin{equation} \label{eq:geodesic.network.constrained.generic.k}
	(g_0, S^0_i) \text{ is a regular point of } \pi^{\Lambda}_{G_0} \text{ for all } i = 1, \ldots, m,
\end{equation}
and therefore that, for $\cU_i := \cU_i^3 \cap (\met^k(M) \times \cI \cV_1(M))$,
\[ (\pi^\Lambda_{G_0} \vert \cU_i) : \cU_i \to \met^3(M, g_0, \eps) \cap \met^k(M) \]
is a $C^{k-1}$ diffeomorphism, where $\cU_i \subset \cS^\Lambda_G$ is open and the target space $\met^3(M, g_0, \eps) \cap \met^k(M)$ is from here on out always endowed with the induced topology as a subset of $\met^k(M)$. Observe that its inverse is merely the restriction of $\kappa_i$ to $\met^k(M, g_0, \eps) \cap \met^k(M)$. Henceforth we will only ever write $\kappa_i$ to mean this restriction, i.e.,
\[ \kappa_i : \met^3(M, g_0, \eps) \cap \met^k(M) \to \cU_i, \]
which is a $C^{k-1}$ map (as the inverse of a $C^{k-1}$ map). Likewise, we restrict $\lambda_i$ to $\cU_i$ and from now on will always write $\lambda_i$ to mean
\[ \lambda_i : \cU_i \to \RR, \]
which is a $C^{k-1}$ map on $\cU_i$, by Lemma \ref{lemm:geodesic.network.manifold.L.regularity}. If we denote
\[ \vec{\kappa} := (\kappa_1, \ldots, \kappa_m) : \met^3(M, g_0, \eps) \cap \met^k(M) \to \cU_1 \times \cdots \times \cU_m, \]
\[ \vec{\lambda} := (\lambda_1, \ldots, \lambda_m) : \cU_1 \times \cdots \times \cU_m \to \RR^m, \]
then the composite Banach map below is $C^{k-1}$, too:
\begin{equation} \label{eq:geodesic.network.constrained.submersion}
	\vec{\lambda} \circ \vec{\kappa} : \met^3(M, g_0, \eps) \cap \met^k(M) \to \RR^m.
\end{equation}

\begin{theo} [cf. Theorem {\ref{theo:geodesic.network.manifold}}] \label{theo:geodesic.network.constrained}
	Suppose $\Lambda > 0$, $g_0 \in \met^k(M)$, and $S^0_1, \ldots, S^0_m \in \cS^\Lambda(g_0)$ satisfy \eqref{eq:geodesic.network.constrained.assumption.regular}, \eqref{eq:geodesic.network.constrained.assumption.generic}, \eqref{eq:geodesic.network.constrained.assumption.discrete}. Fix $\eps > 0$ per Claim \ref{clai:geodesic.network.constrained.eps.k3}, and $Q \in \NN$ so that \eqref{eq:geodesic.network.finite.stratification.Q} holds with $\Lambda$, $Q$, and every $g \in \met^3(M, g_0, \eps) \cap \met^k(M)$. Define:
	\[ \cC \met^k := \met^3(M, g_0, \eps) \cap \met^k(M) \cap (\vec{\lambda} \circ \vec{\kappa})^{-1} \big[ (\vec{\lambda} \circ \vec{\kappa})(g_0) \big] \]
	Then:
	\begin{enumerate}
		\item $\cC \met^k$ is a $C^{k-1}$ Banach submanifold of $\met^3(M, g_0, \eps) \cap \met^k(M)$ with
			\[ \Tan_g (\cC \met^k) = \ker D(\vec{\lambda} \circ \vec{\kappa})(g) \text{ for all } g \in \cC \met^k. \]
		\item For every $G \in \cG(\Lambda, Q)$,
			\[ \cS^\Lambda_{G,\cC} := (\pi^\Lambda_{G})^{-1}(\cC \met^k) \]
			is a $C^{k-1}$ Banach submanifold of $\cS^\Lambda_{G}$.
		\item For every $G \in \cG(\Lambda, Q)$, the restriction
			\[ \pi^\Lambda_{G,\cC} := (\pi^\Lambda_{G} \vert \cS^\Lambda_{G,\cC}) : \cS^\Lambda_{G,\cC} \to \cC \met^k \]
		is $C^{k-1}$ with Fredholm index zero.
		\item For every $G \in \cG(\Lambda, Q)$, $(g, S) \in \cS^\Lambda_{G,\cC}$:
		\begin{align*}
			& (g, S) \text{ is a singular point of } \pi^\Lambda_{G,\cC} \\
			& \qquad \iff (g, S) \text{ is a singular point of } \pi^\Lambda_{G}.
		\end{align*}
		\item For every $G \in \cG(\Lambda, Q)$, the set of regular values of $\pi^\Lambda_{G,\cC}$ is comeager (``Baire generic'') in $\cC \met^k$.
	\end{enumerate}
\end{theo}

\begin{rema} [cf. Remark \ref{rema:geodesic.network.manifold.regularity}] \label{rema:geodesic.network.constrained.regularity}
	As before, the regularity $k \in \NN$, $k \geq 3$, enters into $\cS^{\Lambda}_{G,\cC}$, albeit mildly once again. If we were to write $\cS^{\Lambda,k}_{G,\cC}$ rather than $\cS^\Lambda_{G,\cC}$ for the space in Theorem \ref{theo:geodesic.network.constrained}, then it is straightforward to check directly from the definition and Remark \ref{rema:geodesic.network.manifold.regularity}'s \eqref{eq:geodesic.network.manifold.regularity.S} that for every $k' \geq k$,
	\begin{equation} \label{eq:geodesic.network.constrained.regularity.M}
		\cC \met^{k'} = \cC \met^k \cap \met^{k'}(M)
	\end{equation}
	and thus
	\begin{equation} \label{eq:geodesic.network.constrained.regularity.S}
		\cS^{\Lambda,k'}_{G,\cC} = \cS^{\Lambda,k}_{G,\cC} \cap (\met^{k'}(M) \times \cI\cV_1(M)).
	\end{equation}
	Likewise, if the projection to $\cC \met^k$ is denoted by $\pi^{\Lambda,k}_{G,\cC}$ and the set of its regular values by $\cR^{\Lambda,k}_{G,\cC} \subset \cC \met^k$, then for every $k' \geq k$
	\begin{equation} \label{eq:geodesic.network.constrained.regularity.R}
		\cR^{\Lambda,k}_{G,\cC} = \cR^{\Lambda,k'}_{G,\cC} \cap \met^{k'}(M)
	\end{equation}
	by virtue of \eqref{eq:geodesic.network.constrained.regularity.M}, Theorem \ref{theo:geodesic.network.constrained}'s (4), and Remark \ref{rema:geodesic.network.manifold.regularity}'s \eqref{eq:geodesic.network.manifold.regularity.R}.
\end{rema}

As before, the theorem readily implies a genericity result for smooth metrics. To that end, we denote
\[ \cC \met := \cC \met^k \cap \met(M), \]
and endow it with the usual subset topology induced from the $C^\infty$ topology of $\met(M)$.

\begin{coro} [cf. Corollary \ref{coro:geodesic.network.manifold.generic.smooth}] \label{coro:geodesic.network.constrained.generic.smooth}
	The set of regular values of $\pi^\Lambda_{G,\cC}$ that are also in $\cC \met$ is comeager (``Baire generic'') in $\cC \met$.
\end{coro}

Before setting out to prove the theorem and corollary, we prove:

\begin{lemm} \label{lemm:geodesic.network.constrained.submersion}
	$\vec{\lambda} \circ \vec{\kappa} : \met^3(M, g_0, \eps) \cap \met^k(M) \to \RR^m$ is a $C^{k-1}$ submersion.
\end{lemm}
\begin{proof}
	The regularity statement follows from Theorem \ref{theo:geodesic.network.manifold} and the regularity of $\cL$. The submersion statement will follow one we verify that $D(\vec{\lambda} \circ \vec{\kappa})(g)$, $g \in \met^3(M, g_0, \eps) \cap \met^k(M)$, is surjective (since the target space is finite dimensional). 
	
	To that end, let $(\ell_1, \ldots, \ell_m) \in \RR^m$ be arbitrary. We compute via the variation formula for induced volume forms under metric changes that the $i$-th coordinate of $D(\vec{\lambda} \circ \vec{\kappa})(g)\{h\} \in \RR^m$, $i = 1, \ldots, m$, is
	\begin{equation} \label{eq:geodesic.network.constrained.submersion.derivative}
		\big[ D(\vec{\lambda} \circ \vec{\kappa})(g)\{h\} \big]_i = \int_{\supp S_i(g)} \tfrac12 (\tr_{g|\supp S_i(g)} h) \, d\ell_g.
	\end{equation}
	Now take open sets $V_1, \ldots, V_m \subset M$ such that
	\begin{equation} \label{eq:geodesic.network.constrained.submersion.supp.1}
		V_i \cap \supp S_i(g) \neq \emptyset \text{ for all } i = 1, \ldots, m,
	\end{equation}
	\begin{equation} \label{eq:geodesic.network.constrained.submersion.supp.2}
		V_i \cap V_j = \emptyset \text{ for all } i \neq j \in \{1, \ldots, m \}.
	\end{equation}
	The latter is possible due to \eqref{eq:geodesic.network.constrained.assumption.discrete} and Claim \ref{clai:geodesic.network.constrained.eps.k3}. For each $i = 1, \ldots, m$, \eqref{eq:geodesic.network.constrained.submersion.supp.1} allows us to construct $F_i \in C^k(M)$ compactly supported in $V_i$ with
	\begin{equation} \label{eq:geodesic.network.constrained.submersion.int}
		\int_{\supp S_i(g)} \tfrac12 F_i \, d\ell_g = \ell_i.
	\end{equation} 
	It is easy to see then, using \eqref{eq:geodesic.network.constrained.submersion.derivative}, \eqref{eq:geodesic.network.constrained.submersion.supp.2}, that $h := \sum_{i=1}^m F_i g$ satisfies $D(\vec{\lambda} \circ \vec{\kappa})(g)\{h\} = (\ell_1, \ldots, \ell_m)$, as desired.
\end{proof}

\begin{proof}[Proof of Theorem \ref{theo:geodesic.network.constrained}]
	(1). This is a consequence of Lemma \ref{lemm:geodesic.network.constrained.submersion}.
	
	(2). For this we need to show that $\pi^\Lambda_{G}$ is transversal over $\cC \met^k$, i.e., that for every $(g, S) \in \cS^\Lambda_{G,\cC}$ the composite map
	\begin{multline} \label{eq.geodesic.network.constrained.transversal}
		\Tan_{(g,S)} \cS^\Lambda_{G} \longrightarrow \Tan_g (\met^3(M, g_0, \eps) \cap \met^k(M)) \\
			\longrightarrow \Tan_g (\met^3(M, g_0, \eps) \cap \met^k(M)) / \Tan_g (\cC \met^k)
	\end{multline}
	is surjective and that its kernel splits (\cite[Proposition II.2.4]{Lang:fundamentals.dg}). The latter is of course a consequence of how, by part (1), $\Tan_g (\cC \met^k) = \ker D(\vec{\lambda} \circ \vec{\kappa})(g)$ is a finite codimension (equal to $m$) subspace of $\Tan_g (\met^3(M, g_0, \eps) \cap \met^k(M))$. So what is left to show is surjectivity in \eqref{eq.geodesic.network.constrained.transversal}. 
	
	Fix $(g, S) \in \cS^\Lambda_{G,\cC}$. We work in local coordinates of $\cS^\Lambda_G$ using one of our charts  constructed in the previous section via $\mathfrak{g}$, $\mathfrak{p}$, $\mathfrak{e}$, $\mathring{B}_{\eps}(\mathbf{0})$, $\mathring{Y}$, $\cL$, $\mathring{\cH}$, $\mathring{J}$, $\mathring{K}$, $\Pi$ and $\cS_{\textrm{par}} = \{ \mathring{\cH} = \mathbf{0} \} \subset \met^k(M, \mathfrak{g}, \eps) \times \mathring{B}_\eps(\mathbf{0})$. In these local coordinates, $(g, S)$ becomes $(g, \mathbf{p})$, for some $\mathbf{p} \in \mathring{B}_\eps(\mathbf{0})$.
		
	For every $(h, \mathbf{q}) \in \Tan_{(g, \mathbf{p})} \cS_{\textrm{par}} = \ker D\mathring{\cH}$,
	\[ D(\vec{\lambda} \circ \vec{\kappa})(g)\{ D(\Pi \vert \cS_{\textrm{par}})(g, \mathbf{p})\{h, \mathbf{q}\} \} = D(\vec{\lambda} \circ \vec{\kappa})(g)\{h\}, \]
	so the surjectivity will be a consequence of \eqref{eq:geodesic.network.manifold.submersion.action} and:
	
	\begin{claim} \label{clai:geodesic.network.constrained.submersion.h.construction}
		The restriction of $D(\vec{\lambda} \circ \vec{\kappa})(g)$ to $\ker D_1 \mathring{\cH}(g, \mathbf{p})$ is surjective.
	\end{claim}
	\begin{proof}[Proof of claim]
		We proceed as in Lemma \ref{lemm:geodesic.network.constrained.submersion} except we need to choose our $h$ with some extra care due to the potential interactions between $S_i$ and $S$. Take any $(\ell_1, \ldots, \ell_m) \in \RR^m$. We will construct $F_i \in C^k(M)$ with
		\begin{equation} \label{eq:geodesic.network.constrained.submersion.h.Si}
			D(\vec{\lambda} \circ \vec{\kappa})(g)\{F_i g\} = \ell_i \vec{e}_i,
		\end{equation}
		where $\vec{e}_i \in \RR^m$ is the $i$-th standard basis vector of $\RR^m$, and
		\begin{equation} \label{eq:geodesic.network.constrained.submersion.h.S}
			D_1 \mathring{\cH}(g, \mathbf{p})\{F_i g\} = \mathbf{0}.
		\end{equation}
		The claim clearly follows from \eqref{eq:geodesic.network.constrained.submersion.h.Si}, \eqref{eq:geodesic.network.constrained.submersion.h.S} with $h := \sum_{i=1}^m F_i g$. 
		
		First, let $i = 1, \ldots, m$ be such that $\supp S_i(g) \cap \supp S$ is discrete. Then, take $V_i \subset M$ to be an open set such that
		\begin{equation} \label{eq:geodesic.network.constrained.submersion.h.disc.supp.1}
			V_i \cap S_i(g) \neq \emptyset,
		\end{equation}
		\begin{equation} \label{eq:geodesic.network.constrained.submersion.h.disc.supp.2}
			V_i \cap \supp S = V_i \cap (\cup_{j \neq i} \supp S_i(g)) = \emptyset.
		\end{equation}
		This is easy to do by \eqref{eq:geodesic.network.constrained.assumption.discrete} and Claim \ref{clai:geodesic.network.constrained.eps.k3} when $\supp S_i(g) \cap \supp S$ is discrete. As in the proof of Lemma \ref{lemm:geodesic.network.constrained.submersion}, the combination of \eqref{eq:geodesic.network.constrained.submersion.derivative} and  \eqref{eq:geodesic.network.constrained.submersion.h.disc.supp.1} allows us to construct $F_i$ supported in $V_i$ such that  $D(\lambda_i \circ \kappa_i)(g)\{F_i g\} = \ell_i$. By \eqref{eq:geodesic.network.constrained.submersion.derivative} and \eqref{eq:geodesic.network.constrained.submersion.h.disc.supp.2}, we get \eqref{eq:geodesic.network.constrained.submersion.h.Si} for this $i$ as well as \eqref{eq:geodesic.network.constrained.submersion.h.S}.
		
		Now let $i = 1, \ldots, m$ be such that $\supp S_i(g) \cap \supp S$ is not discrete. Then, by \eqref{eq:geodesic.network.constrained.assumption.regular}, \eqref{eq:geodesic.network.constrained.assumption.discrete}, and Claim \ref{clai:geodesic.network.constrained.eps.k3}, it follows that $\supp S_i(g) \cap \supp S$ contains a segment $\sigma_g(\mathfrak{e}\mathbf{p}(u), \mathfrak{e}\mathbf{p}(v))$, $\{u,v\} \in E$, such that
		\begin{equation} \label{eq:geodesic.network.constrained.submersion.h.nondisc.intersection}
			\sigma_g(\mathfrak{e}\mathbf{p}(u), \mathfrak{e}\mathbf{p}(v)) \cap \supp S_j(g) \text{ is discrete whenever } j \neq i.
		\end{equation}
		We now take $V_i \subset M$ to be an open set such that
		\begin{equation} \label{eq:geodesic.network.constrained.submersion.h.nondisc.supp.1}
			V_i \cap (\operatorname{int} \sigma_g(\mathfrak{e}\mathbf{p}(u), \mathfrak{e}\mathbf{p}(v))) \neq \emptyset,
		\end{equation}
		\begin{multline} \label{eq:geodesic.network.constrained.submersion.h.nondisc.supp.2}
			V_i \cap (\supp S \setminus \sigma_g(\mathfrak{e}\mathbf{p}(u), \mathfrak{e}\mathbf{p}(v))) = V_i \cap (\supp S_i(g) \setminus \sigma_g(\mathfrak{e}\mathbf{p}(u), \mathfrak{e}\mathbf{p}(v))) \\
			= V_i \cap (\cup_{j \neq i} \supp S_i(g)) = \emptyset.
		\end{multline}
		This is possible due to \eqref{eq:geodesic.network.constrained.assumption.regular}, \eqref{eq:geodesic.network.constrained.assumption.discrete}, Claim \ref{clai:geodesic.network.constrained.eps.k3}, and \eqref{eq:geodesic.network.constrained.submersion.h.nondisc.intersection}. Now use \eqref{eq:geodesic.network.constrained.submersion.h.nondisc.supp.1} to construct a $C^k$ function $F_i : M \to \RR$ supported on $V_i$ such that
		\begin{equation} \label{eq:geodesic.network.constrained.submersion.h.nondisc.int}
			\int_{\sigma_g(\mathfrak{e}\mathbf{p}(u), \mathfrak{e}\mathbf{p}(v))} \tfrac12 F_i \, d\ell_g = \ell_i,
		\end{equation}
		and
		\begin{equation} \label{eq:geodesic.network.constrained.submersion.h.nondisc.tang}
			\nabla_g F_i \in \Tan(\operatorname{int} \sigma_g(\mathfrak{e}\mathbf{p}(u), \mathfrak{e}\mathbf{p}(v))) \text{ along } \operatorname{int} \sigma_g(\mathfrak{e}\mathbf{p}(u), \mathfrak{e}\mathbf{p}(v)).
		\end{equation}
		As usual, the combination of \eqref{eq:geodesic.network.constrained.submersion.derivative}, \eqref{eq:geodesic.network.constrained.submersion.h.nondisc.supp.2}, \eqref{eq:geodesic.network.constrained.submersion.h.nondisc.int} implies \eqref{eq:geodesic.network.constrained.submersion.h.Si}, so it only remains to verify \eqref{eq:geodesic.network.constrained.submersion.h.S}. Let $\mathbf{q} \in \mathring{B}_\eps(\mathbf{0})$ be arbitrary and $t \mapsto \mathbf{p}(t) \in B_\eps$ be any $C^1$ path with $\mathbf{p}(0) = \mathbf{p}$, $\dot{\mathbf{p}}(0) = \mathbf{q}$. If $\omega$ denotes the invariant $\omega(\{u,v\})$ along $\{u,v\} \in E_C(G)$, and $(J^{\mathbf{q}}_{u,v})_{\{u,v\} \in E}$ is as in \eqref{eq:geodesic.network.manifold.jacobi.field}, then \eqref{eq:geodesic.network.manifold.submersion.nondegenerate.computation}, \eqref{eq:geodesic.network.constrained.submersion.h.nondisc.supp.2},  \eqref{eq:geodesic.network.constrained.submersion.h.nondisc.tang} imply
		\begin{align*}
			\langle D_1 \mathring{\cH}(g, \mathbf{p})\{F_i g\}, \mathbf{q} \rangle 
				& = \Big[ \tfrac{\partial^2}{\partial s \partial t} \cL(g(s), \mathbf{p}(t)) \Big]_{s=t=0} \\
				& = \tfrac{\omega}{2} \sum_{\{u,v\} \in E_C(G)} \int_{\sigma_g(\mathbf{p}(u), \mathbf{p}(v))} \langle \nabla^\perp_g F, (J^{\mathbf{q}}_{u,v})^\perp \rangle \, d\ell_g = 0.
		\end{align*}
		Since this is true for all $\mathbf{q}$, it follows that $D_1 \cH(g, \mathbf{p})\{F_i g\} = \mathbf{0}$. This completes the proof of \eqref{eq:geodesic.network.constrained.submersion.h.S}, and thus the claim.
	\end{proof}
	
	(3). We proceed as in Lemma \ref{lemm:geodesic.network.manifold.fredholm}. We choose $(g, S) \in \cS^\Lambda_{G,\cC}$ and, as in part (2), we continue to work in local coordinates of $\cS^\Lambda_G$ near $(g, S)$ using one of our charts as constructed in the previous section. Abusing notation, let us still write $\cC \met^k$, $\vec{\lambda}$, $\vec{\kappa}$ for the local objects. In coordinates, $\ker D\pi^\Lambda_{G,\cC}$ equals:
	\begin{align} \label{eq:geodesic.network.constrained.ker.pi}
		& \ker D(\Pi \vert (\cS_{\textrm{par}} \cap \Pi^{-1}(\cC \met^k)))(g, \mathbf{p}) \\
		& \qquad = \ker \Pi \cap \Tan_{(g, \mathbf{p})} \cS_{\textrm{par}} \cap \Pi^{-1}(\Tan_g (\cC \met^k)) \nonumber \\
		& \qquad = \ker \Pi \cap \ker D\mathring{\cH}(g, \mathbf{p}) \cap \Pi^{-1}(\ker D(\vec{\lambda} \circ \vec{\kappa})(g)) \nonumber \\
		& \qquad = (\{0\} \times \mathring{K}) \cap (\ker D(\vec{\lambda} \circ \vec{\kappa})(g) \times \mathring{Y}) = \{0\} \times \mathring{K}. \nonumber
	\end{align}
	It remains to show that $\image D\pi^\lambda_{G,\cC}(g,S)$ has codimension $\dim \mathring{K}$ in $\Tan_g (\cC \met^k)$. In coordinates, the prior equals:
	\begin{align*}
		& \image D(\Pi \vert (\cS_{\textrm{par}} \cap \Pi^{-1}(\cC \met^k)))(g, \mathbf{p}) \\
		& \qquad = \Pi(\Tan_{(g, \mathbf{p})} \cS_{\textrm{par}}) \cap \Tan_g (\cC \met^k) \\
		& \qquad = \ker(\pi_{\mathring{K}} \circ D_1 \cH)(g, \mathbf{p}) \cap \ker D(\vec{\lambda} \circ \vec{\kappa})(g)
	\end{align*}
	The result follows from the following:
	\begin{claim} \label{clai:geodesic.network.constrained.submersion.full.img}
		$\pi_{\mathring{K}} \circ D_1 \mathring{\cH}(g, \mathbf{p})$ restricted to $\ker D(\vec{\lambda} \circ \vec{\kappa})(g)$ remains surjective.
	\end{claim}

	To prove it, we follow the same proof strategy as in Claim \ref{clai:geodesic.network.manifold.submersion.full.img}, which had been split up in three subclaims. The first, Claim \ref{clai:geodesic.network.manifold.jacobi.field.nontrivial}, will be applied unchanged. The second and third need to be replaced:
		
	\begin{claim}[cf. Claim \ref{clai:geodesic.network.manifold.submersion.nondegenerate}]  \label{clai:geodesic.network.constrained.submersion.nondegenerate}
		Suppose $\mathbf{q} \in \mathring{Y}$ satisfies \eqref{eq:geodesic.network.manifold.jacobi.field.nontrivial}. Then, 
		\begin{equation} \label{eq:geodesic.network.constrained.submersion.nondegenerate}
			\big[ \tfrac{\partial^2}{\partial s \partial t} \cL(g(s), \mathbf{p} + t \mathbf{q}) \big]_{s=t=0} \neq \mathbf{0}
		\end{equation}		
		for at least one $C^1$ path $s \mapsto g(s)$ with $g(0) = g$ and $\dot{g}(0) \in \ker D(\vec{\lambda} \circ \vec{\kappa})(g)$.
	\end{claim}
	\begin{proof}[Proof of claim]
		We proceed as in Claim \ref{clai:geodesic.network.manifold.submersion.nondegenerate} except we need to watch out for potential interactions between $S_i$ and $S$. However, as long as 
		\begin{equation} \label{eq:geodesic.network.constrained.submersion.nondegenerate.zero}
			\dot{g}(0) \equiv 0 \text{ along } \cup_{i=1}^m \supp S_i(g),
		\end{equation}
		then the constraint $\dot{g}(0) \in \ker D(\vec{\lambda} \circ \vec{\kappa})(g)$ is guaranteed by \eqref{eq:geodesic.network.constrained.submersion.derivative}.
		
		By virtue of \eqref{eq:geodesic.network.manifold.jacobi.field.nontrivial}, there exists $\{u,v\} \in E$ along which the Jacobi field $J^{\mathbf{q}}_{u,v}$ of \eqref{eq:geodesic.network.manifold.jacobi.field} has a nontrivial normal component $(J^{\mathbf{q}}_{u,v})^\perp$. Set:
		\[ g(s) := (1 + sF) g, \]
		where $F : M \to \RR$ is $C^k$ is supported on an open set $U \subset M$ satisfying
		\begin{equation} \label{eq:geodesic.network.constrained.submersion.U.1}
			\emptyset \neq U \cap \supp S \Subset \operatorname{int} \sigma_g(\mathfrak{e}\mathbf{p}(u), \mathfrak{e}\mathbf{p}(v)) \cap \{ (J^{\mathbf{q}}_{u,v})^\perp \neq 0 \}.
		\end{equation}
		
		If $\operatorname{int} \sigma_g(\mathfrak{e}\mathbf{p}(u), \mathfrak{e}\mathbf{p}(v)) \cap (\cup_{i=1}^m \supp S_i(g))$ is discrete, then we may take
		\begin{equation} \label{eq:geodesic.network.constrained.submersion.U.2}
			U \cap (\cup_{i=1}^m \supp S_i(g)) = \emptyset,
		\end{equation}
		and the remainder of the proof of Claim \ref{clai:geodesic.network.manifold.submersion.nondegenerate} applies verbatim with $U$ satisfying \eqref{eq:geodesic.network.constrained.submersion.U.1}, \eqref{eq:geodesic.network.constrained.submersion.U.2} instead of just \eqref{eq:geodesic.network.manifold.submersion.U}. Note that \eqref{eq:geodesic.network.constrained.submersion.U.2} guarantees \eqref{eq:geodesic.network.constrained.submersion.nondegenerate.zero}, and the proof of claim is complete.
		
		So now assume $\operatorname{int} \sigma_g(\mathfrak{e}\mathbf{p}(u), \mathfrak{e}\mathbf{p}(v)) \cap (\cup_{i=1}^m \supp S_i(g))$ isn't discrete. Note that \eqref{eq:geodesic.network.constrained.assumption.discrete} and Claim \ref{clai:geodesic.network.constrained.eps.k3} allow us to shrink $U$ so that, for some $i \in \{1, \ldots, m\}$:
		\begin{equation} \label{eq:geodesic.network.constrained.submersion.U.3}
			U \cap (\cup_{j \neq i} \supp S_j(g)) = U \cap (\supp S_i(g) \setminus \sigma_g(\mathfrak{e}\mathbf{p}(u), \mathfrak{e}\mathbf{p}(v)) = \emptyset.
		\end{equation}
		Recall that, by \eqref{eq:geodesic.network.manifold.submersion.nondegenerate.computation}, 
		\begin{equation} \label{eq:geodesic.network.constrained.submersion.nondegenerate.computation}
			\big[ \tfrac{\partial^2}{\partial s \partial t} \cL(g(s), \mathbf{p} + t \mathbf{q}) \big]_{s=t=0} = \tfrac{\omega(\{u,v\})}{2} \int_{\sigma_g(\mathfrak{e}\mathbf{p}(u), \mathfrak{e}\mathbf{p}(v))} \langle \nabla^\perp_g F, (J^{\mathbf{q}}_{u,v})^\perp \rangle \, d\ell_g, \nonumber
		\end{equation}		
		where $\perp$ denotes orthogonal projection (with respect to $g$) to the normal bundle of $\sigma_g(\mathfrak{e}\mathbf{p}(u), \mathfrak{e}\mathbf{p}(v))$. We choose $F = 0$ along $\operatorname{int} \sigma_g(\mathfrak{e}\mathbf{p}(u), \mathfrak{e}\mathbf{p}(v))$, so that \eqref{eq:geodesic.network.constrained.submersion.U.1}, \eqref{eq:geodesic.network.constrained.submersion.U.3} imply \eqref{eq:geodesic.network.constrained.submersion.nondegenerate.zero}. We additionally require that $\langle \nabla^\perp_g F, (J^{\mathbf{q}}_{u,v})^\perp \rangle_g \geq 0$ along  $\sigma_g(\mathfrak{e}\mathbf{p}(u), \mathfrak{e}\mathbf{p}(v))$, with strict inequality at some interior point as is allowed by \eqref{eq:geodesic.network.constrained.submersion.U.1}. Then,  \eqref{eq:geodesic.network.constrained.submersion.nondegenerate.computation}  implies \eqref{eq:geodesic.network.constrained.submersion.nondegenerate}. 
	\end{proof}
	
	\begin{claim}[cf. Claim \ref{clai:geodesic.network.manifold.submersion.degenerate}]  \label{clai:geodesic.network.constrained.submersion.degenerate}
		Suppose
		\[ \mathbf{q} \in \mathring{K} \cap (\image \pi_{\mathring{K}} \circ D_1 \mathring{\cH}(g, \mathbf{p}) \vert \ker D(\vec{\lambda} \circ \vec{\kappa})(g))^\perp. \]
		Then
		\[ \big[ \tfrac{\partial^2}{\partial s \partial t} \cL(g(s), \mathbf{p} + t \mathbf{q}) \big]_{s=t=0} = \mathbf{0} \]
		for all $C^1$ paths $s \mapsto g(s)$ with $g(0) = g$ and $\dot{g}(0) \in \ker D(\vec{\lambda} \circ \vec{\kappa})(g)$.
	\end{claim}
	\begin{proof}[Proof of claim]
		The computation in the proof of Claim \ref{clai:geodesic.network.manifold.submersion.degenerate} applies.
	\end{proof}
 
 	\begin{proof}[Proof of Claim \ref{clai:geodesic.network.manifold.submersion.full.img}]
 		The claim follows as before from Claims \ref{clai:geodesic.network.manifold.jacobi.field.nontrivial}, \ref{clai:geodesic.network.constrained.submersion.nondegenerate}, \ref{clai:geodesic.network.constrained.submersion.degenerate}.
	\end{proof}

	(4). This is a consequence of \eqref{eq:geodesic.network.constrained.ker.pi} and Theorem \ref{theo:geodesic.network.manifold}'s (3).
	
	(5). This is a consequence of the Sard--Smale theorem \cite{Smale:sard}, Theorem \ref{theo:geodesic.network.constrained}'s (1), (2), (3), and the finiteness of $\# \cG(\Lambda, Q)$.
\end{proof}

\begin{proof}[Proof of Corollary \ref{coro:geodesic.network.constrained.generic.smooth}]

This follows similarly to Corollary \ref{coro:geodesic.network.manifold.generic.smooth}. One can invoke the proof of \cite[Lemma 7.2]{Staffa:bumpy.geodesic.nets} verbatim with $\cM^k = \cC \met^k$, $\cM^\infty = \cC \met$, $\cN^k =$ regular values of $\pi^\Lambda_{G,\cC}$ in $\cC \met^k$ (i.e., $\cR^{\Lambda,k}_{G,\cC}$), and $\cN^\infty =$ regular values of $\pi^\Lambda_{G,\cC}$ in $\cC \met$. The key is that $\cN^{k'} = \cN^k \cap \cC \met^{k'}$ whenever $k' \geq k$ by \eqref{eq:geodesic.network.constrained.regularity.R} and that $\cM^\infty \subseteq \cM^k$ is still dense, i.e.,
	\[ \cC \met \subset \cC \met^k \text{ is dense}. \]
%
	To see this, fix $g \in \cC \met^k$. By the denseness of $\met(M) \subset \met^k(M)$, there exists a sequence $\{ g_n \}_{n=1}^\infty \subset \met(M)$ such that $g_n \to g$ in $\met^k(M)$. For $n$ large and $u \in C^\infty(M)$ small in $C^k$ so that $\vec{\lambda} \circ \vec{\kappa}$ is well-defined at $g_n$ and $(1+u)^2 g_n$, note that, for every $i \in \{ 1, \ldots, m \}$:
	\begin{equation} \label{eq:geodesic.network.constrained.generic.smooth.u.length}
		\int_{\supp S_i(g_n)} d\ell_{(1+u)^2 g_n} = (\lambda_i \circ \kappa_i)(g_n) + \int_{\supp S_i(g_n)} u \, d\ell_{g_n}.
	\end{equation}
	Since $(\lambda_i \circ \kappa_i)(g_n) \to (\lambda_i \circ \kappa_i)(g)$ as $n \to \infty$, and $\supp S_i(g_n)$ is $C^\infty$, there exist $u = u_n \in C^\infty(M)$ such that $u_n \to 0$ in $C^k(M)$ as $n \to \infty$ and
	\begin{equation} \label{eq:geodesic.network.constrained.generic.smooth.un.length}
		\int_{\supp S_i(g_n)} u_n \, d\ell_{g_n} = (\lambda_i \circ \kappa_i)(g) - (\lambda_i \circ \kappa_i)(g_n) \text{ for all } i = 1, \ldots, m,	
	\end{equation}
	\begin{equation} \label{eq:geodesic.network.constrained.generic.smooth.un.geodesic}
		\nabla_{g_n} u_n = 0 \text{ along } \supp S_i(g_n) \text{ for all } i = 1, \ldots, m.
	\end{equation}
	It is easiest to have $u_n$ be compactly supported away from the pairwise intersections of the $S_i$, which are discrete by \eqref{eq:geodesic.network.constrained.assumption.discrete}.
	
	It follows from \eqref{eq:geodesic.network.constrained.generic.smooth.un.geodesic} that all $\supp S_i(g_n)$ are geodesics in $(M, g_n)$ as well as in $(M, (1+u_n)^2 g_n)$. Moreover, since $u_n \to 0$ in $C^k$ as $n \to \infty$, for each $i \in \{1, \ldots, m\}$ we see that $\supp S_i(g_n)$ and $\supp S_i((1+u_n)^2 g_n)$ are $C^{k+1}$-$o(1)$-close as $n \to \infty$. However, by our definition of $\vec{\kappa}$, $\vec{\lambda}$ (which hinges on the non-degeneracy of $S^0_i$) it follows then that $S_i((1+u_n)^2 g_n) = S_i(g_n)$ as integral varifolds, for all $i \in \{1, \ldots, m\}$. Thus, by \eqref{eq:geodesic.network.constrained.generic.smooth.u.length}, \eqref{eq:geodesic.network.constrained.generic.smooth.un.length}, 
	\[ (\vec{\lambda} \circ \vec{\kappa})((1+u_n)^2 g_n) = (\vec{\lambda} \circ \vec{\kappa})(g). \]
	Thus, $(1+u_n)^2 g_n \in \cC \met^k$. This completes the proof of the denseness of $\cC \met \subset \cC \met^k$, and thus the corollary.
\end{proof}

\section{Lusternik--Schnirelmann theory on a perturbed $(\SS^2, g_0)$}\label{sec:LS}

\subsection{Choosing good metrics on $\SS^2$}

Consider the ellipsoids
\[ E(a_1, a_2, a_3) := \{ (x_1, x_2, x_3) \in \RR^3 : a_1 x_1^2 + a_2 x_2^2 + a_3 x_3^2 = 1 \} \subset \RR^3 \]
and the three geodesics
\[ \gamma_i(a_1, a_2, a_3) := E(a_1, a_2, a_3) \cap \{ x_i = 0 \}, \; i = 1, 2, 3 \]
on them. It is shown in \cite[Theorems IX 3.3, 4.1]{Morse:calculus.variations} that for every $\Lambda > 2\pi$, if $a_1 < a_2 < a_3$ are sufficiently close to $1$ (depending on $\Lambda$), then every closed connected immersed geodesic $\gamma \subset E(a_1, a_2, a_3)$ (coverings allowed) satisfies:
\begin{equation} \label{eq:geodesic.network.ellipsoid.geodesic.genericity}
	\gamma \text{ has no nontrivial normal Jacobi fields if } \length(\gamma) < 2\Lambda,
\end{equation} 
and
\begin{equation} \label{eq:geodesic.network.ellipsoid.geodesic.uniqueness}
	\gamma \text{ is an iterate of one of } \gamma_i(a_1, a_2, a_3), \; i = 1, 2, 3, \text{ if } \length(\gamma) < 2\Lambda.
\end{equation}
For $(a_1, a_2, a_3) \in \RR^3$, consider the vector $\vec{\ell}(a_1, a_2, a_3) \in \RR^3$ whose $i$-th component, $i=1, 2, 3$, is
\[ \big[ \vec{\ell}(a_1, a_2, a_3) \big]_i := \length(\gamma_i(a_1, a_2, a_3)). \]
It is easy to see that $(a_1, a_2, a_3) \mapsto \vec{\ell}(a_1, a_2, a_3)$ is smooth near $(1, 1, 1)$, and
\[ D\vec{\ell}(1, 1, 1) = \begin{pmatrix} 0 & \pi & \pi \\ \pi & 0 & \pi \\ \pi & \pi & 0 \end{pmatrix}. \]
Then, by the inverse function theorem there are smooth functions $\mu \mapsto a_i(\mu)$, $a_i(0) = 1$, $i = 1, 2, 3$, so that, for $\mu$ near $0$,
\[ \vec{\ell}(a_1(\mu), a_2(\mu), a_3(\mu)) = (2\pi, 2\pi + \mu, 2\pi+2\mu). \]
Thus, for sufficiently small $\mu$, 
\begin{equation} \label{eq:geodesic.network.ellipsoid.geodesic.lengths}
	\length(\gamma_i(a_1, a_2, a_3)) = 2\pi + (i-1) \mu.
\end{equation}
Then Corollary \ref{coro:geodesic.network.constrained.generic.smooth} implies:

\begin{theo}[Choosing a good metric] \label{theo:good.metric.final}
	Let $\Lambda > 0$ and $U$ be any neighborhood, in the $C^\infty$ topology, of the unit round metric $g_0 \in \met(\SS^2)$. There is a $\mu_0 = \mu_0(\Lambda, U) > 0$, so that for all $\mu \in (0, \mu_0)$, there exists $g_\mu \in U$ with all these properties:
	\begin{enumerate}
		\item There are simple closed geodesics $\gamma_1$, $\gamma_2$, $\gamma_3 \subset (\SS^2, g_\mu)$ so that $\length_{g_\mu}(\gamma_i) = 2\pi + (i-1) \mu$.
		\item If a closed connected geodesic in $(\SS^2, g_\mu)$ has $\length_{g_\mu} < \Lambda$, then it is an iterate of $\gamma_i$, $i = 1$, $2$, $3$.
		\item There are no not-everywhere-tangential stationary varifold Jacobi fields along any $S \in \cS^{\Lambda}(g_\mu)$.
	\end{enumerate}
	Moreover, $g_\mu \to g_0$ as $\mu \to 0$ in the $C^\infty$ topology.
\end{theo}
\begin{proof}
	We choose $\mu_0$ small enough that $g^E_\mu \in U$ for all $\mu \in (0, \mu_0)$, where $g^E_\mu$ is the  induced metric of $E(a_1(\mu), a_2(\mu), a_3(\mu)) \subset \RR^3$.
	
	Fix any such $\mu$. By \eqref{eq:geodesic.network.ellipsoid.geodesic.genericity}, Corollary \ref{coro:geodesic.network.constrained.generic.smooth} applies at $g^E_\mu$ with $2\Lambda$ in place of $\Lambda$ and with $S_i^0$ the multiplicity-one varifolds on $\gamma_i(a_1(\mu), a_2(\mu), a_3(\mu))$, $i = 1$, $2$, $3$. As a consequence of Corollary \ref{coro:geodesic.network.constrained.generic.smooth} and Remark \ref{rema:geodesic.network.constrained.regularity}, there is a neighborhood $V \subset U$ of $g^E_\mu$ inside of which there is a dense set $D \subset V$ of metrics satisfying conclusions (1) and (3) of our theorem; denote the distinguished geodesics by $\gamma_j(g)$, $j = 1$, $2$, $3$. 
	
	It remains to show that at least one of them satisfies conclusion (2) too and can be taken arbitrarily close to $g_\mu^E$. Suppose that were not the case. Take any sequence $\{ g_\mu^i \}_{i=1}^{\infty} \subset D$ with $g_\mu^i \to g^E_\mu$ as $i \to \infty$. Let $\gamma_\mu^i$ be a closed connected geodesic in $(\SS^2, g_\mu^i)$ with $\length_{g_\mu^i}(\gamma_\mu^i) < \Lambda$ that is not an iterate of any of $\gamma_j(g_\mu^i)$, $j = 1, 2, 3$. Pass to $i \to \infty$ along a subsequence (not relabeled) so that $\gamma^i_\mu \to \gamma^E_\mu$, a geodesic in $(\SS^2, g_\mu^E)$ with $\length_{g^E_\mu}(\gamma^E_\mu) \leq \Lambda < 2 \Lambda$. By \eqref{eq:geodesic.network.ellipsoid.geodesic.uniqueness}, $\gamma^E_\mu$ is an iterate of a $\gamma_j(g^E_\mu)$, $j = 1, 2, 3$. So, $\gamma^i_\mu$ is $o(1)$-close (as $i \to \infty$) to being an iterate of $\gamma_j(g^i_\mu)$. However, recall that the three $\gamma_j(g_\mu)$ and their iterates with length $< 2 \Lambda$ are isolated in $(\SS^2, g_\mu)$ by \eqref{eq:geodesic.network.ellipsoid.geodesic.uniqueness}. Therefore, the geodesics $\gamma_j(g^i_\mu)$ and their iterates with length $\leq \Lambda$ are isolated in $(\SS^2, g^i_\mu)$, contradicting the existence of $\gamma^i_\mu$ when $i$ is large. This completes the proof.
\end{proof}

\begin{lemm}\label{lemm:discrete-widths}
	Fix $p \in \NN^*$. There are $\mu_1 > 0$ and an open neighborhood $U$ of the round metric $g_0$ on $\SS^2$, depending on $p$, so that, for every $\mu \in (0,\mu_1)$:
	\begin{enumerate}
		\item $\omega_p(\SS^2,g_\mu) \leq 2\pi \lfloor \sqrt{p} \rfloor +1$.
		\item For any $\cF$-homotopy class $\Pi \subset \cP_{p,m}^{\bF}$, $m \in \NN^{*}$,
			\begin{multline*}
				\bL_{\textrm{AP}}(\Pi,g_\mu) \in \Big( \{2\pi(n_1+n_2+n_3) + \mu(n_2+2n_3) : (n_1,n_2,n_3) \in \NN^3 \} \setminus \{0\} \Big) \\
				\cup \big[ 2\pi \lfloor \sqrt{p} \rfloor +2,\infty \big). 
			\end{multline*}
	\end{enumerate}
\end{lemm}
\begin{proof}
By an argument of Aiex (see Corollary \ref{coro:upper.bds}), we can consider sweepouts constructed from the zero sets on $\SS^2$ of degree $\leq \lfloor \sqrt{p} \rfloor$ polynomials on $\RR^3$, yielding
\[
\omega_p(\SS^2,g_{\SS^2}) \leq 2\pi \lfloor \sqrt{p} \rfloor. 
\]
Therefore, by Lemma \ref{lem:p-width-cont-metric}, 
\[
\omega_p(\SS^2,g) \leq 2\pi \lfloor \sqrt{p} \rfloor + 1. 
\]
as long as $g \in U$ and $U$ is small. This fixes $U$. Set
\[
\mu_1 :=  \mu_0(2\pi \lfloor \sqrt{p} \rfloor+2, U),
\] 
for $\mu_0$ defined in Theorem \ref{theo:good.metric.final}. This completes the proof of (1).

Next, pick any $\cF$-homotopy class $\Pi \subset \cP_{p,m}^{\bF}$ with $\bL_\textrm{AP}(\Pi) < 2\pi \lfloor \sqrt{p} \rfloor + 2$. By Proposition \ref{prop:AC.min.max}, Proposition \ref{prop:AC-AP}, and Theorem \ref{theo:SG-lim}, there are closed connected geodesics $\sigma_1, \ldots, \sigma_N$ so that
\[
\bL_\textrm{AP}(\Pi) = \sum_{j=1}^N \length_{g_\mu} (\sigma_j).
\]
In particular, for every $j = 1, \ldots, N$, $\length_{g_\mu}(\sigma_j) < 2\pi \lfloor \sqrt{p} \rfloor + 2$ so, by Theorem \ref{theo:good.metric.final} parts (1) and (2), $\sigma_j$ is an $s_j$-time iterate of some $\gamma_{i_j}$ with $i_j \in \{ 1, 2, 3 \}$ and $s_j \in \{ 1, 2, \ldots \}$, and thus has $\length_{g_\mu}(\sigma_j) = s_j(2\pi + (i_j - 1) \mu)$. This completes the proof. 
\end{proof}
\begin{coro}[{cf.\ \cite[Proposition 4.8]{MarquesNeves:posRic}}]\label{coro:exist-cub-complex-p-width}
For $p\in \NN^*$ and $\mu \in (0,\mu_1)$, there exists $X \subset I^{2p+1}$ and an $\cF$-homotopy class $\Pi \subset \cP_{p,2p+1}^{\bF}$ so that 
\[
\bL_{\textrm{AP}}(\Pi,g_\mu) = \omega_p(\SS^2,g_\mu)
\]
\end{coro}
\begin{proof}
By Lemma \ref{lemm:P.p.m}, there are $X_k\subset I^{2p+1}$ and homotopy classes $\Pi_k$ so that
\[
\lim_{k\to\infty} \bL_\textrm{AP}(\Pi_k,g_\mu) = \omega_p(\SS^2,g_\mu). 
\]
By Lemma \ref{lemm:discrete-widths}, $k\mapsto \bL_\textrm{AP}(\Pi_k,g_\mu)$ eventually stabilizes. 
\end{proof}

\subsection{Lusternik--Schnirelmann category covering lemma} 
It will be crucial for our proof of Proposition \ref{prop:LS} to know that, for the metric produced by Theorem \ref{theo:good.metric.final}, a certain set of candidate min-max objects have Lusternik--Schnirelmann category equal to zero. See Remark \ref{rema:cover.category}. Our proof involves a new covering lemma, and adaptations of arguments used in \cite[Theorem 6.1]{MarquesNeves:posRic}, \cite[Appendix A]{Aiex:ellipsoids}.

Fix an arbitrary closed Riemannian $2$-manifold $(M, g)$ throughout the section. We define
\[
\cT : 2^{\cI \cV_1(M)} \to 2^{\cZ_2(M;\ZZ_2)}
\]
by setting
\[
\cT : \cS \mapsto \{T \in \cZ_2(M;\ZZ_2) : \supp T \subset \supp V \textrm{ for some } V \in \cS\}.
\]

\begin{lemm}\label{lemm:cont-cT-operation}
Consider $V_i,V \in \cI \cV_1(M)$ with $\bF(V_i,V)\to0$. If $T_i \in \cT(V_i)$ then up to passing to a subsequence, there is $T\in \cT(V)$ with $\cF(T_i,T) \to 0$. 
\end{lemm}
\begin{proof}
Note that $\bM(T_i) \leq \Vert V_i\Vert(M) = \Vert V\Vert(M) + o(1)$ so we can pass to a subsequence and find $T\in\cZ_2(M;\ZZ_2)$ with $\cF(T_i,T) \to 0$. By lower semi-continuity of mass, $\supp T \subset \supp V$, i.e., $T\in \cT(V)$. This completes the proof. 
\end{proof}

Below, we will write $\Lim(\cdot)$ to denote the limit points of a set, and $\Lim^{(k)}(\cdot)$ to denote the $k$-times iterated limit points. 

\begin{lemm}\label{lemm:lim.curr.lim.var}
Consider a compact subset $\cS\subset \cI\cV_1(M)$ with all $V \in S$ stationary. Then $\cT(\cS)$ is compact with respect to the $\cF$-topology and
\[
\Lim(\cT(\cS)) \subset \cT(\Lim(\cS))
\]
in the corresponding topologies.
\end{lemm}
\begin{proof}
Compactness of $\cT(\cS)$ follows from the assumed compactness of $\cS$ and Lemma \ref{lemm:cont-cT-operation}. Now, consider $T \in \Lim(\cT(\cS))$. Fix $\{T_i\}_{i=1}^\infty \subset \cT(\cS)\setminus\{T\}$ with $\cF(T_i,T) \to 0$. Since $T_i \in \cT(\cS)$ there are $V_i \in \cS$ with $\supp T_i\subset \supp V_i$. By compactness of $\cS$, we can pass to a subsequence so that $\bF(V_i,V) \to 0$ for some $V\in\cS$. By Lemma \ref{lemm:cont-cT-operation}, $T\in \cT(V)$. If $V_i\neq V$ for infinitely many $i$, then $V \in \Lim(\cS)$. As such, it remains to consider the case that $V_i = V$ for all $i$. Since $V \in \cI\cV_1(M)$ is stationary, \cite[Section 3]{AllardAlmgren:1varifold} implies that $V$ is supported on a fixed geodesic net with finitely many singularities. The constancy theorem for currents implies that $\cT(V)$ is a finite set. Since $\{T_i\}_{i=1}^\infty \subset \cT(V) \setminus \{T\}$, this is a contradiction. 
\end{proof}

\begin{lemm}\label{lemm:cover.cF.bF}
Consider a compact subset $\cS\subset \cI\cV_1(M)$ with all $V\in\cS$ stationary and $\Lim^{(N)}(\cS) = \emptyset$ for some $N \in \NN^*$. Fix $\eps>0$. There a finite subset $\{V_1,\dots,V_k\} \subset \cS$ and positive numbers $\eta_1,\dots,\eta_k,\eps_1,\dots,\eps_k$ with $0 < \eps_i < \eps$ so that the following properties hold:
\begin{enumerate}
\item Covering in the $\bF$ topology:
\[
\cS \subset \bigcup_{i=1}^k B_{\eta_i}^\bF(V_i),
\]
\item Covering in a mixed $\bF$/$\cF$ sense:
\[
\{T \in \cZ_1(M;\ZZ_2) : |T| \in \cup_{i=1}^k B_{2\eta_i}^\bF(V_i)\} \subset \bigcup_{i=1}^k B_{\eps_i}^\cF(\cT(V_i)),
\]
\item $\cF$-balls are disjoint or contained in each other:
\begin{multline*}
	i_1, i_2 \in \{1, \ldots, k\}, \; T_1 \in \cT(V_{i_1}), \; T_2 \in \cT(V_{i_2}) \\
	\implies B^{\cF}_{\eps_{i_1}}(T_1) \cap B^{\cF}_{\eps_{i_2}}(T_2) = \emptyset, \text{ or } B^{\cF}_{\eps_{i_1}}(T_1) \subset B^{\cF}_{\eps_{i_2}}(T_2),\\  \text{ or } B^{\cF}_{\eps_{i_1}}(T_1) \supset B^{\cF}_{\eps_{i_2}}(T_2).
\end{multline*}
\end{enumerate}
\end{lemm}

\begin{rema} \label{rema:cover.category}
	This lemma implies that $\cT(\cS)$ has Lusternik--Schnirelmann $\operatorname{cat}(\cT(\cS)) = 0$. Indeed, the balls $B^{\cF}_{\eps_i}(T)$, $T \in \cT(\{ V_1, \ldots, V_k\})$, are homotopically trivial when $\eps > 0$ is small; see \cite[Proposition 3.3]{MarquesNeves:posRic}.
\end{rema}

The lemma readily implies:

\begin{coro} \label{coro:cover.cF.bF}
	In the setting above, if $\Phi : \SS^1 \to \cZ_1(M;\ZZ_2)$ is continuous and
	\[ |\Phi(t)| \subset \cup_{i=1}^k B_{2\eta_i}^\bF(V_i) \text{ for all } t \in \SS^1, \]
	then 
	\[ \Phi(\SS^1)\subset B_\eps^\cF(T) \text{ for some } T \in \cT(\{ V_1, \ldots, V_k \}). \]
\end{coro}

\begin{proof}
	Fix $\Phi : \SS^1 \to \cZ_1(M;\ZZ_2)$ as above, and let $t_0\in \SS^1$. By conclusion (2) of the lemma, $\Phi(t_0) \in B_{\eps_i}^\cF(T)$ for some $T \in \cT(V_i)$, $i \in \{ 1, \ldots, k \}$. By conclusion (3) of the lemma, by possibly taking a different choice of $i$, we can assume that the following ``maximality'' condition holds: for any $i'\in\{1,\dots,k\}$ and $T'\in\cT(V_{i'})$, either $B_{\eps_{i'}}^\cF(T') \cap B_{\eps_{i}} ^\cF(T) = \emptyset$ or $B_{\eps_{i'}}^\cF(T') \subset B_{\eps_{i}}^\cF(T)$. If $\Phi(\SS^1) \not \subset B_{\eps_i}^\cF(T)$ then there is $\{t_m\}_{m=1}^\infty\subset \SS^1$ with $t_m \to t_\infty$ and $\Phi(t_m) \in B_{\eps_i}^\cF(T)$, but $\Phi(t_\infty) \not \in B_{\eps_i}^\cF(T)$. By conclusion (2) of the lemma, $\Phi(t_\infty)  \in B_{\eps_{i'}}^\cF(T')$ for some $T' \in \cT(V_{i'})$, $i' \in \{ 1, \ldots, k \}$. By the maximality property arranged above, it must hold that $B_{\eps_i}^\cF(T) \cap B_{\eps_{i'}}^\cF(T') = \emptyset$, so $\Phi(t_m) \not \in B_{\eps_{i'}}^\cF(T')$. This contradicts $\cF(\Phi(t_m), \Phi(t_\infty)) \to 0$.
\end{proof}

\begin{proof}[Proof of Lemma \ref{lemm:cover.cF.bF}] 
We cover $\cS$ backwards, starting from $\Lim^{(N-1)}(\cS)$ and working down to $\Lim^{(0)}(\cS) = \cS$. In fact, we will prove by induction on $\ell \in \{1,\dots,N\}$ that we can choose
\[
\{V_1,\dots,V_{k_{\ell}}\}\subset \cS
\]
and $\eta_1,\dots,\eta_{k_{\ell}},\eps_1,\dots,\eps_{k_{\ell}} > 0$, with $0< \eps_i < \eps$, so that:
\begin{enumerate}
\item[($1_\ell$)] $\eps_i \not \in \{\cF(T,T') : T \in \cT(\cS),T'\in \cT(V_i)\}$ for all $i\in\{1,\dots,k_{\ell}\}$.
\item[($2_\ell$)] $\Lim^{(N-\ell)}(\cS) \subset \cup_{i=1}^{k_{\ell}} B_{\eta_i}^\bF(V_i)$.
\item[($3_\ell$)] For all $T\in \cZ_1(M;\ZZ_2)$, $|T| \in \cup_{i=1}^{k_{\ell}} B_{2\eta_i}^\bF(V_i) \implies T \in \cup_{i=1}^{k_\ell} B_{\eps_i}^\cF(\cT(V_i))$. 
\item[($4_\ell$)] For all $i_1,i_2 \in \{1,\dots,k_{\ell}\}$, $T_1 \in \cT(V_{i_1})$, $T_2\in\cT(V_{i_2})$, either:
\begin{enumerate}
\item [($4_\ell.a$)] $B^\cF_{\eps_{i_1}}(T_1) \cap B^\cF_{\eps_{i_2}}(T_2) = \emptyset$,
\item [($4_\ell.b$)] $B^\cF_{\eps_{i_1}}(T_1) \subset B^\cF_{\eps_{i_2}}(T_2)$, or
\item [($4_\ell.c$)] $B^\cF_{\eps_{i_1}}(T_1) \supset B^\cF_{\eps_{i_2}}(T_2)$.
\end{enumerate}
\end{enumerate}
Notice that $(2_N)$-$(4_N)$ imply the statement of the lemma, with $k := k_N$.

We proceed with the induction and start with the base case, $\ell = 1$. Note that $\Lim^{(N-1)}(\cS)$ is compact with no limit points since $\Lim(\Lim^{(N-1)}(\cS))= \Lim^{(N)}(\cS) = \emptyset$. Therefore, $\Lim^{(N-1)}(\cS)$ is finite, so we may write
\begin{equation} \label{eq:cover.lemm.base.case.limset}
	\Lim^{(N-1)}(\cS) =: \{V_1,\dots,V_{k_1}\}.
\end{equation}
As in Lemma \ref{lemm:lim.curr.lim.var}, $\cT_1 : = \cT(\Lim^{(N-1)}(\cS))$ is finite. For $i \in \{ 1, \ldots, k_1 \}$, choose 
\begin{equation}\label{eq:cover.lemm.base.case.flat.balls.disjoint}
0 < \eps_i < \min\{\eps,\min\{\cF(T_1,T_2) : T_1\neq T_2\in \cT_1\}/3\}
\end{equation}
so that 
\begin{equation}\label{eq:cover.lemm.base.case.not.in.bdry.eps}
\eps_i \not \in \{\cF(T,T') : T\in \cT(\cS), \; T'\in \cT(V_i)) \}.
\end{equation}
Such a choice is possible since (i) $\cT_1$ is finite and (ii) $\{\cF(T,T') : T\in \cT(\cS), \; T'\in \cT(V_i)\}$ is at most countable, by combining Lemma \ref{lemm:lim.curr.lim.var} with Corollary \ref{coro:countable}. 
\begin{claim}
There is $\eta>0$ so that, for every $T \in \cZ_1(M;\ZZ_2)$,
\[
|T| \in B^{\bF}_{2\eta}(\Lim^{(N-1)}(\cS)) \implies T \in \bigcup_{i=1}^{k_1} B_{\eps_i}^\cF(\cT(V_i)). 
\]
\end{claim}
\begin{proof}[Proof of claim] 
If not, there is $\{T_j\}_{j=1}^\infty \subset \cZ_1(M;\ZZ_2)$ with
\[ |T_j| \in B^{\bF}_{1/j}(\Lim^{(N-1)}(\cS)) = \bigcup_{i=1}^{k_1} B^{\bF}_{1/j}(V_i) \]
but 
\begin{equation}\label{eq:claim.in.covering.arg.not.in.flat.balls}
T_j \not \in \bigcup_{i=1}^k B_{\eps_i}^\cF(\cT(V_i)) \text{ for all } j = 1, 2, \ldots
\end{equation}
Pass to a subsequence and fix $i \in \{1,\dots,k_1\}$ so that 
\[
\lim_{j\to\infty} \bF(|T_j|,V_i) = 0. 
\]
By Lemma \ref{lemm:cont-cT-operation} we can pass to a further subsequence so that 
\[
\lim_{j\to\infty} \cF(T_j,T) = 0
\]
for some $T \in \cT(V_i)$. This contradicts \eqref{eq:claim.in.covering.arg.not.in.flat.balls}, completing the proof. 
\end{proof}
We now fix $\eta_1 = \dots = \eta_{k_1}$ to be equal to the $\eta$ in the previous claim. The base case will be completed by:

\begin{claim}
The $\{V_1,\dots,V_{k_1}\}$ and $\eta_1,\dots,\eta_{k_{1}},\eps_1,\dots,\eps_{k_{1}}$ chosen above satisfy $(1_1)$-$(4_1)$. 
\end{claim}
\begin{proof}[Proof of claim]
We arranged $(1_1)$ in \eqref{eq:cover.lemm.base.case.not.in.bdry.eps}. Note that $(2_1)$ holds trivially by \eqref{eq:cover.lemm.base.case.limset}. By choice of $\eta$ in the previous claim, $(3_1)$ holds too. Finally, by \eqref{eq:cover.lemm.base.case.flat.balls.disjoint}, $(4_1.a)$ holds for all $i_1 \neq i_2 \in \{1,\dots,k_1\}$, while $(4_1.b)$ holds when $i_1=i_2$. This proves the claim. 
\end{proof}

We proceed with the inductive step. Fix $\ell\in\{1,\dots,N-1\}$ so that $(1_\ell)$-$(4_\ell)$ hold with $\{V_1,\dots,V_{k_{\ell}}\}\subset \cS$, and $\eta_1,\dots,\eta_{k_{\ell}},\eps_1,\dots,\eps_{k_{\ell}} > 0$ with $0< \eps_i < \eps$. Consider
\[
\cS' = \cS \setminus \bigcup_{i=1}^{k_{\ell}} B_{\eta_i}^\bF(V_i). 
\]
Note that $\cS'$ is compact and by $(2_\ell)$, $\Lim^{(N-\ell)}(\cS') = \emptyset$. Thus, $\Lim^{(N-\ell-1)}(\cS')$ and $T_{\ell+1} : = \cT(\Lim^{(N-\ell-1)})(\cS')$ are finite sets. Write
\begin{equation} \label{eq:cover.lemm.induct.limset}
	\Lim^{(N-\ell-1)}(\cS') =: \{V_{k_{\ell}+1},\dots,V_{k_{\ell+1}}\}
\end{equation}
and set
\begin{equation}\label{eq:cover.lemm.eps.ell.defn}
\eps'_{\ell+1} : = \min\{\eps, \min\{\cF(T_1,T_2): T_1\neq T_2 \in \cT_{\ell+1}\}/3\}
\end{equation}
Below, we fix $i\in \{k_{\ell}+1,\dots,k_{\ell+1}\}$ and $T \in \cT(V_i)$. For all previously covered indices $j \in \{1,\dots,k_{\ell}\}$ and $T' \in \cT(V_j)$ note that 
\[
0<\cF(T,T') \neq \eps_j
\]
by $(1_\ell)$. Thus, we can take
\begin{multline}\label{eq:cover.lemm.induct.eps.i}
0 < \eps_i < \min \{\eps_{\ell+1}', \min\{|\cF(T,T')-\eps_j| : T\in\cT(V_i), \\
j\in\{1,\dots,k_{\ell-1}\}, T'\in \cT(V_j)\}\}
\end{multline}
with 
\[
\eps_i \not \in \{\cF(T,T') : T\in \cT(\cS), T' \in \cT(V_i)\}.
\]
This ensures that $(1_{\ell+1})$ holds. For any choice of $\eta_{k_{\ell}+1},\dots \eta_{k_{\ell+1}}>0$, it is clear that $(2_{\ell+1})$ holds. We can then fix $\eta_{k_{\ell}+1},\dots \eta_{k_{\ell+1}}$ sufficiently small so that $(3_{\ell+1})$ holds by the same argument as in claim in the base case. Finally we verify $(4_{\ell+1})$. Fix $i_1,i_2 \in \{1,\dots,k_{\ell+1}\}$. Note that if $i_1,i_2 \leq k_{\ell}$ then $(4_\ell)$ ensures that the relevant condition holds. If $i_1,i_2 > k_{\ell}$ then by \eqref{eq:cover.lemm.eps.ell.defn} it is clear that $(4_{\ell+1}.a)$ holds. Finally, it remains to consider $i_1 \leq k_{\ell} < i_2$ (the condition is symmetric in $i_1,i_2$). Fix $T_1 \in \cT(V_{i_1})$ and $T_2 \in \cT(V_{i_2})$ By \eqref{eq:cover.lemm.induct.eps.i}, 
\[
\eps_{i_1} + \eps_{i_2} < \cF(T_1,T_2),
\]
so $(4_{\ell+1}.a)$ holds. This completes the inductive step, and thus the proof of the lemma. 
\end{proof}

\subsection{Lusternik--Schnirelmann theory} We now modify the arguments from \cite[\S 6]{MarquesNeves:posRic} and \cite[Appendix A]{Aiex:ellipsoids} to prove the following result.

\begin{prop}\label{prop:LS}
	Fix $p \in \NN^*$, and $\mu_1=\mu_1(p)>0$ as in Lemma \ref{lemm:discrete-widths}. Then,
	\[ \omega_p(\SS^2,g_\mu) < \omega_{p+1}(\SS^2,g_\mu) \]
	for every $\mu \in (0, \mu_1)$.
\end{prop}
\begin{proof}
Assume that, for the sake of contradiction, that 
\[
\omega_p(\SS^2,g_\mu)= \omega_{p+1}(\SS^2,g_\mu)
\]
for some $\mu$. We will show this contradicts the Lusternik--Schnirelmann category zero property of a certain space of cycles, as arranged for by Theorem \ref{theo:good.metric.final} and Lemma \ref{lemm:cover.cF.bF}.

By Corollary \ref{coro:exist-cub-complex-p-width} there exists $X\subset I^{2p+3}$ with corresponding homotopy class $\Pi$ so that 
\begin{equation}\label{eq:fix.X.LS.arg}
\bL_\textrm{AP}(\Pi,g_\mu) = \omega_{p+1}(\SS^2,g_\mu). 
\end{equation}
Define $\Lambda' := \bL_\textrm{AP}(\Pi,g_\mu) + 5^{2p+3}$, $\Lambda := \Lambda'+1$, and
\[
\cS : = \{ V \in \bar \cS^{\Lambda'}(g_\mu) : \Vert V \Vert = \bL_\textrm{AP}(\Pi,g_\mu) \} \subset \cI \cV_1(M)
\]
Note that $\Lambda$ is bounded from above depending on $p$, by Lemma \ref{lemm:discrete-widths}.  By (3) in Theorem \ref{theo:good.metric.final}, alternative (3) of Theorem \ref{theo:geodesic.network.trichotomy} is the only one that can hold and implies that there is $N = N(p)$ so that 
\[
\Lim^{(N)}(\cS) = \emptyset. 
\]
Fix $\eps>0$ so that every continuous map $\Phi : \SS^1\to \cZ_1(M;\ZZ_2)$ with $\Phi(\SS^1) \subset B_{\eps}^\cF(T)$ for some $T \in \cZ_1(M;\ZZ_2)$ is homotopically trivial (this is possible thanks to \cite[Proposition 3.3]{MarquesNeves:posRic}). Since $\cS$ is compact in the $\bF$-topology (by Allard's integral compactness theorem \cite[Remark 42.8]{Simon83}), we can apply Lemma \ref{lemm:cover.cF.bF} and Corollary \ref{coro:cover.cF.bF} to $\cS$ and $\eps>0$ just chosen, yielding $\{V_1,\dots,V_k\}\subset \cS$, $\eta_1,\dots,\eta_k,\eps_1,\dots,\eps_k>0$ with $0<\eps_j < \eps$, 
\begin{equation}\label{eq:LS.arg.cS.cover}
\cS \subset \bigcup_{j=1}^k B_{\eta_j}^\bF(V_j),
\end{equation}
and
\begin{equation}\label{eq:LS.arg.cS.cover.2}
	\begin{gathered}
		\text{for any continuous } \Phi : \SS^1 \to \cZ_1(M;\ZZ_2) \text{ satisfying } \\
		|\Phi(t)| \subset \cup_{j=1}^k B_{2\eta_j}^\bF(V_j) \text{ for all } t \in \SS^1, \\
		\text{we have } \Phi(\SS^1) \subset B_\eps^{\cF}(T) \text{ for some } T \in \cT(\{V_1, \ldots, V_k\}).
	\end{gathered}
\end{equation}
We obtain a contradiction as in the proof of \cite[Theorem 6.1]{MarquesNeves:posRic}.  By Proposition \ref{prop:pull.tight} we can choose a minimizing sequence $\{\Phi_i\}_{i=1}^\infty\subset \Pi$ so that every element of $\bC(\{\Phi_i\})$ is stationary. By \cite[Corollary 3.9]{MarquesNeves:posRic} we can assume that $\Phi_i : X \to \cZ_1(M;\bM;\ZZ_2)$ is continuous (i.e., with respect to the mass topology).

For $\{\ell_i\}_{i=1}^\infty\subset \NN$ to be chosen, define $Y_i \subset X(\ell_i)$ to be the subcomplex of $X(\ell_i)$ consisting of cells $\alpha \in X(\ell_i)$ so that 
\[
|\Phi_i(x)| \not \in \bigcup_{j=1}^k B_{\eta_j}^\bF(V_j)
\]
for every vertex $x\in\alpha_0$. We fix $\ell_i$ sufficiently large so that (i) if $x \in \overline{X\setminus Y_i}$, then
\[
|\Phi_i(x)| \in \bigcup_{j=1}^k B_{2\eta_j}^\bF(V_j)
\]
and (ii) the fineness\footnote{Recall that the fineness of $\phi : W_0\to \cZ_1(M;\ZZ_2)$ is the maximum of $\bM(\phi(x)-\phi(y))$ over all adjacent vertices $x,y\in W_0$; see \cite[p.\ 141]{Pitts} and \cite[p.\ 583]{MarquesNeves:posRic}.} of $\Phi_i$ restricted to $X(\ell_i)_0$ is less than $1/i$.

By the argument in \cite[Claim 6.3]{MarquesNeves:posRic}, $\Psi_i: = (\Phi_i)|_{Y_i}$ are $p$-sweepouts for $i$ sufficiently large, because whenever $\gamma : \SS^1 \to Z_i := \overline{X\setminus Y_i}$ is continuous, we have
\[
|\Phi_i \circ \gamma(t)| \in \bigcup_{j=1}^k B_{2\eta_j}^\bF(V_j)
\]
for all $t \in \SS^1$, so $\Phi_i\circ\gamma$ is homotopically trivial by our choice of $\eps$ and by \eqref{eq:LS.arg.cS.cover.2}.

At this point, a contradiction follows as in \cite[pp.\ 604--5]{MarquesNeves:posRic}. Indeed:
\[
\omega_p(\SS^2,g_\mu) \leq \limsup_{i\to\infty}\sup_{x\in Y_i} \bM(\Psi_i(x)) \leq \omega_{p+1}(\SS^2,g_\mu) = \omega_p(\SS^2,g_\mu).
\]
The first inequality holds because the $\Psi_i$ are $p$-sweepouts. The second inequality holds because the $\Psi_i$ are restrictions of the $\Phi_i$ that form a minimizing sequence for $\Pi$ and \eqref{eq:fix.X.LS.arg} holds. The equality on the right is the hypothesis we made at the start of the proof. Thus:
\[ \limsup_{i \to \infty} \sup_{x \in Y_i} \bM(\Psi_i(x)) = \omega_p(\SS^2,g_\mu). \]
Because $\Psi_i$ are $p$-sweepouts, it follows that alternative (2) of Proposition \ref{prop:am.vary.dom} cannot occur. Thus, by alternative (1), there must exist $V \in \bC(\{\Psi_i\})\cap \cS^{\Lambda}(g_\mu) \subset \cS$. By (ii) in the choice of $\ell_i$ and \cite[p.\ 66]{Pitts}, we can pass to a subsequence and find $x_i \in (Y_i)_0 \subset X(\ell_i)_0$ with 
\[
\lim_{i\to\infty}\bF(|\Phi_i(x_i)|,V) = 0.
\]
On the other hand, since $x_i \in (Y_i)_0$, we have $|\Phi_i(x_i)| \not \in \bigcup_{j=1}^k B_{\eta_j}^\bF(V_j)$ for all $i$, so
\[
V \not \in \bigcup_{j=1}^k B_{\eta_j}^\bF(V_j).
\]
This contradicts \eqref{eq:LS.arg.cS.cover}, completing the proof. 
\end{proof}

\section{The $p$-widths of a round 2-sphere} \label{sec:sphere.p.widths}

\begin{proof}[Proof of Theorem \ref{theo:sphere.pwidths}]
Rather than compute $\omega_p(\SS^2, g_0)$ one $p$ at a time, we will show that
\begin{equation} \label{eq:main.restated}
	\omega_{n^2}(\SS^2, g_0) = \ldots = \omega_{(n+1)^2-1}(\SS^2, g_0) = 2 \pi n, \text{ for every } n \in \NN^*.
\end{equation}
So, fix $n \in \NN^*$ and also $\mu_1=\mu_1((n+1)^2-1) > 0$ as in Lemma \ref{lemm:discrete-widths}. Assume that 
\[
0 < \mu < \min\{\mu_1,1/2n\}.
\]
\begin{claim} \label{clai:main}
	For $m \in \{1,\dots,n\}$, we have
	\begin{multline*}
		\{\omega_p(\SS^2,g_\mu) : p = 1,\dots, (m+1)^2-1\} =\\
		 \{2\pi(n_1+n_2+n_3) + \mu(n_2+2n_3) : (n_1,n_2,n_3) \in \NN^3\} \cap (0,2\pi m+1].
	\end{multline*} 
\end{claim}
\begin{proof}[Proof of claim]
	Lemma \ref{lemm:discrete-widths} implies that one inclusion above holds:
	\begin{multline*}
		\{\omega_p(\SS^2,g_\mu) : p = 1,\dots, (m+1)^2-1\} \subset \\
		 \{2\pi(n_1+n_2+n_3) + \mu(n_2+2n_3) : (n_1,n_2,n_3) \in \NN^3\} \cap (0,2\pi m+1].
	\end{multline*} 
	Note that by Proposition \ref{prop:LS},
	\[
		\#\{\omega_p(\SS^2,g_\mu) : p = 1,\dots, (m+1)^2-1\}=(m+1)^2-1.
	\]
	Furthermore, for $j \in \{1,\dots,n\}$,
	\[
		\{\mu(n_2+2n_3) : (n_1,n_2,n_3)\in \NN^3, n_1+n_2+n_3=j\} = \{0,\mu,\dots,2j\mu\}
	\] 
	and since $\mu<1/2n$, we find
	\begin{align*}
		& \# \Big( \{2\pi(n_1+n_2+n_3) + \mu(n_2+2n_3) : (n_1,n_2,n_3) \in \NN^3\} \cap (0,2\pi m+1] \Big) \\
		& = \sum_{j=1}^m (2j+1) = (m+1)^2-1. 
	\end{align*}
	This completes our proof of the claim.
\end{proof}

Now a simple induction argument on $m \in \{1,\dots,n\}$ shows that
\[
2\pi m \leq \omega_p(\SS^2,g_\mu) \leq (2\pi + 2\mu) m \text{ for all } p \in \{m^2,\dots,(m+1)^2-1\}. 
\]
Since the $p$-widths are continuous in the metric by Lemma \ref{lem:p-width-cont-metric}, we can send $\mu\to 0$ above to obtain
\[ \omega_p(\SS^2, g_0) = 2 \pi m \text{ for all } p \in \{ m^2, \ldots, (m+1)^2-1 \}, \]
and all $m \in \NN^*$, as required. The fact that these values are attained by a sweepout of homogeneous polynomials follows from Proposition \ref{prop:upper.bds}.
\end{proof}

\section{Open questions} \label{sec:open}

\begin{enumerate}
	\item \textbf{Morse index}. 
		In ambient dimensions $n+1 \geq 3$, the Morse index of $\Sigma_p$ in Theorem \ref{theo:ap.existence} has been shown by Marques--Neves \cite{MarquesNeves:multiplicity} and Li \cite{Li:morse} to satisfy:
		\begin{equation} \label{eq:ap.index}
			\operatorname{index}_g(\Sigma_p) \leq p,
		\end{equation} 
		provided we consider variations supported away from $\bar \Sigma_p \setminus \Sigma_p$ (see also Gaspar \cite{Gaspar}, Hiesmayr \cite{Hiesmayr} for the phase transition approach). Moreover, if $3 \leq n+1 \leq 7$, the work of Zhou \cite{Zhou:multiplicity-one} and Marques--Neves \cite{MarquesNeves:uper-semi-index} (see also \cite{CM:3d} for the phase transition approach) shows
		\begin{equation} \label{eq:ap.index.bumpy}
			\operatorname{index}_g(\Sigma_p) = p,
		\end{equation}
		for generic metrics $g$ on $M$. 

		It would be interesting to relate the Morse index of the geodesics $\sigma_j$ in Theorem \ref{theo:immersed.geo.min.max} to $p$, similarly to \eqref{eq:ap.index} or \eqref{eq:ap.index.bumpy}. In two-dimensions, such a relationship is somewhat more complicated, since  points of non-embeddedness will contribute to the index count. We conjecture (based on the ideas contained in \cite{dPKW:neumann.mult,Hiesmayr,Gaspar,Mantoulidis,CM:3d,LiuWei}) that if $(M^2,g)$ is bumpy (i.e., no immersed geodesic admits a nontrivial normal Jacobi field) then an expression of the following kind should hold:
		\[
			\sum_{j=1}^{N(p)} \Index(\sigma_{p,j}) + \sum_{i \leq j}c_{ij} = p
		\]
		where $c_{ij}$ is a function of the (self) intersections between $\sigma_{p,i}$ and $\sigma_{p,j}$. For example, if $\sigma_1$ is a figure-eight type curve, we expect that $c_{11} = 1$. Note that the resolution of such a conjecture would likely yield an alternative proof of Theorem \ref{theo:sphere.pwidths} that avoids any reference to Almgren--Pitts theory.
	
	\item \textbf{Computing the min-max configurations}. What configuration of great circles can occur in $\mathbf{C}_{\textrm{PT}}(\tilde\Pi)$? Is $\bC_{\textrm{AP}}(\Pi) \setminus \bC_\textrm{PT}(\tilde\Pi)$ non-empty? If so, are there elements that are unions of great circles that are not attained by phase transition min-max? Are there non-trivial geodesic nets in $\bC_{\textrm{AP}}(\Pi)$?

		One expects (by Lusternik--Schnirelmann theory) that there should be non-trivial families of configurations, since the $p$-widths of the round two-sphere are not strictly increasing. It seems plausible that one possible configuration that can occur is the union of $\lfloor\sqrt{p}\rfloor$ great circles intersecting in antipodal points at equally distributed angles (like an orange). Note that phase transition solutions with the this symmetry can be constructed by a reflection argument (this was observed by Guaraco, cf.\ \cite[\S 6]{Guaraco:ACnotes}). To this end, it would be interesting to compute the Morse index  and nullity of these solutions. 
	
		One can draw a parallel between this and recent work of Kapouleas--Wiygul \cite{KaouleasWiygul} concerning the Morse index (nullity) of the Lawson surfaces\footnote{cf.\ \cite{Lawson,Kapouleas:survey,Brendle:survey}} $\xi_{g,1}$ (and it seems likely the techniques used in \cite{KaouleasWiygul} could be applied to compute the index and nullity of the reflection solutions). Other potential candidates for (1) above would be an equator with multiplicity $\lfloor \sqrt{p}\rfloor$; it seems possible that this could arise from phase transition solutions corresponding to a non-trivial solution to the Jacobi--Toda system on the equator (cf.\ \cite{dPKW:neumann.mult,delPinoKowalczykWeiYang:interface}). It would be interesting to generalize the results \cite{KaouleasWiygul} to the full Lawson family $\xi_{m,k}$, since one can view them as a desingularization of $k+1$ great spheres intersecting in a common equatorial circle (with equal angles); cf.\ \cite{Kapouleas:survey}. As such, Theorem \ref{theo:immersed.geo.min.max} could suggest that (some of) the $p$-widths of the round three-sphere might be attained by $\xi_{m,k}$ for $k \sim p^{1/3}$ and $m$ large (but this is quite speculative). Determining the index of $\xi_{m,k}$ would be an interesting first step in understanding if this is a reasonable suggestion. 
		
	\item \textbf{Other double-well potentials}. We do not know whether Theorem \ref{theo:immersed.geo.min.max} can be proven using the standard double-well form of the phase transition regularization, or any other double-well potential that isn't trivially related to the sine-Gordon one. 

		 Note that is not the first work in the theory of phase-transitions that may or may not crucially rely on a particular double-well potential. For instance, the precise form of the potential also plays an important role for Taubes \cite{Taubes:first.second.order}; see also the more recent work of Pigati--Stern \cite{PigatiStern} on codimension-two phase transition min-max (\cite[Remark 1.2]{PigatiStern}). (These potentials are different than ours.)
	
	\item \textbf{Other regularizations}. Can the prescribed mean curvature regularization used in \cite{Zhou:multiplicity-one} (cf.\ \cite{ZZ:cmc,ZZ:pmc,ChengZhou}) be used to prove Theorem \ref{theo:immersed.geo.min.max}? 

	\item \textbf{Other surfaces} It would be interesting to study the $p$-widths on other surfaces (e.g., flat tori and hyperbolic surfaces), by combining Theorem \ref{theo:immersed.geo.min.max} with the knowledge that $a(1)=\sqrt{\pi}$. See also \cite{Liokumovich:short.cycles}.  The case of surfaces with boundary (e.g., flat disks) would also be interesting to consider, e.g., as discussed in Remark \ref{rema:guth.bound}.
\end{enumerate}

\appendix

\section{Metric space notions}\label{app:metric}

For $S\subset (X,d)$, we denote:
\begin{itemize}
	\item $\Lim(S)$ to be the set of limit points of $S$. Recall that $\Lim(S)$ is closed by a standard diagonal argument.
	\item For $x \in X$, $D(x,S) : = \{d(y,x) : y\in S\}$. 
\end{itemize}

We record the following elementary lemmas:

\begin{lemm}\label{lemm:limit.dist.dist.limit}
For $S \subset (X,d)$ compact, $\Lim(D(x,S)) \subset D(x,\Lim(S))$. 
\end{lemm}
\begin{proof}
Suppose that $t \in \Lim(D(x,S))$. Fix $\{t_i\}_{i=1}^\infty\subset D(x,S)\setminus\{t\}$ with $t_i\to t$ and correspondingly $y_i \in S$ with $d(y_i,x) = t_i$. By compactness of $S$, we can assume that $y_i\to y \in S$. Since $t_i\neq t$, we see that $y_i\neq y$ for all $i$. Hence, $y\in\Lim(S)$ so $t=d(y,x) \in D(x,\Lim(S))$. 
\end{proof}

\begin{coro}\label{coro:countable}
If $S\subset (X,d)$ is compact with $\Lim^{(N)}(S) = \emptyset$ for some $N \in\NN$, then $D(x,S)$ is at most countable for all $x \in X$. 
\end{coro}

\begin{proof}
By Lemma \ref{lemm:limit.dist.dist.limit}, $\Lim^{(N)}(D(x,S)) = 0$. An uncountable subset of $\RR$ has uncountably many limit points, so $D(x,S)$ must be at most countable. 
\end{proof}

\section{Geometric measure theory} \label{sec:GMT}

Consider $(M,g)$ a closed oriented Riemannian $2$-manifold isometrically embedded in some $\RR^J$. The relevant spaces considered here are (see also \cite[\S 2.1]{MarquesNeves:posRic}): 
\begin{itemize}
\item the space $\bI_k(M;\ZZ_2)$ of $k$-dimensional mod $2$ flat chains in $\RR^J$ with support in $M$,
\item the space $\cZ_1(M;\ZZ_2)\subset \bI_1(M;\ZZ_2)$ of cycles,
\item the space $\cV_1(M)$ of $1$-varifolds on $M$, and
\item the space $\cI\cV_1(M)$ of integral rectifiable $1$-varifolds on $M$.
\end{itemize}
For $T \in \bI_1(M;\ZZ_2)$, denote by $|T|,\Vert T\Vert$ the associated integral varifold and Radon measure on $M$; similarly, for $V \in \cV_1(M)$, write $\Vert V\Vert$ for the associated Radon measure on $M$. We will use the \emph{flat metric} $\cF(\cdot)$ on $\bI_k(M;\ZZ_2)$ and write $\bM(\cdot)$ for the \emph{mass functional}. We will also use the $\bF$-metrics on $\cV_1(M)$ and $\bI_1(M;\ZZ_2)$ (cf.\ \cite[p.\ 66]{Pitts}). We will give $\cZ_1(M;\ZZ_2)$ the flat metric and write $\cZ_1(M;\bF;\ZZ_2)$, $\cZ_1(M;\bM;\ZZ_2)$ when we use the $\bF$ or $\bM$ metrics. 

If $K\subset M$ is a countably $1$-rectifiable, $\cH^1$-measurable set and $\theta$ is a $\cH^1$-measurable function on $M$ with $\int_M \theta d\cH^1 < \infty$, we will write $\bv(K,\theta)$ for the associated integral $1$-varifold, cf.\  \cite[\S 4]{Simon83}. If $V \in \cI\cV_1(M)$ we will write $\reg V$ for the regular set of $V$, defined to be $p\in M$ at which there is a neighborhood $U$ so that $V\restr G_1(U) = \bv(\sigma,\theta_0)$ for $\sigma$ a properly embedded $C^1$-curve in $U$ and $\theta_0 \in \NN^*$. The complementary set $\sing V := \supp V \setminus \reg V$ is the singular set.

\begin{prop} \label{prop:geodesic.network.stationary.varifold.limit}
	Let $(M, g)$ be a closed 2-dimensional manifold. Suppose $S_1, S_2, \ldots \in \cI \cV_1(M)$ are $g$-stationary and $S_i \rightharpoonup S_\infty$ as $i \to \infty$. Then:
	\begin{equation} \label{eq:geodesic.network.stationary.varifold.limit.supp}
		\supp S_\infty = \lim_{i \to \infty} \supp S_i,
	\end{equation}
	\begin{equation} \label{eq:geodesic.network.stationary.varifold.limit.sing}
		\sing S_\infty \subset \lim_{i \to \infty} \sing S_i,
	\end{equation}
	in the Hausdorff sense.
\end{prop}
\begin{proof}
	Equation \eqref{eq:geodesic.network.stationary.varifold.limit.supp} is a well-known consequence of the monotonicity formula for stationary integral varifolds and holds true in all codimensions; see \cite[Section 2]{AllardAlmgren:1varifold}.
	
	Equation \eqref{eq:geodesic.network.stationary.varifold.limit.sing} is only true in this dimension. Suppose $p \in \sing S_\infty$ were not in $\lim_{i \to \infty} \sing S_i$. Then, there will exist some $\eps > 0$ such that $\sing S_i \cap B_\eps(p) = \emptyset$ for all $i$. Then, by \cite[Section 3]{AllardAlmgren:1varifold}, $\supp S_i \cap B_\eps(p)$ must consist of non-intersecting smooth geodesic segments. Such segments have curvature estimates, so their limit must be smooth too, violating $p \in \sing S_\infty$.
\end{proof}

\section{Phase transition regularity results} \label{sec:ac.regularity}

We will rely on the various general results concerning $\eps \to 0$ limits of solutions to \eqref{eq:ac.pde} on 2-dimensional Riemannian manifolds, which we now recall. 

The first result is Hutchinson--Tonegawa's compactness theorem, which concerns taking $\eps \to 0$ limits of arbitrary critical points with suitable $L^\infty$ and energy bounds and obtaining a limiting stationary integral 1-varifold.

\begin{prop}[{\cite[Theorem 1]{HutchinsonTonegawa00}, cf.\ \cite[Appendix B]{Guaraco}}]\label{prop:HT.theory}
	Suppose $(M, g_\infty)$ is a complete 2-dimensional Riemannian manifold, $\{g_i\}_{i=1}^\infty \subset \met(M)$ are complete metrics with $\lim_i g_i = g_\infty$ in $C^\infty_\textnormal{loc}(M)$, $U \subset M$ is open, $\{(u_i,\eps_i)\}_{i=1}^\infty\subset C^\infty_\textnormal{loc}(U)\times (0,\infty)$, $\lim_i \eps_i=0$, and each $u_i$ satisfies \eqref{eq:ac.pde} on $(U, g_i)$ with
\[
	\Vert u_i \Vert_{L^\infty(U)} \leq 1 \text{ and } (E_{\eps_i}\restr(U,g_i))[u_i] \leq E_0, 
\]
	for $i=1,2,\dots$ After passing to a subsequence, we have
	\begin{itemize}
		\item $\lim_i u_i = u_\infty$ in $L^1_\textnormal{loc}(U)$, $u_\infty \in BV_\textnormal{loc}(U)$, $u_\infty = \pm 1$ a.e.\ on $U$,
		\item $\lim_i V_{\eps_i}[u_i] \restr G_1(U)  = V^\infty$ for a stationary integral $1$-varifold $V^\infty \in \cI\cV_1(U)$,
		\item $\lim_i(h_0^{-1}E_{\eps_i}\restr (U',g_i))[u_i] = \Vert V^\infty\Vert(U')$ for all $U' \Subset U$,
		\item $\lim_i \{u_i=t\} \cap U' = \supp \Vert V^\infty\Vert \cap U'$ in the Hausdorff topolgy, for all $U' \Subset U$ and all $t \in (-1,1)$.
		\item the density (``multiplicity'') of $V^\infty$ is a.e. odd on $\partial^*\{u_\infty = +1\}\cap U$ and a.e. even on $\supp \Vert V^\infty\Vert \cap U \setminus \partial^* \{u_\infty = +1\}$. 
	\end{itemize}
\end{prop}

We can get improved convergence in Proposition \ref{prop:HT.theory} if we additionally assume that each $u_i$ is a linearly stable critical point, i.e., $\Index_{\eps_i}(u_i; U) = 0$. This is because one then has estimates on the following curvature-type quantity:
 
 \begin{defi}
 For $(M,g)$ a Riemannian $2$-manifold and $u\in C^\infty_\textrm{loc}$, if $x \in M\setminus\{\nabla u=0\}$ then the \emph{enhanced second fundamental form} of $u$ at $x$ is
 \[
 \cA = |\nabla u|^{-1}(\nabla^2u - \nabla^2 u(\cdot,\nu) \otimes \nu^\flat)
 \]
 where $\nu=|\nabla u|^{-1}\nabla u$. 
 \end{defi}
\begin{rema}\label{rema:gen.sff}
 It is straightforward to check that at points where $\nabla u \neq 0$,
 \[
 |\cA|^2 = k^2 + |\nabla ^T\log|\nabla u||^2
 \]
 where $k$ denotes the curvature of the level curve through $x$ and $\nabla^T$ denotes the tangential gradient along the level curve.
 \end{rema}
 
Curvature estimates were first obtained by Tonegawa \cite{Tonegawa05} in the form of $L^2$ estimates on $\cA$  (cf. \cite[Lemma 4.6]{Mantoulidis}) in the setting of Proposition \ref{prop:HT.theory} with the additional assumption $\Index_{\eps_i}(u_i; U) = 0$. While these suffice for certain applications, such as controlling the number of singular points when $\Index_{\eps_i}(u_i; U)$ is $\leq I$ rather than $0$, they do not suffice when studying the finer structure of the singularity. 

We will, instead, rely on the following fundamental curvature estimates due to Wang--Wei \cite{WangWei}  (cf.\ \cite{CM:3d,WangWei2} for higher dimensional extensions).

\begin{prop}[{\cite[Theorem 3.5]{WangWei}, cf.\ \cite[Theorem 4.13 and Corollary 4.14]{Mantoulidis}}]\label{prop:WW.est}
	Suppose $(M, g)$ is a complete 2-dimensional Riemannian manifold, $U \subset M$ is open, and $u$ satisfies \eqref{eq:ac.pde} on $U$ with 
	\[ \Vert u \Vert_{L^\infty(U)} \leq 1 \text{ and } \Index_{\eps}(u; U) = 0. \]
	For all $U' \Subset U$ and $\beta \in (0, 1)$, there are $C, \theta, \eps_0 > 0$ depending on $\beta$, $\dist_g(U',\partial U)$, $\inj(U, g)$, and the $C^\infty$ norm of $g$ with respect to a fixed background metric so that, if $\eps \in (0,\eps_0)$, then
 \[
 \eps |\nabla u| \geq C^{-1} \textrm{ and } |\cA(x)| \leq C\eps^\theta \text{ on } U' \cap \{ |u| \leq 1-\beta \}.
 \]
\end{prop}
\begin{rema}
	In the two-dimensional setting above, Proposition \ref{prop:WW.est} does not require an energy estimate for $u$, but the proof is simpler if one does assume it holds. In the context of the current paper, we will always have such an estimate available (cf.\ Lemma \ref{lemm:monotonicity.AC.energy} below).
\end{rema}
 
\begin{lemm}[{\cite[Proposition 3.4]{HutchinsonTonegawa00}, cf.\ \cite[Appendix B]{Guaraco}, \cite[Lemma 4.3]{Mantoulidis}}]\label{lemm:monotonicity.AC.energy}
	Suppose $(M, g)$ is a complete 2-dimensional Riemannian manifold, $U \subset M$ is open, and $u$ satisfies \eqref{eq:ac.pde} on $U$ with
	\[ (E_{\eps} \restr U)[u] \leq E_0. \]
	For all $U' \Subset U$, there are $C, \eps_0, r_0 > 0$ depending on $E_0$, $\dist_g(U', \partial U)$, $\inj(U, g)$, and the $C^\infty$ norm of $g$ with respect to a fixed background metric so that, if $\eps \in (0, \eps_0)$, then 
	\[
		(E_\eps\restr B_r(p))[u] \leq C r \text{ for all } r \in (0, r_0), \; p \in U'.
	\]
\end{lemm}

\section{Upper bounds for the $p$-widths of the two-sphere} \label{sec:upp.bds}
Guth considered upper bounds for the $p$-widths of a disk in \cite[\S 6]{Guth:minimax} coming from zero sets of polynomials. We now recall Aiex's construction of similar sweepouts on $\SS^2$ \cite[\S 5-6]{Aiex:ellipsoids}.

Let $\RR[x,y]_k$ (resp. $\RR[x,y,z]_k$) denote the space of real polynomials in two variables  (resp. three variables), of degree $\leq k$. Define 
\[
A_k := \{f + z g : f \in  \RR[x,y]_k , g \in  \RR[x,y]_{k-1}\} \subset \RR[x,y,z]_k.
\]
Clearly $A_k$ is a linear subspace of $\RR[x,y,z]_k$ of dimension 
\[
 \dim \RR[x,y]_k + \dim \RR[x,y]_{k-1} = {k+2 \choose 2} + {k+1 \choose 2}  = (k+1)^2.
\]
Note that any $(f, g) \mapsto f+zg \in A_k$ is a linear isomorphism. Indeed, suppose $f + zg \equiv 0$. For any $(x,y,z)\in \SS^2 \setminus\{z=0\}$ we have
\[
f(x,y) + z g(x,y) = f(x,y) - z g(x,y) = 0 \Rightarrow f(x,y) = g(x,y) = 0. 
\]
This implies that $f,g$ vanish on the open unit ball in $\RR^2$, so $f \equiv g \equiv 0$, and the claim follows.  In what follows, we identify $\RR P^D$ with $A_k \setminus \{0\}$ mod $\RR^*$.

\begin{prop} \label{prop:upper.bds}
	Let $k \in \NN^*$ and $A_k$ be as above. Set $D = (k+1)^2-1$. The projectivization 
	\[ F_D : \RR P^D \to \cZ_1(\SS^2; \ZZ_2), \; F_D(f) := \{ f = 0 \}, \]
	gives a continuous map with respect to the flat norm, with no concentration of mass, detecting the generator of $H^*(\cZ_1(\SS^2; \ZZ_2))$, and satisfying
	\[ \sup_{x \in \RR P^D} \bM(F_D(x)) \leq 2\pi k. \]
\end{prop}
\begin{proof}
	Flat-norm continuity follows from the arguments in \cite[Lemma 6.2]{Guth:minimax}. That $F_D$ detects the generator of $H^*(\cZ_1(\SS^2; \ZZ_2)) \cong \ZZ_2[\bar \lambda]$ follows by considering the linear sweepout $ax+b \in A_k$. The no-concentration of mass property follows from the Crofton formula as in \cite[Lemma 5.1]{Aiex:ellipsoids} as does $\sup_{x\in \RR P^D} \bM(F_D(x)) \leq 2\pi k$; 
see \cite[Theorem 5.2]{Aiex:ellipsoids}. 
\end{proof}

Proposition \ref{prop:upper.bds} and the easy fact that $p$-sweepouts are $p'$-sweepouts when $p' \leq p$ imply:
\begin{coro} \label{coro:upper.bds}
	$\omega_p(\SS^2,g_{\SS^2}) \leq 2\pi \lfloor \sqrt{p}\rfloor$ for all $p \in \NN^*$. 
\end{coro}

\bibliographystyle{alpha}
\bibliography{bib}
\end{document}